\documentclass[10pt]{amsart} 

\usepackage{graphicx}
\usepackage[pdftex,dvipsnames]{xcolor} 
\usepackage{amsmath,amssymb,amsfonts,euscript,enumerate}
\usepackage{amsthm}  
\usepackage{color,wrapfig}
\usepackage{verbatim}
\usepackage{hyperref}
\usepackage{latexsym}
\usepackage{kotex}
\usepackage[nameinlink]{cleveref}
\usepackage{mathrsfs} 
\usepackage{a4wide} 
\usepackage{bbm}
\usepackage{mathtools}
\usepackage{cases}

\newtheorem{theorem}{Theorem}[section]
\newtheorem{lemma}[theorem]{Lemma}

\newtheorem{corollary}[theorem]{Corollary}

\theoremstyle{plain}

\theoremstyle{definition}
\newtheorem{definition}[theorem]{Definition}

\theoremstyle{remark}
\newtheorem{remark}[theorem]{Remark}

\numberwithin{equation}{section}

\makeatletter
\@namedef{subjclassname@2020}{\textup{2020} Mathematics Subject Classification}
\makeatother

\newcommand{\bR}{{\mathbb R}}

\newcommand{\cA}{{\mathcal A}}

\title[Generalized Schauder Theory]{Generalized Schauder Theory and its Application to Degenerate/Singular Parabolic Equations} 

\author{Takwon Kim}
\address{Research Institute of Mathematics, Seoul National University, Seoul 08826, Republic of Korea}
\email{xkrkr@snu.ac.kr}

\author{Ki-Ahm Lee}
\address{Research Institute of Mathematics, Seoul National University, Seoul 08826, Republic of Korea\\
\indent Department of Mathematical Sciences, Seoul National University, Seoul 08826, Republic of Korea}
\email{kiahm@snu.ac.kr}

\author{Hyungsung Yun}
\address{Department of Mathematical Sciences, Seoul National University, Seoul 08826, Republic of Korea}
\email{euler@snu.ac.kr}


\begin{document} 

\begin{abstract}
In this paper, we study generalized Schauder theory for the degenerate/singular parabolic equations of the form
$$u_t = a^{i'j'}u_{i'j'} + 2 x_n^{\gamma/2} a^{i'n} u_{i'n} + x_n^{\gamma} a^{nn} u_{nn} + b^{i'} u_{i'} + x_n^{\gamma/2} b^n u_{n} + c u + f \quad (\gamma \leq1).$$
When the equation above is singular, it can be derived from Monge--Ampère equations by using the partial Legendre transform. Also, we study the fractional version of Taylor expansion for the solution $u$, which is called $s$-polynomial. To prove $C_s^{2+\alpha}$-regularity and higher regularity of the solution $u$, we establish generalized Schauder theory which approximates coefficients of the operator with $s$-polynomials rather than constants. The generalized Schauder theory not only recovers the proof for uniformly parabolic equations but is also applicable to other operators that are difficult to apply the bootstrap method to obtain higher regularity.\\

\end{abstract}

\maketitle

%
%
\section{Introduction}
In this paper, we study generalized Schauder theory and fractional order expansion of solutions $u$ for the following degenerate/singular parabolic equations
 \begin{equation} \label{eq:main}
u_t = L u + f  \quad \textnormal{in } Q_1^+ ,
 \end{equation}
where
 \begin{equation*}
	L = a^{i'j'}(X) D_{i'j'} + 2 x_n^{\gamma/2} a^{i'n}(X) D_{i'n} + x_n^{\gamma} a^{nn}(X) D_{nn} + b^{i'}(X) D_{i'} + x_n^{\gamma/2} b^n(X) D_{n} + c(X) 
\end{equation*}
with a constant $\gamma \leq 1$ and 
 $$Q_1^+= \{ (x,t) \in \mathbb{R}_+^n  : |x_i| < 1\,(1 \leq i \leq n)  \} \times (-1,0] .$$
The repeated index with prime $i'$ means the summation from $1$ to $(n-1)$, that is
\begin{equation*}
	A^{i'} B_{i'}  =\sum_{i=1}^{n-1} A^i B_i \quad \text{and} \quad A^{i'j'} B_{i'j'} =\sum_{i,j=1}^{n-1} A^{ij} B_{ij}.
\end{equation*}
In classical regularity theory, to obtain $C^{k,\alpha}$-regularity of the solution $u$, it is shown that the following statement holds (cf. \cite{Wan92}): For each $Y \in \overline{\Omega}$, there exists a polynomial $p^Y$ of degree $k$ satisfying
\begin{equation} \label{def_cka}
	\|u - p^Y\|_{L^{\infty}(B_r(Y) \cap \Omega)} \leq C r^{k+\alpha} \quad \text{for all } r>0.
\end{equation}
Furthermore, we can obtain higher order partial derivative at $Y$ of $u$ can be obtained from the coefficients of the polynomial $p^Y$. However, the function that is not smooth enough cannot be approximated by polynomials. Breaking away from the stereotype that polynomial approximation should be used, we consider generalized polynomials represented by monomials of fractional order, and we call them $s$-polynomials in this study. In simple terms, $s$-polynomial approximation can be understood as a fractional version of Taylor expansion. The coefficients of an $s$-polynomial $p$ approximating the solution can be regarded as suitable constant multiples of the weighted derivative, and the degree of $p$ gives the order of regularity. While conventional regularity theory so far has focused on differentiability for classical solutions, in this study, we are interested in how well $s$-polynomial approximates the solution instead of differentiability. In other words, a generalized concept of regularity theory can be developed through the $s$-polynomial.

The regularity of solutions for \eqref{eq:main} can be expected to be $C^{1,\alpha}$-regularity up to the boundary, but higher regularity up to the boundary can be obtained by considering $s$-polynomials and a new metric that preserves scaling. In classical Schauder theory, by showing that the derivative satisfies the equations of the same class, $C^{2,\alpha}$-regularity of the gradient $Du$ is obtained, and iteratively, higher regularity can be obtained. However, this bootstrap method is not applicable in \eqref{eq:main} since the equation that the partial derivative $u_n$ satisfies is not of the same class as \eqref{eq:main}, so the generalized Schauder theory was developed to solve this issue. It is also applicable to other operators that are difficult to apply the bootstrap method to obtain higher regularity.

The study on the regularity of solutions for \eqref{eq:main} was inspired by several previous studies. 
\begin{enumerate}
\item Monge--Ampère equations: Daskalopoulos and Savin \cite{DS09} converted the Monge--Ampère equations
$$\det D^2 u = (x^2+y^2)^{\gamma/2} \quad \text{in } B_1$$
into the singular equations with $\gamma<0 $
$$ |x|^{\gamma} v_{xx} + v_{yy} = 0 \quad \text{in } B_1$$
using the partial Legendre transform. In \cite{DS09}, they solved the problem for Monge--Ampère equation by finding the fractional order expansion 
\begin{align*}
	v(x,y) &= a_1 + a_2 x + a_3 y + a_4 xy \\
	&\qquad+ a_5 \left( \frac{1}{2} y^2 - \frac{1}{(2-\gamma)(1-\gamma)} |x|^{2-\gamma} \right) + O\big( (y^2+|x|^{2-\gamma})^{1+\delta} \big)
\end{align*}
for some universal constant $\delta=\delta(\gamma) > 0$. 

In the study of Schauder estimates up to boundary for degenerate Monge--Ampère equations, Le and Savin \cite{LS17} show that a bounded solution $w$ of the singular equations with $\gamma<0 $
\begin{equation*} 
	\left\{\begin{aligned}
	 	\Delta_{x'} w + x_n^{\gamma} w_{nn} &=0 && \text{in } B_1^+\\
		w &= 0 &&  \text{on } \{x_n=0\}
	\end{aligned} \right.
\end{equation*}
satisfies
\begin{equation*}
	|w(x) - p(x') x_n | \leq C(x_1^2 + \cdots x_{n-1}^2 + x_n^{2-\gamma})^{\frac{3-\gamma}{2-\gamma}}
\end{equation*}
for all $x \in B_{1/2}^+$, where $C$ is a universal constant and $p(x')$ is a standard polynomial of degree $1$. 

As they used a distance function that allows solutions and equations to be scaling invariant, we also defined a distance function corresponding to \eqref{eq:main} in our study. It is noteworthy that such a fractional order expansion and estimates are developed to higher orders by $s$-polynomial and the distance function.
\item Gauss Curvature Flow: Related studies can also be found in differential geometry, for example, Daskalopoulos and Hamilton \cite{DH99} showed that the regularity of the interface for the Gauss curvature flow with flat sides can be transformed into the regularity of solutions for the following degenerate equations
\begin{equation} \label{eq_gcf}
	v_t = x v_{xx} + v_{yy}  + \nu v_x + f \quad \text{in } \mathbb{R}_+^2 \times (0,T]. 
\end{equation}
They showed Schauder estimates for smooth solutions of \eqref{eq_gcf} using a new metric that preserves scaling. 
\item Non-local Equations: In the research field of non-local equations, Caffarelli and Silvestre \cite{CS07} showed that the solution $u$ of the degenerate/singular equations 
\begin{equation} \label{eq:ext_prob}
	\left\{\begin{aligned}
	 	\Delta_x u + z^{\frac{2s-1}{s}} u_{zz} &=0 && \text{in } \mathbb{R}^n \times [0,\infty)\\
		u &= f &&  \text{on } \mathbb{R}^n \times \{0\}
	\end{aligned} \right.
\end{equation}
satisfies
\begin{equation*}
	(-\Delta)^s f(x) = - C(n,s) u_z (x,0)
\end{equation*}
for some constant $C(n,s) > 0$. The regularity of solutions for the extension problem \eqref{eq:ext_prob} can be applied to the regularity of solutions for the fractional Laplace equations. In fact, in \cite{CS07}, H\"older's regularity of solutions for fractional Laplace equations was shown by using the Harnack inequality for the extension problem \eqref{eq:ext_prob}.
\item Mathematical Finance: The Black--Scholes equations for the constant elasticity of variance (CEV) model which was introduced by Cox \cite{Cox75} and Cox and Ross \cite{CR76}. The risky asset's price $X_t$ of the CEV model evolves according to the following stochastic differential equations
\begin{equation*}
	\left\{\begin{aligned}
		dX_t&=\mu X_t \,dt+\sigma X_t^{\gamma/2} \,dW_t \\
		X_0&=x,
	\end{aligned}\right.
\end{equation*}
where $W_t$ is a one-dimensional Brownian motion for some positive constants $\mu$, $\sigma$, and $\gamma$. Using the Feynman-Kac formula, we can derive the following degenerate equations
\begin{equation} \label{eq:cev}
	u_t+\frac{1}{2}\sigma^2 x^{\gamma}u_{xx} +rxu_x -ru  =0  \quad  \text{in } \mathbb{R}_{>0} \times [0,T) .
\end{equation}
Kim, Lee, and Yun \cite{KLY23} generalized \eqref{eq:cev} to $n$-dimensions and showed higher regularity of \eqref{eq:cev} in the case of $1<\gamma<2$. However, as with the aforementioned equation, many applications occur for $\gamma\leq1$, so we need to study \eqref{eq:main} for $\gamma\leq1$.
\end{enumerate}
Although there is a slight difference compared to \eqref{eq:main}, the methodology covered in our study is expected to be applicable to equations that have similar degenerate/singular structures of \eqref{eq:main} like aforementioned equations. In addition, it is a new version of Schauder theory in the sense that can directly show higher regularity of solutions without using the bootstrap method for uniformly parabolic equations. 

In order to obtain the heuristic idea of $s$-polynomial, let us start with a simplified version of \eqref{eq:main}. The following equation
\begin{equation} \label{eq:simple}
	u_t = x^{\gamma} u_{xx} + 1 \quad \text{in } Q_1^+ \quad (0<\gamma<1)
\end{equation}
admits a solution of the form
\begin{align} \label{sol_simple}
	u(x,t) &= \frac{t^2}{2} x - \frac{1}{(2-\gamma)(1-\gamma)} x^{2-\gamma} + \frac{t}{(3-\gamma)(2-\gamma)} x^{3-\gamma}  \\
	&\qquad\qquad\qquad\qquad\qquad + \frac{1}{(5-2\gamma)(4-2\gamma)(3-\gamma)(2-\gamma)}x^{5-2\gamma} .\nonumber
\end{align}
Since $u$ is not twice differentiable at $x=0$ in the $x$-direction, the regularity of $u$ for a spatial variable is at most $C^{1,\alpha}$, so the polynomial that satisfies \eqref{def_cka} with $k=1$ is $p(x,t) = \frac{t^2}{2} x$. However, considering the following estimate
\begin{equation} \label{inq:frac_exp} 
	\left| u(x,t) - \frac{t^2}{2} x - \frac{1}{(2-\gamma)(1-\gamma)} x^{2-\gamma} - \frac{t}{(3-\gamma)(2-\gamma)} x^{3-\gamma} \right| \leq C r^{5-2\gamma}
\end{equation}
for all $X\in Q_r^+$, it can be expected that fractional order expansion is possible. Also, \eqref{inq:frac_exp} contains more terms than $p$ and provides a more accurate approximation.

Looking at the pattern of each monomial when it is substituted into \eqref{eq:simple}, it is possible to find a solution with a fractional order expansion of more terms than \eqref{sol_simple}. This analysis can be extended for \eqref{eq:main}, which is the motivation to consider $s$-polynomials. Although the forcing term $f$ is a general function rather than a constant function, the solution of \eqref{eq:main} can be approximated by such a fractional order expansion. This is why we established Schauder estimates using $s$-polynomials rather than standard polynomials. In addition, the distance function for uniformly parabolic equations does not preserve the scaling at the boundary $\{x_n=0\}$, so it is not suitable for describing the behavior of solutions for \eqref{eq:main} in the neighborhood of $\{x_n = 0\}$. Inspired by \cite{DH99,DS09,LS17}, we define a scaling-preserving distance function $s : \overline{Q_1^+} \times \overline{Q_1^+} \to [0,\infty) $ and use it to show Schauder estimates.

Since the bootstrap method is not applicable in \eqref{eq:main}, it is necessary to obtain $C_s^{k,2+\alpha}$-regularity of solutions for \eqref{eq:main} directly from $C_s^{k,\alpha}$-regularity of the coefficients of $L$. However, it is not enough to consider perturbative methods, by “freezing” the coefficients around a certain point. In order to solve this difficulty, we will use $s$-polynomial approximation of coefficients that fully reflects the regularity of coefficients and study equations with $s$-polynomial coefficients first.

The solution $u$ of \eqref{eq:main} satisfies the following estimates depending on the range of $\gamma$.
\begin{numcases}{|u(X)|\leq}
	C x_n \approx C s[X,O]^{\frac{2}{2-\gamma}} & $(\gamma  <1)$ \label{lip} \\ 
	Cx_n^{2-\gamma} \approx C s[X,O]^2 & $(1 < \gamma < 2)$ \label{hol}  
\end{numcases}
for all $X \in \overline{Q^+_{1/2}}$, where $C>0$ is a universal constant. Assuming $f = 0$, the order of \eqref{hol} can be improved to a higher order, but not in the case of \eqref{lip}. In the case of $1<\gamma<2$, $C_s^{2+\alpha}$-regularity of solutions for \eqref{eq:main} is directly obtained from the improved estimates of \eqref{hol}, and then the generalized coefficient freezing method can be applied immediately thanks to $C_s^{2+\alpha}$-regularity. However, for $\gamma < 1$, the regularity of solutions for \eqref{eq:main} expected from \eqref{lip} is $C_s^{\alpha}$-regularity or $C_s^{1,\alpha}$-regularity depending on the range on $\gamma$. Thus, a process to show $C_s^{2+\alpha}$-regularity of solutions for \eqref{eq:main} is necessary. In addition, since $s$-polynomials for $\gamma=1$ include logarithmic functions, an appropriate method for dealing with them is also required.

The paper is organized as follows. In \Cref{sec:pre}, we introduce notations used throughout this paper and state the main theorem. Also, the maximal principle, comparison principle, the existence of a unique solution for the degenerate/singular equations are covered in \Cref{sec:pre}. It is similar to uniformly parabolic equations, so readers familiar with it can skip it. In the last part of \Cref{sec:pre}, we derive global regularity from interior regularity and boundary regularity of solutions for \eqref{eq:main}. In \Cref{sec:hol_reg}, we prove the boundary Lipschitz estimates of solutions for \eqref{eq:main} and use it to show global $C_s^{\alpha}$-regularity of solutions  for \eqref{eq:main}. In \Cref{sec:reg_cc}, we prove $C^{1,\alpha}$-regularity and $C_s^{k,2+\alpha}$-regularity of solutions for equations with constant coefficients. In \Cref{sec:gst}, we deal with generalized Schauder theory, which is the main idea of this study. Unlike equations with constant coefficients, $C_s^{k,2+\alpha}$-regularity cannot be obtained immediately, so $C_s^{2,\alpha}$-regularity for solutions of \eqref{eq:main} and $C_s^{k,2+\alpha}$-regularity of solutions for equations with $s$-polynomial coefficients are first shown, and then $C_s^{k,2+\alpha}$-regularity of solutions for \eqref{eq:main} can be obtained by using generalized coefficient freezing method.
%
%
\section{Preliminaries} \label{sec:pre}
\subsection{Notations}
We summarize some basic notations as follows.
\begin{enumerate}
\item Points: For $x=(x_1, \cdots,x_n) \in \mathbb{R}^n$, we denote $x'=(x_1, \cdots,x_{n-1}) \in \mathbb{R}^{n-1}$, $X=(x,t) \in \mathbb{R}^{n+1}$, and $O=(0,\cdots,0)\in \bR^{n + 1}$, respectively. For $r>0$, we denote $rX = (rx',r^{\frac{2}{2-\gamma}}x_n, r^2 t)$ which allows solutions and equations to be scaling invariant.
\item Sets: We denote the open upper half-space and the set of nonnegative integers as $\mathbb{R}_{+}^n = \{x \in \mathbb{R}^n : x_n > 0\} $ and $ \mathbb{N}_{0}=\mathbb{N} \cup \{0\}$, respectively. We denote an intrinsic cube with side $2r$ and center $Y=(y,\tau)$ as
$$ Q_r^+(Y) =  \{x \in \mathbb{R}_{+}^n: |x_i-y_i|< r \,(1 \leq i < n), \, | x_n^{\frac{2-\gamma}{2}}-y_n^{\frac{2-\gamma}{2}} |< r  \} \times (\tau -r^2 , \tau]$$
and standard cube with side $2r$ and center $Y=(y,\tau) \in  \mathbb{R}^n $ as
$$ Q_r(Y) =  \{x \in  \mathbb{R}^n :|x_i-y_i|< r \,(1 \leq i \leq n) \} \times (\tau -r^2 , \tau].$$ 
For convenience, we denote $Q_r^+ = Q_r^+(O)$. We denote the set of degrees for $s$-polynomial as
\begin{equation*}
	\mathcal{D} = \Big\{i+\frac{2j}{2-\gamma}: i,j \in \mathbb{N}_0 \Big\}.
\end{equation*}
\item Universal constant means a constant that depends only on $n$, $\lambda$, $\Lambda$, $\gamma$, $k$, and $\alpha$ with $n \in \mathbb{N} $, $k \in \mathbb{N}_0$, $0 < \alpha < 1$, $0<\lambda\leq \Lambda$, and $\gamma \leq 1$.
\item Distance functions: The parabolic distance function $d : \overline{Q_1^+} \times \overline{Q_1^+} \to [0,\infty) $ from $X=(x,t)$ to $Y=(y,\tau)$ is given by
\begin{equation*}
	d[X,Y] = \max \left\{ \max_{1 \leq i \leq n} |x_i-y_i|, \sqrt{|t-\tau|}  \right\} .
\end{equation*}
	In addition, the intrinsic distance function $s : \overline{Q_1^+} \times \overline{Q_1^+} \to [0,\infty) $ from $X=(x,t)$ to $Y=(y,\tau)$ is given by
\begin{equation*}
	s[X,Y] =  \max \left\{ \max_{1\leq i < n}|x_i-y_i|,  |x_n^{\frac{2-\gamma}{2}}-y_n^{\frac{2-\gamma}{2}}|, \sqrt{|t-\tau|}  \right\} .
\end{equation*}
\item Partial derivatives:
	We denote partial derivatives of $u$ as subscriptions.
\begin{equation*}
	u_t = \partial_t u = \frac{\partial u}{\partial t} , \quad u_i = D_iu =\frac{\partial u}{\partial x_i} , \quad \text{and} \quad  u_{ij} = D_{ij} u = \frac{\partial^2 u}{\partial x_i \partial x_j} .
\end{equation*}
\item Multiindex notation: A vector of the form $\beta =(\beta_1,\beta_2,\cdots,\beta_k) \in \mathbb{N}_{0}^k$ is called a $k$-dimensional multiindex of order $|\beta| = \beta_1 + \beta_2 + \cdots + \beta_k$. For $k$-dimensional multiindices $\beta, \tilde{\beta} \in \mathbb{N}_0^k$, $\tilde{\beta} \leq \beta$ means $\tilde{\beta}_i \leq \beta_i $ $(i=1,2,\cdots,k)$. The factorial and binomial coefficient of a mutiindex are defined as follows:
$$\beta! = \beta_1! \beta_2! \cdots \beta_k! \qquad \text{and} \qquad {\beta \choose \tilde{\beta}} = \frac{\beta!}{\tilde{\beta}! (\beta-\tilde{\beta})!} .$$
Given $\beta \in \mathbb{N}_0^k$, $m \in \mathbb{N}$, and $(x_1,x_2,\cdots, x_k) \in \mathbb{R}^k$, define
$$(x_1,x_2,\cdots, x_k)^{\beta} = x_1^{\beta_1} x_2^{\beta_2}  \cdots  x_k^{\beta_k}, \qquad  D_{(x_1,x_2,\cdots, x_k) }^{\beta} u= \frac{\partial^{|\beta|} u}{\partial x_1^{\beta_1}\partial x_2^{\beta_2} \cdots \partial x_k^{\beta_k}} , $$
and $ D_{(x_1,x_2,\cdots, x_k) }^m u = \{D_{(x_1,x_2,\cdots, x_k) }^{\beta} u : |\beta| = m\}.$
\begin{remark}[Leibniz's formula] \label{leibniz_form}
Let $\Omega$ be a domain in $\mathbb{R}^n$ and let $u,v :\Omega \to \mathbb{R}$ be smooth functions. Then
 $$D_x^{\beta}(uv) = \sum_{\tilde{\beta} \leq \beta} {\beta \choose \tilde{\beta}} D_x^{\tilde{\beta}} u D_x^{\beta-\tilde{\beta}} v.$$
 \end{remark}
\item $s$-polynomial: Let $\kappa$ be a positive real number. We say $p$ is an $s$-polynomial of degree $m$ corresponding to $\kappa$ at $Y=(y,\tau) \in  \mathbb{R}^n_+ \times \mathbb{R}$ provided 
\begin{equation}\label{def:s-pol}		
p(X) = \left \{ 
\begin{aligned}
	& \sum_{\beta, i, j, l} A^{\beta  ijl} (x'-y')^{\beta}  (x_n^{\frac{2-\gamma}{2}}-y_n^{\frac{2-\gamma}{2}} )^{i} (x_n-y_n)^j (t-\tau)^l && (\gamma < 1) \\
	& \sum_{\beta, i, j, l} A^{\beta  i j l} (x'-y')^{\beta} (\sqrt{x_n} - \sqrt{y_n} )^{i}  (\sqrt{x_n} \log x_n  -\sqrt{ y_n} \log y_n )^{j} (t-\tau)^l && (\gamma = 1) 
\end{aligned} \right.
\end{equation}
for some $A^{\beta ijl} \in \mathbb{R}$, where $\beta \in \mathbb{N}_0^{n-1}$ and $i,j,l \in \mathbb{N}_0$ satisfy 
$$  \left \{ 
\begin{aligned}
&| \beta | + i + \frac{2j}{2-\gamma} + 2 l < \kappa && (\gamma<1) \\
&| \beta | + i + j + 2 l < \kappa  && (\gamma=1).
\end{aligned} \right.$$
The degree $m$ corresponding to $\kappa$ is given by
$$\deg p =  
 \left \{ 
\begin{aligned}
&\max \Big\{ | \beta | + i + \frac{2j}{2-\gamma} + 2l \in [0,\kappa): \beta \in \mathbb{N}_{0}^{n-1}, \, i,j,l \in \mathbb{N}_{0} \Big\}  &&  (\gamma<1) \\
&\max \left\{ | \beta | + i + j + 2l \in [0,\kappa): \beta \in \mathbb{N}_{0}^{n-1}, \, i,j,l \in \mathbb{N}_{0} \right\} && (\gamma=1) .
\end{aligned} \right.$$
\item Hölder norms: Let $\Omega \subset \mathbb{R}^n_+ \times \mathbb{R} $ be an open set and $\alpha \in (0,1)$. We define the $\alpha^{\rm th}$-Hölder seminorm of $u:\Omega \to \mathbb{R}$ to be
\begin{equation*}
	{[u]}_{C^{\alpha}(\overline{\Omega})} \coloneqq \sup_{\substack{X \not = Y \\ X, Y \in \Omega}}\frac{|u(X)-u(Y)|}{d[X,Y]^{\alpha}} .
\end{equation*}
Also we define the $\alpha^{\rm th}$-H\"older norm
\begin{equation*}
	\|u\|_{C^\alpha(\overline{\Omega})}  \coloneqq \|u\|_{C^0(\overline{\Omega})} + [u]_{C^\alpha(\overline{\Omega})} .
\end{equation*}
Moreover, we define for a nonnegative integer $k$,
\begin{equation*}
	\|u\|_{C^{k,0}(\overline{\Omega})}  \coloneqq \sum_{\substack{|\beta|+2i\leq k \\ \beta \in \mathbb{N}_{0}^n, \, i \in \mathbb{N}_{0} }} \|D_x^{\beta} \partial_t^i u \|_{C^0(\overline{\Omega})} 
\end{equation*}
and
\begin{equation*}
	\|u\|_{C^{k,\alpha}(\overline{\Omega})}  \coloneqq \|u\|_{C^{k,0}(\overline{\Omega})} + \sum_{\substack{|\beta|+2i = k \\ \beta \in \mathbb{N}_{0}^n, \, i \in \mathbb{N}_{0} }} [D_x^{\beta} \partial_t^i u ]_{C^\alpha (\overline{\Omega})}. 
\end{equation*}
Next, we define the H\"older norm with respect to the distance function $s$
\begin{equation*}
	\|u\|_{C^{k,\alpha}_s(\overline{\Omega})}  \coloneqq \sup_{Y \in \overline{\Omega}} \sup_{r>0} \, \inf_{p} \bigg\{\frac{|u(X)-p(X)|}{r^{k+\alpha}} + \sum_{\beta,i,j,l}| A^{\beta ijl}| : X \in \overline{Q_{r}^+ (Y)} \cap \overline{\Omega} \bigg\},
\end{equation*}
where $p$ is among $s$-polynomials of degree $m$ corresponding to $\kappa = k + \alpha$. In particular, we denote the $\alpha^{\rm th}$-Hölder norm with respect to distance function $s$,
$$\|u\|_{C^{\alpha}_s(\overline{\Omega})} \coloneqq \|u\|_{C^{0,\alpha}_s(\overline{\Omega})} .$$
Finally, we define the $(2+\alpha)^{\rm th}$-Hölder norm and the higher $(2+\alpha)^{\rm th}$-Hölder norm of $u:\Omega \to \mathbb{R}$ to be
\begin{equation*}
	\|u\|_{C^{2+\alpha}_s(\overline{\Omega})}  \coloneqq \| u \|_{C^{2,\alpha}_s(\overline{\Omega})} + \|(\partial_t - L) u \|_{C^{\alpha}_s(\overline{\Omega})} 
\end{equation*}
and	
\begin{equation} \label{norm_rel}
	\|u\|_{C^{k,2+\alpha}_s(\overline{\Omega})} \coloneqq \| u \|_{C^{k+2,\alpha}_s(\overline{\Omega})} + \|(\partial_t - L) u \|_{C^{k,\alpha}_s(\overline{\Omega})}.
\end{equation}
\item Function spaces:
	For any nonnegative integer $k$, $0 < \alpha < 1$, and domain $\Omega \subset \mathbb{R}_+^{n} \times \mathbb{R}$, we define function space $C^{k,\alpha}_s(\overline{\Omega})$ as the completion of $\{u\in C^\infty(\Omega)\cap C(\overline{\Omega}):\|u\|_{C^{k,\alpha}_s(\overline{\Omega})}<\infty\} $ for $\|\cdot\|_{C^{k,\alpha}_s(\overline{\Omega})}$ and $ C^{k,2+\alpha}_s(\overline{\Omega}) = \{u\in C^{k+2,\alpha}_s(\overline{\Omega}):\|u\|_{C^{k,2+\alpha}_s(\overline{\Omega})}<\infty\}$ which is the space that allows the operator $\partial_t-L:C^{k,2+\alpha}_s(\overline{\Omega}) \to C^{k,\alpha}_s(\overline{\Omega})$ to be well-defined.

All of the function spaces defined so far become Banach spaces.
\end{enumerate}

Through this article, we assume that the symmetric matrix $(a^{ij})$ has the following property:
\begin{equation} \label{parabolicity}
	\lambda|\xi|^2\leq a^{ij}(X)\xi_i\xi_j\leq \Lambda |\xi|^2\quad\mbox{for any } X \in \overline{Q_1^+}, ~ \xi \in \mathbb R^n 
\end{equation}
and $b^i,c$ have the boundedness as follows:
\begin{equation*}
	\sum_{i=1}^n \|b^i\|_{L^\infty(Q_1^+)}+\|c\|_{L^\infty(Q_1^+)} \leq \Lambda.
\end{equation*}

In addition, since the function $v = e^{-(\Lambda+1) t} u$ satisfies 
$$ v_t = L v - (\Lambda+1) v + f  \quad \mbox{in } Q^+_1 ,$$
we may assume that the coefficient $c$ of $L$ is negative. 
\subsection{Main Result}
In this subsection, we are going to state the main theorem.
\begin{theorem}\label{thm:main}
Let $k \in \mathbb{N}_{0}$, $0<\alpha< 1$ with $ k+2+\alpha \notin \mathcal{D}$, and assume 
\begin{equation*}
	a^{ij}, \, b^i, \, c, \, f\in C^{k,\alpha}_s(\overline{Q_1^+}) \quad (i,j=1,2,\cdots,n).
\end{equation*}
Suppose $u \in C^2 (Q^+_1)\cap  C(\overline{Q^+_1})$ is a solution of \eqref{eq:main} satisfying $u=0$ on $\{ X \in \partial_p Q^+_1 :x_n=0\}$. Then $u\in C^{k,2+\alpha}_s(\overline{Q_{1/2}^+})$ and 
\begin{equation*}
	\|u\|_{C^{k, 2 + \alpha}_s(\overline{Q^+_{1/2}})} \leq C \left( \|u\|_{L^{\infty}(Q^+_1)} + \|f\|_{C^{k,\alpha}_s(\overline{Q^+_1})} \right) ,
\end{equation*}
where $C$ is a positive constant depending only on $n$, $\lambda$, $\Lambda$, $\gamma$, $k$, $\alpha$, $\|a^{ij}\|_{C^{k,\alpha}_s(\overline{Q^+_1})}$, $\|b^{i}\|_{C^{k,\alpha}_s(\overline{Q^+_1})}$, and $\|c\|_{C^{k,\alpha}_s(\overline{Q^+_1})}$.
\end{theorem}
\subsection{Maximum Principle}  
First, we will show the maximum principle for the following initial/boundary-value problem
\begin{equation}
\left\{\begin{aligned} \label{eq:IBVP}
 	u_t&= Lu + f  && \textnormal{in } Q_1^+ \\
	u &=g   && \textnormal{on } \partial_p Q_1^+.
\end{aligned}\right.
\end{equation}
\begin{lemma} \label{weak_max1}
Suppose $u \in C^2 (Q_1^+) \cap C(\overline{Q_1^+})$ satisfies $u_t < Lu$. If $u<0$ on $\partial_p Q_1^+$, then 
$$u(X) < 0 \quad \mbox{for all } X \in \overline{Q_1^+} . $$
\end{lemma}

\begin{proof}
Suppose not. Then there exists $\tau = \inf \{t: u(x,t) \ge 0 \mbox{ for some } (x,t) \in \overline{Q_1^+} \} $.
Since $u \in C(\overline{Q_1^+})$ and $u<0$ on $\partial_p Q_1^+$, we know that there is $Y=(y, \tau) \in \overline{Q_1^+}  \setminus \partial_p Q_1^+$ such that $u(Y) = 0$. Meanwhile $\tau$ is the first time $u$ becomes nonnegative, the function $U(x)\coloneqq u(x,\tau)$ attains its maximum at $y$ and hence $u_t(Y) \ge 0$, $Du(Y) =0$, and $D^2 u(Y) \leq 0$.
Then we have
$$\mathscr{D}^2 u  (Y) \coloneqq \left(\begin{array}{c|c}
  I_{n-1} & \textbf{0} \\ \hline
   \textbf{0} &  y_{n}^{\gamma/2}
\end{array}
\right)
 D^2 u (Y)
\left(\begin{array}{c|c}
  I_{n-1} & \textbf{0} \\ \hline
   \textbf{0} &  y_{n}^{\gamma/2}
\end{array}
\right) \leq 0 $$
and this implies that  $u_t (Y) - Lu(Y) = u_t (Y) - \mbox{tr}\big(a^{ij}(Y) \mathscr{D}^2 u  (Y) \big) \ge 0$. This yields a contradiction to $u_t < Lu$.
\end{proof}
\begin{lemma} \label{weak_max2}
Suppose $u \in C^2 (Q_1^+) \cap C(\overline{Q_1^+})$ satisfies $u_t \leq Lu$. If $u \leq 0$ on $\partial_p Q_1^+$, then $u \leq 0$ in $ \overline{ Q_1^+} $.
\end{lemma}

\begin{proof}
For any $\varepsilon > 0$, consider $v(X) = u(X) -\varepsilon (t+2)$. Then we have $v_t < Lv$ in $Q_1^+$ and $v  < 0$ on $\partial_p Q_1^+$. Thus, by \Cref{weak_max1} $u(X) < \varepsilon (t+2) $ for all $X \in \overline{Q_1^+}$. We finish the proof by letting $\varepsilon \to 0$.
\end{proof}
\begin{lemma} \label{apriori_bound1}
	Suppose $u \in C^2(Q_1^+) \cap C(\overline{Q_1^+})$ is a solution of \eqref{eq:IBVP}. Then
\begin{equation*}
\|u\|_{L^{\infty}(Q_1^+)} \leq C \left( \|f\|_{L^{\infty}(Q_1^+)} +  \|g\|_{L^{\infty}(Q_1^+)} \right) ,
\end{equation*}	
where $C>0$ is a universal constant.
\end{lemma}

\begin{proof}
Let $v = u - e^{ t + 1 } ( \|f\|_{L^{\infty}(Q_1^+)} +  \|g\|_{L^{\infty}(Q_1^+)} )$. Then we have  $v_t \leq Lv$ in $Q_1^+ $ and $v \leq 0$ on $\partial_p Q_1^+.$ By \Cref{weak_max2}, we have  $u(X) \leq e  ( \|f\|_{L^{\infty}(Q_1^+)} +  \|g\|_{L^{\infty}(Q_1^+)}  ) $ for all $X \in Q_1^+$. Considering the function $w = -u - e^{ t + 1 } ( \|f\|_{L^{\infty}(Q_1^+)} +  \|g\|_{L^{\infty}(Q_1^+)} )$, we can obtain the lower bound in a similar way.
\end{proof}

\begin{lemma} \label{apriori_bound2}
	Suppose $u \in C^2(Q_1^+) \cap C(\overline{Q_1^+})$ is a solution of \eqref{eq:IBVP} with $\gamma<0$ and $x_n^{-\gamma} f \in L^{\infty}(Q_1^+)$.
	Then
\begin{equation*}
\|u\|_{L^{\infty}(Q_1^+)} \leq C \left(\|x_n^{-\gamma} f\|_{L^{\infty}(Q_1^+)} +  \|g\|_{L^{\infty}(Q_1^+)} \right) ,
\end{equation*}	
where $C>0$ is a universal constant.
\end{lemma}

\begin{proof}
Let $w = e^{ \mu(t + 1) } \|x_n^{-\gamma}f\|_{L^{\infty}(Q_1^+)}(x_n-x_n^2)$. Then for sufficiently large $\mu>0$, we have  $w_t - Lw \ge \lambda\|x_n^{-\gamma}f\|_{L^{\infty}(Q_1^+)} x_n^{\gamma}$ in $Q_1^+ $. As in the proof \Cref{apriori_bound1}, we can see the desired result, by considering the function $v=u-(w/\lambda + \|g\|_{L^{\infty}(Q_1^+)} )$.
\end{proof}
\subsection{Existence of Solutions}
In this subsection, we assume that $a^{ij}$, $b^i$, $c$, and $f$ are $C_{\textnormal{loc}}^{\alpha}(Q_1^+)$ for the existence of a classical solution.   
\begin{definition} 
A  function $v \in C(\overline{Q_1^+})$ is called a subsolution (resp. supersolution) of \eqref{eq:IBVP} if 
\begin{equation*} 
	v \leq g \, (\textnormal{resp}. \ge g) \quad \text{on }  \partial_p Q_1^+
\end{equation*}
and if for any $\Omega \subset  Q_1^+ $, the solution $\tilde{v} \in C^2(\Omega) \cap C(\overline{\Omega})$ of
\begin{equation*}
\left\{\begin{aligned}
	\tilde{v}_t &= L\tilde{v} + f   &&\mbox{in } \Omega \\
	\tilde{v} &=v   &&\mbox{on }\partial_p \Omega  
\end{aligned}\right.
\end{equation*}	
	is greater (resp. less) than or equal to $v$ in $\Omega $. 
\end{definition}	
\begin{lemma}[Comparison Principle] \label{com_prin}
If $w \in  C(\overline{Q_1^+})$ is a supersolution of \eqref{eq:IBVP} and  $v \in C(\overline{Q_1^+})$ is a subsolution of \eqref{eq:IBVP}, then
\begin{equation*}
w \ge v \quad \text{in } Q_1^+.
\end{equation*}	
\end{lemma}
\begin{proof}
Suppose $h=v-w $ has a positive maximum value $M$ achieved at some point $Y=(y,\tau) \in \overline{Q_1^+} \setminus  \partial_p Q_1^+$. Then, we have $h(Y) = M   >0$ and $h(X) \leq M$ for all $X \in \Omega,$ where $\Omega=\{X \in Q_1^+: x_n > y_n/2, \, t < \tau\}$. Take $\tilde{v} \in C^2(\Omega) \cap C(\overline{\Omega})$ and $\tilde{w}\in C^2(\Omega) \cap C(\overline{\Omega})$ satisfying 
\begin{equation*}
\left\{\begin{aligned}
	\tilde{v}_t &= L\tilde{v} + f   && \text{in } \Omega  \\
	\tilde{v} &=v   && \text{on }\partial_p \Omega 
\end{aligned}\right.
\qquad \text{and} \qquad 
\left\{\begin{aligned}
	\tilde{w}_t &= L\tilde{w} + f   && \text{in } \Omega \\
	\tilde{w} &=w   && \text{on }\partial_p \Omega  .
\end{aligned}\right.
\end{equation*}
Then, the function $\tilde{h} = \tilde{v}-\tilde{w}-M$ satisfying $\tilde{h}_t \leq L\tilde{h}$ in $\Omega$ and $\tilde{h} \leq 0$ on $\partial_p \Omega$. Furthermore, since $v$ and $w$ are subsolution and supersolution of \eqref{eq:IBVP}, respectively, we know that $\tilde{v} \ge v$ in $\Omega$ and $\tilde{w} \leq w$ in $\Omega$. Hence, we have $\tilde{h}(Y)  \ge 0 $. 
The operator $(\partial_t - L)$ is a uniformly parabolic in $\Omega$ since $\Omega$ is away from the boundary $\{x_n=0\}$. By the strong maximum principle for uniformly parabolic equations, we have $\tilde{h} =  \tilde{v}-\tilde{w}-M  = 0$ in $\overline{\Omega}$. It follows that $v-w = M$ on $\partial_p \Omega$. Since $\partial_p Q_1^+ \cap \partial_p \Omega$ is nonempty, this contradicts the assumption that $ v \leq g \leq w \quad \text{on } \partial_p Q_1^+$.
\end{proof}
\begin{lemma} 
Let $k \in \mathbb{N}_{0}$, $0<\alpha<1$, and assume
$$a^{ij}, \, b^i, \, c, \, f \in C^{k,\alpha}_s(\overline{Q_1^+})~(i,j=1,2,\cdots,n) \qquad \text{and} \qquad  g \in C(\overline{Q_1^+}).$$
Then there exists a unique bounded solution $u \in C_{\textnormal{loc}}^{k+2,\alpha}(Q_1^+)\cap C(\overline{Q_1^+})$ of \eqref{eq:IBVP} such that
$$\|u \|_{L^{\infty}(Q_1^+)} \leq C \left( \|f\|_{L^{\infty}(Q_1^+)} +  \|g\|_{L^{\infty}(Q_1^+)} \right) ,$$
where $C>0$ is a universal constant.
\end{lemma}

\begin{proof}
Let $\cA$ be a set of bounded subsolutions of \eqref{eq:IBVP} and let  $u=\sup_{w\in \cA}w$. Consider the function $\varphi_{\pm}=\pm e^{ t + 1}   ( \|f\|_{L^{\infty}(Q_1^+)} +  \|g\|_{L^{\infty}(Q_1^+)}  )$. Note first that $\varphi_+$ is a supersolution of \eqref{eq:IBVP} and $\varphi_{-} $ is a subsolution of \eqref{eq:IBVP}, so $u$ is bounded by $\varphi_{\pm}$ by \Cref{com_prin}. For any cube $Q_r(Y) \subset \joinrel \subset Q_1^+$, the operator $(\partial_t - L)$ is a uniformly parabolic in $Q_r(Y)$ since $Q_r(Y)$ is away from the boundary $\{x_n=0\}$. Thus, we can apply Perron's method for the problem \eqref{eq:IBVP} as in \cite[Lemma 3.3]{KLY23} and hence $u \in C_{\textnormal{loc}}^{k+2,\alpha} (Q_1^+)$. Finally, for any $Y=(y,\tau) \in \{ X \in \partial_p Q_1^+ : x_n = 0 \}$, take a local barrier $w$ near $Y$ given by
\begin{equation*}
	w(x,t)=|x'-y'|^2+  (t-\tau)^2 + 
	\begin{dcases}
		\frac{2x_n}{r} - \left(\frac{x_n}{r}\right)^{2-\gamma/2} - \left(\frac{x_n}{r}\right)^{2-\gamma}& (\gamma <0)\\
		\frac{x_n}{r} - \left(\frac{x_n}{r}\right)^{2-\gamma}& (0 \leq \gamma < 1 )\\
		-\frac{x_n}{r} \log x_n & (\gamma =1) .\\
	\end{dcases}
\end{equation*}
for some small $r>0$. The continuity of $u$ up to the boundary $\{x_n = 0 \}$ is automatically guaranteed since $Y$ is a regular boundary point. In the case $Y \in \{ X \in \partial_p Q_1^+ : x_n \ne 0 \}$, the continuity of $u$ is guaranteed by using the barrier function of the uniformly parabolic equations as in \cite{Lie96}.
\end{proof}
\subsection{Global Regularity} 
The proof of \Cref{thm:main} is sufficient if we show higher regularity of solutions at the boundary. Main equation \eqref{eq:main} is a uniformly parabolic equation in any $\Omega \subset \joinrel \subset Q_1^+$, so if the coefficients and forcing term are in $C_s^{k,\alpha}(\overline{Q_1^+})$, the solution of \eqref{eq:main} is in $C^{k+2,\alpha}_{\textnormal{loc}}(\Omega)$ and global regularity can be obtained by combining interior regularity and boundary regularity. More strictly speaking, it is described by the following lemma and theorem.
\begin{lemma} \label{cond1_grt}
Let  $k \in \mathbb{N}_{0}$, $0<\alpha< 1$, and assume 
$$a^{ij}, \, b^i, \, c, \, f \in C^{k,\alpha}_s(\overline{Q_1^+}) \quad (i,j=1,2,\cdots,n).$$
Suppose $u\in C_{\textnormal{loc}}^{k+2,\alpha}(Q^+_1) $ be a solution of \eqref{eq:main}. Then, 
for each $Y \in \{ X \in \overline{Q^+_{1/2}} : x_n>0\} $, there is an $s$-polynomial $p $ of degree $(k+2)$ at $Y$ such that 
\begin{equation*}
	\|p\|_{C_s^{k+2,\alpha}(\overline{Q_{r/2}^+(Y)}) } \leq C \left( \frac{ \|u\|_{L^{\infty}(Q_{r\rho}^+(Y))} }{r^{k+2+\alpha}}+ \frac{\|f\|_{C^{k,\alpha}_s( \overline{Q_{r\rho}^+(Y)})} }{r^{k+\alpha}}\right)   
\end{equation*}
and
\begin{equation*}
	|u(X)-p (X)| \leq C \left( \frac{ \|u\|_{L^{\infty}(Q_{r\rho}^+(Y))}}{r^{ k+2+\alpha}} + \frac{ \|f\|_{C_s^{k,\alpha}(\overline{Q_{r\rho}^+(Y)})}}{r^{k+\alpha}} \right) s [X,Y]^{k+2+\alpha} .
\end{equation*}
for all $X \in \overline{Q_{r/2}^+(Y)}$ and fixed $\rho \in (1/2,1)$, where $r=y_n^{\frac{2-\gamma}{2}}$ and $C>0$  is a constant which is independent of $Y$. Moreover 
\begin{equation*}
	|p (X)|  \leq C \left( \frac{ \|u\|_{L^{\infty}(Q_{r\rho}^+(Y))} }{r^{k+2+\alpha}}+ \frac{\|f\|_{C^{k,\alpha}_s( \overline{Q_{r\rho}^+(Y)})} }{r^{k+\alpha}}\right)  s[X,Y]^{k + 2 + \alpha}
\end{equation*}
for all $X \in \overline{Q_1^+}$ with $s[X,Y] \ge r/2$. 
\end{lemma}

\begin{proof}
For each $Y= (y,\tau) \in \{ X \in \overline{Q^+_{1/2}} : x_n>0\} $,  we consider the function $v$ defined in a standard cube $Q_{\rho} (0',1,0) $ by $v(x,t)= u(y' + r x',  (rx_n)^{\frac{2}{2-\gamma}}, \tau + r^2 t),$ where $\rho \in( 1/2, 1)$ and $r=y_n^{\frac{2-\gamma}{2}}$. Then, $v$ satisfies the uniformly parabolic equation 
\begin{equation} \label{eq:unif_int}
	v_t  = \tilde{a}^{ij}(X) v_{ij} + \tilde{b}^{i}(X) v_{i} - \frac{\gamma}{2-\gamma} \frac{\tilde{a}^{nn}(X)}{x_n} v_n + \tilde{c}(X)v + \tilde{f}(X)  \quad \mbox{in } Q_{\rho} (0',1,0) ,
\end{equation}
where
\begin{equation*}
	\begin{aligned}
		\tilde{a}^{ij}(X)&=
			\begin{cases} 
				a^{ij} (y' + r x',  (rx_n)^{\frac{2}{2-\gamma}}, \tau + r^2 t) &\text{if }i,j\not= n \\
				\frac{2-\gamma}{2} a^{in} (y' + r x',  (rx_n)^{\frac{2}{2-\gamma}}, \tau + r^2 t)& \text{if }i\not=n, j=n \\
				\frac{2-\gamma}{2} a^{nj} (y' + r x',  (rx_n)^{\frac{2}{2-\gamma}}, \tau + r^2 t)& \text{if }i =n, j \ne n \\
				\frac{(2-\gamma)^2}{4}a^{nn} (y' + r x',  (rx_n)^{\frac{2}{2-\gamma}}, \tau + r^2 t) & \text{if } i=j=n,
			\end{cases}\\
		\tilde{b}^{i}(X)&=
			\begin{cases}
				rb^{i} (y' + r x',  (rx_n)^{\frac{2}{2-\gamma}}, \tau + r^2 t) & \text{if }i\not= n \\
				\frac{2-\gamma}{2}rb^{n} (y' + r x',  (rx_n)^{\frac{2}{2-\gamma}}, \tau + r^2 t) & \text{if }i=n,
			\end{cases}\\
		\tilde{c}(X)&=r^2 c (y' + r x',  (rx_n)^{\frac{2}{2-\gamma}}, \tau + r^2 t) ,\\
		\tilde{f}(X)&=r^2 f(y' + r x',  (rx_n)^{\frac{2}{2-\gamma}}, \tau + r^2 t).
	\end{aligned}
\end{equation*}
Since $Q_{r\rho}^+ (Y)$ is away from $\{X \in \partial_p  Q_1^+ : x_n=0\}$, we can show that 
\begin{equation*} 
\left\{\begin{aligned}
	\|\tilde{a}^{ij}\|_{C^{k,\alpha}(\overline{Q_{\rho} (0',1,0)})} &\leq C\|a^{ij}\|_{C^{k,\alpha}_s(\overline{Q_1^+})} \\
	\|\tilde{b}^{i}\|_{C^{k,\alpha}(\overline{Q_{\rho} (0',1,0)})} &\leq C\|b^{i}\|_{C^{k,\alpha}_s(\overline{Q_1^+})} \\	
	\|\tilde{c}\|_{C^{k,\alpha}(\overline{Q_{\rho} (0',1,0)})} &\leq C\|c\|_{C^{k,\alpha}_s(\overline{Q_1^+})} 
\end{aligned} \right.
\end{equation*}
and  $\|\tilde{f}\|_{C^{k,\alpha}(\overline{Q_{\rho} (0',1,0)})}  \leq r^2 C\|f\|_{C^{k,\alpha}_s( \overline{Q_{r\rho}^+(Y)})} $ for some constant $C$ which is independent $Y$. Thus, by the interior Schauder estimate, there exists a polynomial $\tilde{p}$ of degree $(k+2)$ of the form 
\begin{equation*}
	\tilde{p}(X)=\sum_{\substack{ |\beta|+i+2j\leq k+2 \\ \beta \in \mathbb{N}_0^{n-1}, \,i,j \in \mathbb{N}_0 }}A^{\beta ij} {x'}^\beta(x_n-1)^i t^j
\end{equation*}
and
\begin{equation} \label{norm_tilde_p}
	\|\tilde{p}\|_{C^{k+2,\alpha}(\overline{Q_{\rho} (0',1,0)})} \leq C\Big(\|v \|_{L^{\infty}(Q_{\rho} (0',1,0))} + \|\tilde{f}\|_{C^{k,\alpha}(\overline{Q_{\rho} (0',1,0)})} \Big),
\end{equation}
where $C$ depends on $n$, $\lambda$, $\Lambda$, $k$, $\alpha$, $\|\tilde{a}^{ij}\|_{C^{k,\alpha}(\overline{Q_{\rho} (0',1,0)})}$, $\|\tilde{b}^{i}\|_{C^{k,\alpha}(\overline{Q_{\rho} (0',1,0)})}$, and $\|\tilde{c}\|_{C^{k,\alpha}(\overline{Q_{\rho} (0',1,0)})}$ such that
\begin{equation} \label{inq:int_scha}
	|v(X)- \tilde{p}(X)| \leq C\Big( \|v\|_{L^{\infty}(\overline{Q_{\rho} (0',1,0)})} + \|\tilde{f}\|_{C^{k,\alpha}(\overline{Q_{\rho} (0',1,0)})} \Big)d [X,(0',1,0)]^{k+2+\alpha}
\end{equation}
for all $X \in \overline{Q_{1/2} (0',1,0)}$. Putting 
\begin{align*}
	p(X) &=  \tilde{p} \big(r^{-1}(x'-y') , r^{-1} x_n^{\frac{2-\gamma}{2}},r^{-2} (t- \tau)\big) \\
	&= \sum_{\substack{ |\beta|+i+2j\leq k+2 \\ \beta \in \mathbb{N}_0^{n-1}, \,i,j \in \mathbb{N}_0 }} \frac{A^{\beta ij} }{r^{|\beta|+i +2j}} (x'-y')^\beta(x_n^{\frac{2-\gamma}{2}}-y_n^{\frac{2-\gamma}{2}})^i (t-\tau)^j ,
\end{align*}
we can rewrite \eqref{norm_tilde_p} and  \eqref{inq:int_scha} into
\begin{equation*}
	\|p\|_{C^{k+2,\alpha}( \overline{Q_{r/2}^+(Y)} )} \leq C \left( \frac{ \|u\|_{L^{\infty}(Q_{r\rho}^+(Y))} }{r^{k+2+\alpha}}+ \frac{\|f\|_{C^{k,\alpha}_s( \overline{Q_{r\rho}^+(Y)})} }{r^{k+\alpha}}\right)
\end{equation*}
and
\begin{align*}
	&|u(X)- p(X)| \\
	 & \quad \leq  C\Big( \|u\|_{L^{\infty}(Q_{r\rho}^+(Y))} +r^2\|f\|_{C^{k,\alpha}_s( \overline{Q_{r\rho}^+(Y)})} \Big) d [(r^{-1}(x'- y') , r^{-1} x_n^{\frac{2-\gamma}{2}},r^{-2} (t-  \tau)),(0',1,0)]^{k+2+\alpha} \\
	 & \quad = C \left( \frac{\|u\|_{L^{\infty}(Q_{r\rho}^+(Y))}}{r^{k+2+\alpha}} + \frac{\|f\|_{C^{k,\alpha}_s( \overline{Q_{r\rho}^+(Y)})}}{r^{k+\alpha}} \right)  s[X,Y]^{k+2+\alpha} 
\end{align*}
for all $X \in \overline{Q_{r/2}^+(Y)}$. From \eqref{norm_tilde_p}, we know that
\begin{equation*}
	\sum_{\substack{ |\beta|+i+2j\leq k+2 \\ \beta \in \mathbb{N}_0^{n-1}, \,i,j \in \mathbb{N}_0 }}  |A^{\beta i j }| \leq C \Big( \|u\|_{L^{\infty}(Q_{r\rho}^+(Y))}  + r^2 \|f\|_{C^{k,\alpha}_s( \overline{Q_{r\rho}^+(Y)})}  \Big) 
\end{equation*} 
and hence, we have
\begin{align*}
	|p(X)| &\leq  C\left( \frac{ \|u\|_{L^{\infty}(Q_{r\rho}^+(Y))} }{r^{k+2+\alpha}}+ \frac{\|f\|_{C^{k,\alpha}_s( \overline{Q_{r\rho}^+(Y)})} }{r^{k+\alpha}}\right)  \sum_{\substack{ |\beta|+i+2j\leq k+2 \\ \beta \in \mathbb{N}_0^{n-1}, \,i,j \in \mathbb{N}_0 }} \frac{r^{k+2+\alpha}}{r^{|\beta|+i +2j}} s[X,Y]^{|\beta| + i + 2j} \\
	&\leq C \left( \frac{ \|u\|_{L^{\infty}(Q_{r\rho}^+(Y))} }{r^{k+2+\alpha}}+ \frac{\|f\|_{C^{k,\alpha}_s( \overline{Q_{r\rho}^+(Y)})} }{r^{k+\alpha}}\right)  s[X,Y]^{k + 2 + \alpha}
\end{align*}
for all $X \in \overline{Q_1^+}$ with $s[X,Y] \ge r/2$. 
\end{proof}
\begin{theorem} [Global Regularity up to $\{x_n=0\}$] \label{thm:int_bd_gb}
Let $k \in \mathbb{N}_{0}$, $0 < \alpha < 1$, and assume $u \in C(\overline{Q_1^+})$. Suppose the following statements hold:
\begin{enumerate}
\item If for each $Y \in \{ X \in \overline{Q^+_{1/2}} : x_n>0\} $, there is an $s$-polynomial $p^{k+2}$ of degree $(k+2)$ at $Y$ such that 
\begin{equation*}
	\|p^{k+2}\|_{C_s^{k+2,\alpha}(\overline{Q_{r/2}^+(Y)}) } \leq A \left( \frac{ \|u\|_{L^{\infty}(Q_{r\rho}^+(Y))} }{r^{k+2+\alpha}}+ \frac{\|u_t - Lu\|_{C^{k,\alpha}_s( \overline{Q_{r\rho}^+(Y)})} }{r^{k+\alpha}}\right)   
\end{equation*}
and
\begin{equation*}
	|u(X)-p^{k+2}(X)| \leq A \left( \frac{ \|u\|_{L^{\infty}(Q_{r\rho}^+(Y))}}{r^{k+2+\alpha}} + \frac{ \|u_t - L u\|_{C_s^{k,\alpha}(\overline{Q_{r\rho}^+(Y)})}}{r^{k+\alpha}} \right) s[X,Y]^{k+2+\alpha} .
\end{equation*}
for all $X \in \overline{Q_{r/2}^+(Y)}$ and fixed $\rho \in (1/2,1)$, where $r=y_n^{\frac{2-\gamma}{2}}$ and $A>1$ is a constant that is independent of $Y$. Moreover 
\begin{equation*}
	|p^{k+2}(X)|  \leq A \left( \frac{ \|u\|_{L^{\infty}(Q_{r\rho}^+(Y))} }{r^{k+2+\alpha}}+ \frac{\|u_t -Lu\|_{C^{k,\alpha}_s( \overline{Q_{r\rho}^+(Y)})} }{r^{k+\alpha}}\right)  s[X,Y]^{k + 2 + \alpha}
\end{equation*}
for all $X \in \overline{Q_1^+}$ with $s[X,Y] \ge r/2$. 
\item If for each $\tilde{Y} \in \{ X \in \partial_p Q^+_{1/2} :x_n=0\}$, there is an $s$-polynomial $p^m$ of degree $m$ corresponding to $\kappa = k + 2 + \alpha$ at $\tilde{Y} $ and a constant $B>1$ such that 
$$\left\{\begin{aligned}
	|(\partial_t - L)(u-p^m)(X) | &\leq B s[X,\tilde{Y}]^{k +\alpha} \\
	|u(X)-p^m(X)| &\leq B s[X,\tilde{Y}]^{k + 2 +\alpha}  
\end{aligned} \right.$$
for all $X \in \overline{Q^+_1}$ and
\begin{equation*}
	\|p^m\|_{C^{k+2,\alpha}_s(\overline{Q_1^+})} \leq B .
\end{equation*}
\item For each $s$-polynomial $p$ of degree $m$ corresponding to $\kappa = k+2+ \alpha$ with $\|p\|_{C^{k+2,\alpha}_s(\overline{Q_1^+})} \leq B $, $v= u-p$ satisfies (1) again.
\end{enumerate}
Then $u \in C_s^{k+2,\alpha}(\overline{Q^+_{1/2}})$ and 
\begin{equation*}
\|u\|_{ C_s^{k+2,\alpha}(\overline{Q^+_{1/2}})} \leq ABC,
\end{equation*}
where $C=C(k,\alpha)>0$ is a constant. 
\end{theorem}

\begin{proof}
For each $Y = (y,\tau) \in \overline{Q^+_{1/2}}$, we just need to find an $s$-polynomial $p$ that satisfies 
\begin{equation*}  	
	\|p\|_{C_s^{k+2,\alpha}(\overline{Q_1^+}) } \leq ABC  
\end{equation*}
and
\begin{equation*}
	|u(X)-p(X)| \leq ABC s[X,Y]^{k+2+\alpha} \quad \mbox{for all } X \in \overline{Q^+_1}.
\end{equation*}
If $y_n=0$, then our goal $p$ is just $p^m$ in (2). Let us think about $y_n>0$ case. Let $\tilde{Y} = (y',0,\tau)$. From (2), there is an $s$-polynomial $p^m$ of degree $m$ corresponding to $\kappa = k+2 + \alpha$ at $\tilde{Y}$ and a constant $B>1$ such that 
\begin{equation}\label{bd_reg} 
	\left\{\begin{aligned}
		|(\partial_t - L)(u-p^m)(X) | &\leq B s[X,\tilde{Y}]^{k +\alpha}  \\
		|u(X)-p^m(X)| &\leq B s[X,\tilde{Y}]^{k+2+\alpha}  
	\end{aligned} \right.
\end{equation}
for all $X \in \overline{Q^+_1}$ and $\|p^m\|_{C^{k+2,\alpha}_s(\overline{Q_1^+})} \leq B$. Now applying (1) on the function $v = u - p^m $, we have an $s$-polynomial $p^{k+2}$ of degree $(k+2)$ such that 
\begin{equation} \label{norm_pk2}
	\|p^{k+2}\|_{C_s^{k+2,\alpha}(\overline{Q_{r/2}^+(Y)}) } \leq A \left( \frac{ \|v\|_{L^{\infty}(Q_{r\rho}^+(Y))} }{r^{k+2+\alpha}}+ \frac{\|v_t - Lv\|_{C^{k,\alpha}_s( \overline{Q_{r\rho}^+(Y)})} }{r^{k+\alpha}}\right)   
\end{equation}
and
\begin{equation} \label{int_reg}
	|v(X)-p^{k+2}(X)| \leq A \left( \frac{ \|v\|_{L^{\infty}(Q_{r\rho}^+(Y))}}{r^{k+2+\alpha}} + \frac{ \| v_t - L v\|_{C_s^{k,\alpha}(\overline{Q_{r\rho}^+(Y)})}}{r^{k+\alpha}} \right) s[X,Y]^{k+2+\alpha} 
\end{equation}
for all $X \in \overline{Q_{r/2}^+(Y)}$ and $\rho \in (1/2,1)$, where $r=y_n^{\frac{2-\gamma}{2}}$ and $A>1$ is a constant that is independent of $Y$. Moreover 
\begin{equation} \label{est_pk2_out}
	|p^{k+2}(X)|  \leq A \left( \frac{ \|v\|_{L^{\infty}(Q_{r\rho}^+(Y))} }{r^{k+2+\alpha}}+ \frac{\|v_t -Lv\|_{C^{k,\alpha}_s( \overline{Q_{r\rho}^+(Y)})} }{r^{k+\alpha}}\right)  s[X,Y]^{k + 2 + \alpha}
\end{equation}
for all $X \in \overline{Q_1^+}$ with $s[X,Y] \ge r/2$. From \eqref{bd_reg}, we see
\begin{align} 
	|v(X)| &\leq B s[X,\tilde{Y}]^{k+2+\alpha} \leq B \big(s[X,Y] + s[Y,\tilde{Y}] \big)^{k+2+\alpha} \leq B \left(r\rho + y_n^{\frac{2-\gamma}{2}} \right)^{k+2+\alpha} \nonumber \\
	& \leq BC r^{k+2+\alpha}  \label{est_v}
\end{align}
for all $X \in \overline{Q_{r\rho}^+(Y)}$ and we also see
\begin{equation} \label{est_vt-Lv}
	|( v_t - L v)(X)|  \leq BC r^{k+\alpha}
\end{equation}
for all $X \in \overline{Q_{r\rho}^+(Y)}$. Thus, combining \eqref{norm_pk2}, \eqref{est_v}, and \eqref{est_vt-Lv} gives  
\begin{equation} \label{norm_pk2'}  	
	\|p^{k+2}\|_{C_s^{k+2,\alpha}(\overline{Q_{r/2}^+(Y)}) } \leq ABC.
\end{equation}
Since $p^{k+2}$ is an $s$-polynomial, we can extend the domain $\overline{Q_{r/2}^+(Y)} $ to $\overline{Q_{1}^+}$ in \eqref{norm_pk2'}. Furthermore, combining \eqref{int_reg}, \eqref{est_v}, and \eqref{est_vt-Lv} gives  
$$|u(X)-p^m(X)-p^{k+2}(X)| \leq ABC s[X,Y]^{k+2+\alpha} \quad \text{for all }X \in \overline{Q_{r/2}^+(Y)}.$$ 
Meanwhile, combining \eqref{bd_reg}, \eqref{est_pk2_out}, \eqref{est_v}, and \eqref{est_vt-Lv} gives
\begin{align*}
	|u(X)-p^m(X)-p^{k+2}(X)| & \leq |u(X)-p^m(X)| + |p^{k+2}(X)| \\
	&\leq B  s[X,\tilde{Y}]^{k+2+\alpha}  + ABC s[X,Y]^{k + 2 + \alpha} \\
	&\leq ABC s[X,Y]^{k+2+\alpha} \
\end{align*}
for all $X \in \overline{Q_1^+}$ with $s[X,Y] \ge r/2$. 
\end{proof}
Thanks to \Cref{thm:int_bd_gb} and \eqref{norm_rel}, we are only concerned with showing regularity at the boundary.
%
%
\section{$C_s^{\alpha}$-Regularity} \label{sec:hol_reg}
The goal of this section is to show global $C_s^{\alpha}$-regularity of solutions for \eqref{eq:main}. In \Cref{sec:reg_cc}, we obtain boundary $C^{1,\alpha}$-regularity of solutions for equations with constant coefficients using the result of this section, and in \Cref{sec:gst}, it is used to prove the approximation lemma.
\begin{lemma} \label{lip_est}
Let $u\in C^2(Q^+_1)\cap C(\overline{Q^+_1})$ be a solution of \eqref{eq:main} satisfying $\|u\|_{L^{\infty}(Q^+_1)}\leq 1$ and $u =0$ on $\{X \in \partial_p Q_1^+ : x_n=0\}$ and for some nonnegative constant $ \delta \leq \min \{1 - \gamma/2, 1 - \gamma \}$ and $M > 0$, if $f$ satisfies
\begin{equation} \label{blowup_f}
|f(X)| \leq M x_n^{- \delta} \quad \mbox{for all } X \in  Q_1^+ ,
\end{equation}
then
\begin{equation*}
|u(X) | \leq  
	\begin{cases} 
		C x_n & (\gamma + \delta <1 ) \\ 
 		-Cx_n \log x_n  & (\gamma + \delta = 1)
	\end{cases} 
\end{equation*}
for all $X \in \overline{Q^+_{1/2}}$, where $C>0$ is a constant depending only on $n$, $\lambda$, $\Lambda$, $\gamma$, and $M$.  
\end{lemma}

\begin{proof}
For $\varepsilon < \min\{1, 2/\Lambda, 2/M \}$, we consider the function $\tilde{u}$ defined in $\mathcal{Q}_{1}$ by $\tilde{u}(X)=u  ( 2^{-1} \varepsilon^2X ),$ where $\mathcal{Q}_{\rho}  = \{(x,t) :|x_i| < 2 \rho \varepsilon^{-1} \, ( 1 \leq i < n), \, 0 < x_n <\rho^{\frac{2}{2-\gamma}}, \, - 2 \varepsilon^{-2} \rho^2 < t \leq 0 \} .$
Then, $\tilde{u}$ is a soltuion of $ \tilde{u}_t  = \tilde{L} \tilde{u} + \tilde{f} $ in $ \mathcal{Q}_{1} $ satisfying $\|\tilde{u}\|_{L^{\infty}( \mathcal{Q}_{1})}\leq 1$ and $\tilde{u}=0$ on $\{X \in  \partial_p \mathcal{Q}_{1} : x_n=0\}$, where
\begin{equation*}
	\tilde{L}  =\tilde{a}^{i'j'}(X)D_{i'j'} + 2 x_n^{\gamma/2} \tilde{a}^{i'n}(X) D_{i'n} + x_n^{\gamma} \tilde{a}^{nn}(X)D_{nn} + \tilde{b}^{i'}(X) D_{i'} + x_n^{\gamma/2} \tilde{b}^n(X) D_n + \tilde{c}(X) 
\end{equation*}
with 
$$\tilde{a}^{ij}(X) =a^{ij}   ( 2^{-1} \varepsilon^2 X  ), \quad 
\tilde{b}^{i}(X) =2^{-1} \varepsilon^2 b^i   ( 2^{-1} \varepsilon^2 X), \quad 
\tilde{c}(X) = (2^{-1} \varepsilon^2 )^2c  ( 2^{-1} \varepsilon^2 X), $$
and $\tilde{f}(X) = (2^{-1} \varepsilon^2 )^2  f  ( 2^{-1} \varepsilon^2 X)$. Since $\|b^i\|_{L^{\infty}(Q^+_1)} \leq \Lambda ~(i=1,2,\cdots,n)$ and \eqref{blowup_f}, we have  $|\tilde{b}^i|\leq 2^{-1} \varepsilon^2 \Lambda < \varepsilon$ $(i=1,2,\cdots,n) $ in $ \mathcal{Q}_{1}$ and 
$$ |\tilde{f}(X)| \leq (2^{-1} \varepsilon^2 )^{2 - \frac{2\delta}{2-\gamma} } M x_n^{-\delta} \leq  2^{-1} \varepsilon^2 M x_n^{-\delta} \leq \varepsilon x_n^{-\delta} \quad \mbox{for all } X \in \mathcal{Q}_1 .$$
From these observations, we may assume without loss of generality that $u\in C^2(\mathcal{Q}_{1})\cap C(\overline{\mathcal{Q}_{1}})$ is a solution of  $u_t = Lu + f $ in $\mathcal{Q}_{1}$ satisfying $\|u\|_{L^{\infty}(\mathcal{Q}_1)}\leq 1$ and $u =0$ on $\{X \in \partial_p  \mathcal{Q}_{1} :x_n=0\}$, 
where the coefficients $b^i$, $c$, and the forcing term $f$ satisfy $|b^i | \leq \varepsilon$ $( i = 1,2, \cdots, n)$, $c \leq 0$ in $ \mathcal{Q}_{1},$ and $|f(X)| \leq \varepsilon x_n^{-\delta} $ for all $X \in \mathcal{Q}_{1}$.   Now, we define $\Omega_{\varepsilon} =\{(x,t):|x_i| < \varepsilon^{-1} \, (1 \leq i <n), \,0 < x_n <1, \, -\varepsilon^{-2} < t \leq 0 \}.$
For each $Y=(y,\tau) \in \overline{\mathcal{Q}_{1/2}}$, consider the function $v$ defined in $\Omega_{\varepsilon} \subset \mathcal{Q}_{1}$ by
$$v(x,t) = u (x' + y',x_n, t + \tau ) + \varphi (x_n) -( |x'|^2 - t)  \varepsilon^2,$$
where
\begin{equation*}  
\varphi (x_n)= \begin{cases}
x_n^{2-\gamma - \delta} + x_n^{2 - \gamma/ 2 } - 3 x_n  & (\gamma + \delta < 0) \\
x_n^{2-\gamma - \delta}- 2 x_n & (0 \leq \gamma + \delta <1 ) \\
x_n \log x_n -  x_n  & (\gamma + \delta =1). \\
\end{cases}
\end{equation*}
Then $v\leq0$ on $\partial_p \Omega_{\varepsilon}$ and for $\varepsilon < \frac{1}{3}\left(2-\frac{\gamma}{2}\right)\left(1-\frac{\gamma}{2}\right) \lambda $, we have
$$ v_t - \hat{L} v \leq x_n^{-\delta} \left( - \lambda \eta(\gamma,\delta)  + \Big(5 + \frac{3}{2} |\gamma| \Big) \varepsilon + [1 + 2(n-1) ( \Lambda +1) ] \varepsilon^2  \right) $$
in $\Omega_{\varepsilon}$, where 
\begin{equation*}
	\hat{L}  =\hat{a}^{i'j'}(X) D_{i'j'} + 2 x_n^{\gamma/2} \hat{a}^{i'n}(X)  D_{i'n} + x_n^{\gamma} \hat{a}^{nn}(X)  D_{nn} + \hat{b}^{i'}(X)  D_{i'} + x_n^{\gamma/2} \hat{b}^n(X)  D_n + \hat{c}(X)  
\end{equation*}
with
\begin{equation*}
	\hat{a}^{ij}(X) = a^{ij} (x' + y',x_n, t + \tau ), \quad
	\hat{b}^{i}(X)  = b^i (x' + y',x_n, t + \tau ), \quad
	\hat{c}(X) = c (x' + y',x_n, t + \tau ),
\end{equation*}
and 
\begin{equation*}
	\eta(\gamma,\delta) = \begin{cases}
	(2-\gamma-\delta)(1-\gamma-\delta) & (\gamma+\delta \ne 1)\\
	1 & (\gamma+\delta =1).
	\end{cases}
\end{equation*}
For sufficiently small $\varepsilon>0$, we obtain $v_t - \hat{L}v<0 $ in $\Omega_{\varepsilon}$. By \Cref{weak_max2}, we have
\begin{equation*}
\begin{aligned}
u(x' + y',x_n, t + \tau )\leq ( |x'|^2  - t ) \varepsilon^2 - \varphi (x_n) \quad \mbox{for all } (x,t) \in \overline{\Omega_{\varepsilon}}.
\end{aligned}
\end{equation*}
Thus, if we take $(x',x_n,t) = (0',y_n, 0)$, then
\begin{equation*}
	u(Y)\leq -\varphi (y_n) \leq 
		\begin{cases} 
			3 y_n & (\gamma + \delta < 1) \\ 
			y_n(1-\log y_n) & (\gamma + \delta =1) 
		\end{cases} 
\end{equation*}
for all $Y \in \overline{\mathcal{Q}_{1/2}}$. Next, consider the following function to find the lower bound
\begin{equation*}
	w(x,t)= -u (x' + y',x_n, t + \tau ) + \varphi (x_n) - (|x'|^2 - t) \varepsilon^2  .
\end{equation*}
Similarly, we can obtain the lower bound. Hence, we conclude that
\begin{equation*}
|u(X) | \leq  
	\begin{cases} 
		C x_n & (\gamma + \delta <1 ) \\ 
		-C x_n \log x_n  & (\gamma + \delta = 1) 
	\end{cases} 
\end{equation*}
for all $X \in \overline{Q^+_{1/2}}$.
\end{proof}
\begin{corollary} \label{bdry_holder}
Let $0 < \alpha < \min \{\frac{2}{2-\gamma}, 1 \}$, and assume $f \in L^{\infty}(Q_1^+)$. Suppose $u\in C^2(Q^+_1)\cap C(\overline{Q^+_1})$ is a solution of \eqref{eq:main} satisfying $u =0$ on $\{X \in \partial_p Q_1^+ : x_n=0\}$. Then $u \in C_s^{\alpha}(\overline{Q^+_{1/2}})$ and
$$\|u\|_{C_s^{\alpha}(\overline{Q^+_{1/2}})} \leq C \left(\|u\|_{L^{\infty}(Q_1^+)} + \|f\|_{L^{\infty}(Q_1^+)}\right),$$
where $C>0$ is a universal constant.
\end{corollary} 

\begin{proof}
By \Cref{lip_est}, we have
\begin{align*}
	|u(X) | &\leq  
	\begin{cases} 
		C \left(\|u\|_{L^{\infty}(Q_1^+)} + \|f\|_{L^{\infty}(Q_1^+)}\right) x_n & (\gamma  <1 ) \\ 
		 -C  \left(\|u\|_{L^{\infty}(Q_1^+)} + \|f\|_{L^{\infty}(Q_1^+)}\right) x_n \log x_n  & (\gamma = 1)
	\end{cases} \\
	&\leq  C \left(\|u\|_{L^{\infty}(Q_1^+)} + \|f\|_{L^{\infty}(Q_1^+)}\right)  s[X,O]^{\alpha} 
\end{align*}
for all $X \in \overline{Q_{1/2}^+}$, since the function $x_n^{1-\alpha/2} \log x_n$ is bounded. Thus, we know that for each $Y= (y,\tau) \in \{ X \in \overline{Q^+_{1/2}} : x_n>0\} $, 
\begin{equation}  \label{bd_hol} 
	|u(X)|  \leq C\left(\|u\|_{L^{\infty}(Q_1^+)} + \|f\|_{L^{\infty}(Q_1^+)}\right) s[X,\tilde{Y}]^{\alpha} \quad \mbox{for all } X \in \overline{Q_1^+},
\end{equation}
where $\tilde{Y}= (y',0,\tau)$, since \eqref{eq:main} is invariant for translation in the $x'$-direction. We now consider the function $v$ defined in a standard cube $Q_{\rho} (0',1,0) $ by $v(x,t)=  u(y' + r x',  (rx_n)^{\frac{2}{2-\gamma}}, \tau + r^2 t),$ where fixed $\rho \in( 1/2, 1)$ and $r=y_n^{\frac{2-\gamma}{2}}$ as in the proof of \Cref{cond1_grt}. Then $v$ is a solution of the uniformly parabolic equation \eqref{eq:unif_int}. By the interior H\"older regularity, 
\begin{equation} \label{inq:int_hol}
	|v(X)- v(0',1,0)| \leq C \Big( \|v\|_{L^{\infty}(\overline{Q_{\rho} (0',1,0)})} + \|\tilde{f}\|_{L^{\infty}(\overline{Q_{\rho} (0',1,0)})} \Big)d [X,(0',1,0)]^{\alpha}
\end{equation}
for all $X \in \overline{Q_{1/2} (0',1,0)}$. From \eqref{bd_hol}, we see
\begin{equation} \label{est_v_hol}
	|u(X)|  \leq C\left(\|u\|_{L^{\infty}(Q_1^+)} + \|f\|_{L^{\infty}(Q_1^+)}\right) r^{\alpha}  \quad \text{for all  }X \in \overline{Q_{r\rho}^+(Y)}.
\end{equation}
Then, we can rewrite \eqref{inq:int_hol} into
\begin{align*}
	&|u(X)- u(Y)| \\
	 & \quad \leq  C \Big( \|u\|_{L^{\infty}(Q_{r\rho}^+(Y))} +r^2\|f\|_{L^{\infty}( \overline{Q_{r\rho}^+(Y)})} \Big)d [(r^{-1}(x'- y') , r^{-1} x_n^{\frac{2-\gamma}{2}},r^{-2} (t-  \tau)),(0',1,0)]^{\alpha} \\
	 &\quad =  C \left( \frac{\|u\|_{L^{\infty}(Q_{r\rho}^+(Y))}}{r^{\alpha}} +  r^{2-\alpha} \|f\|_{L^{\infty}( \overline{Q_{r\rho}^+(Y)})} \right) d [(x' , x_n^{\frac{2-\gamma}{2}}, t),(y',y_n^{\frac{2-\gamma}{2}},\tau)]^{\alpha} \\
	 & \quad \leq C \left( \|u\|_{L^{\infty}( Q_1^+ )} + \|f\|_{L^{\infty}( Q_1^+)}  \right)  s[X,Y]^{\alpha} 
\end{align*}
for all $X \in \overline{Q_{r/2}^+(Y)}$ which is extensible throughout $\overline{Q_1^+}$ by using \eqref{bd_hol} and \eqref{est_v_hol}.
\end{proof}
\begin{lemma} \label{holder_est}
Let $0 < \alpha < \min \{\frac{2}{ 2-\gamma } ,1\}$ and let $u\in C^2(Q^+_1)\cap C(\overline{Q^+_1})$ be a solution of 
\begin{equation} \label{eq:dirich_prob}
	\left\{ \begin{aligned}
		u_t &= Lu && \mbox{in } Q^+_1\\
		u &= g && \mbox{on } \partial_p Q^+_1
	\end{aligned}\right.
\end{equation}
satisfying $\|u\|_{L^{\infty}(Q^+_1)} \leq 1$ and $g =0$ on $\{X \in \partial_p Q_1^+ : x_n=0\}$. If $ \|g\|_{C_s^{\alpha} (\overline{Q^+_1})} \leq 1$,  then
$$|u(X)| \leq C x_n^{\frac{2-\gamma}{2} \alpha} \quad \mbox{for all } X \in  \overline{Q_1^+}, $$
where $C>0$ is a universal constant.
\end{lemma}

\begin{proof}
Let  $K=  \left\{ \frac{\lambda}{2\Lambda} \left(1 -\frac{(2-\gamma) \alpha}{2} \right)\right\}^{\frac{2}{2-\gamma}} 
 \in (0,1)$.
Since $\|g\|_{C_s^{\alpha} (\overline{Q^+_1})} \leq 1$ and $\|u\|_{L^{\infty} (Q^+_1)} \leq 1$, there exists a constant $C>0$ such that $|g(X)| \leq C  x_n^{\frac{2-\gamma}{2} \alpha}$ for all $X \in \partial_p Q^+_1$ and $|u(X)| \leq C K^{\frac{2-\gamma}{2} \alpha}$ for all $X=(x',K,t) \in  Q^+_1$. Consider now the function $v$ defined in $\Omega =\{ X \in Q_1^+ : 0< x_n < K\} $ by $v(X) = u(X) -  C x_n^{\frac{2-\gamma}{2} \alpha}$. Then $v \leq 0$ on $\partial_p \Omega$ and 
\begin{equation*}
\begin{aligned}
v_t - L v & = C \left\{ \frac{(2-\gamma)\alpha}{2} \left( \frac{(2-\gamma)\alpha}{2}  - 1\right) a^{nn} x_n^{\frac{2-\gamma}{2}(\alpha-2)}  +  \frac{(2-\gamma)\alpha}{2}  b^n  x_n^{\frac{2-\gamma}{2}(\alpha-1)} + c x_n^{\frac{2-\gamma}{2} \alpha } \right \} \\
& \leq \frac{C(2-\gamma)\alpha}{2}  \left\{  \left( \frac{(2-\gamma)\alpha}{2}  - 1\right) \lambda + \Lambda x_n^{\frac{2-\gamma}{2}}  \right \} x_n^{\frac{2-\gamma}{2}(\alpha-2)}  \\
&\leq \frac{C(2-\gamma) \alpha \lambda}{4} \left( \frac{(2-\gamma)\alpha}{2}  - 1\right)  x_n^{\frac{2-\gamma}{2}(\alpha-2)}  \\
&< 0.
\end{aligned}
\end{equation*}
By \Cref{weak_max2}, we have $u(X) \leq  C x_n^{\frac{2-\gamma}{2} \alpha}$ for all $ X  \in \overline{ \Omega}$. Considering $v(x,t)=-u(x,t) + C x_n^{\frac{2-\gamma}{2} \alpha} $ similarly, we have lower bound of $u$.
\end{proof}
\begin{corollary} \label{gb_holder}
Let $0 < \alpha < \min \{\frac{2}{2-\gamma}, 1 \}$. Then for each $a^{ij},b^{i},c \in C_{\textnormal{loc}}^{\alpha}(Q_1^+)$, and $g \in C_s^{\alpha}(\overline{Q_1^+})$ satisfying $g=0$ on $\{X \in \partial_p Q_1^+ : x_n=0\}$, there exists a unique solution $u\in C^2(Q^+_1)\cap C_s^{\alpha}(\overline{Q^+_1})$ of \eqref{eq:dirich_prob} and 
$$\|u\|_{C_s^{\alpha}(\overline{Q^+_1})} \leq C  \|g\|_{C_s^{\alpha}(\overline{Q_1^+})} ,$$
where $C>0$ is a universal constant.
\end{corollary}
%
%
\section{Regularity of solutions for equations with constant coefficients} \label{sec:reg_cc}
Before treating \eqref{eq:main} with variable coefficients, we establish a necessary preliminary result for equations with constant coefficients given by
\begin{equation} \label{eq:cst_coeff}
u_t  = L_0 u + f,
\end{equation}
where the operator $L_0$ is given as
$$L_0  = A^{i'j'} D_{i'j'} + 2 x_n^{\gamma/2} A^{i'n} D_{i'n} + x_n^{\gamma} A^{nn} D_{nn} + B^{i'} D_{i'} + x_n^{\gamma/2} B^n D_{n} + C^- ,$$ 
with the constant symmetric matrix $(A^{ij})$ satisfies the following condition
\begin{equation*} 
	\lambda |\xi|^2\leq A^{ij} \xi_i\xi_j\leq \Lambda |\xi|^2\quad\mbox{for any }  \xi \in \mathbb R^n
\end{equation*}
and the constants $B^1,\cdots, B^n, C^-$ satisfy
\begin{equation*}
	\sum_{i=1}^n |B^i| +|C^- |  \leq \Lambda  \qquad \text{and } \qquad C^- < 0.
\end{equation*}
\begin{theorem} \label{thm:ccf_eq}
Let $k \in \mathbb{N}_{0}$, $0 < \alpha < 1$ with $ k+1+\alpha-\frac{2}{2-\gamma}  \not\in \mathbb{N}_0 $, and assume $ f \in C^{k,\alpha}_s(\overline{Q_1^+})$. Suppose $u \in C^2(Q_1^+) \cap C(\overline{Q_1^+})$ is a solution of \eqref{eq:cst_coeff} satisfying $u=0$ on $\{X \in \partial_p  Q_1^+ : x_n=0\}$. Then $u\in C^{k,2+\alpha}_s(\overline{Q^+_{1/2}})$ and 
\begin{equation*}
	\|u\|_{C^{k,2+\alpha}_s(\overline{Q^+_{1/2}})}\leq C \Big(\|u\|_{L^\infty(Q^+_1)}+\|f\|_{C^{k,\alpha}_s(\overline{Q^+_1})}\Big),
\end{equation*}
where $C$ is a universal constant.
\end{theorem}
\begin{lemma}  \label{bdd_uk}
Let $0 < \alpha < \min \{\frac{2}{2-\gamma}, 1 \}$ and let $u \in C^2(Q^+_1) \cap C(\overline{Q^+_1})$ be a solution of \eqref{eq:cst_coeff} satisfying $\|u\|_{L^{\infty}(Q^+_1)}\leq 1$ and $u =0$ on $\{X \in \partial_p Q_1^+ : x_n=0\}$. If 
$$ \| f \|_{L^{\infty}(Q^+_1)} \leq 1 \qquad \mbox{and} \qquad \| f_i \|_{L^{\infty}(Q^+_1)} \leq 1 ~ (i \ne n),$$
then 
$$\|u_i\|_{C_s^{\alpha}(\overline{Q_{1/2}^+})}  \leq C , $$
where $C>0$ is a universal constant.
\end{lemma}

\begin{proof}
It is sufficient to prove only the H\"older continuity of $u_1$ in $\overline{Q_{1/2}^+}$. Consider $\tilde{u}(x_1) $ as a function whose variables are fixed except for the $x_1$-coordinate in the function $u$. By \Cref{bdry_holder}, for $h \in \mathbb{R}$ with $|h| \ll 1 $, we have  $\|u\|_{C_s^{\alpha}(\overline{Q^+_{1-|h|}})} \leq C  (\|u\|_{L^{\infty}(Q_1^+)} + \|f\|_{L^{\infty}(Q_1^+)} )$. This implies that $\tilde{u} \in C^{\alpha} (I_h)$ and the function
\begin{equation*}
v^{\alpha}(x,t) = \frac{u(x_1 + h,x_2,\cdots, x_n, t) - u(x,t)}{|h|^{\alpha}} 
\end{equation*}
is a bounded solution of  $v_t^{\alpha} = L_0 v^{\alpha} + f^{\alpha}$, where 
$$f^{\alpha}(x,t) = \frac{f(x_1 + h,x_2,\cdots, x_n, t) - f(x,t)}{|h|^{\alpha}} 
\qquad \text{and}  \qquad 
I_h = \begin{cases}
	[-1, 1- h] & (h > 0) \\
	[-1-h, 1] & (h < 0) .
\end{cases}$$
We can apply \Cref{bdry_holder} to $v^{\alpha}$ again,
\begin{align*}
	\|v^{\alpha}\|_{C_s^{\alpha}(\overline{Q^+_{1-2|h|}})} &\leq C \left(\| v^{\alpha} \|_{L^{\infty}(Q_{1-|h|}^+)} + \|f^{\alpha}\|_{L^{\infty}(Q_{1-|h|}^+)}\right) \\
	& \leq  C \left(\| u \|_{C_s^{\alpha}(\overline{ Q_{1-|h|}^+})} + \|f_1\|_{L^{\infty}(Q_1^+)} \right) \\
	& \leq  C \left( \|u\|_{L^{\infty}(Q_1^+)} + \|f \|_{L^{\infty}(Q_1^+)}  + \|f_1\|_{L^{\infty}(Q_1^+)} \right). 
\end{align*}
This implies that $\tilde{u} \in C^{2\alpha}( I_{2h} )$. We can repeat this process, we have $\tilde{u} \in C^{0,1}(I_{kh})$ for some $k \in \mathbb{N}$. Finally, $v^1$ is also a bounded solution of $v_t^{1} = L_0 v^{1} + f^{1}$ and $v^1$ is the difference quotient of $u$ for $x_1$-direction, we conclude that $\tilde{u} \in C^{1,\alpha}(I_{(k+1)h})$ and 
$$\| u_1 \|_{C_s^{\alpha}(\overline{Q^+_{1/2}})} \leq C \left( \|u\|_{L^{\infty}(Q_1^+)} + \|f \|_{L^{\infty}(Q_1^+)}  + \|f_1\|_{L^{\infty}(Q_1^+)} \right).$$
\end{proof}
Since $v =u_i \, (i  \ne n)$ satisfies $v_t = L_0 v + f_i$, we can apply \Cref{bdd_uk} again to obtain the H\"older continuity of $u_{ij}$ on $ \overline{Q^+_{1/2}} $.
\begin{corollary}  \label{bdd_uij}
Let $0 < \alpha < \min \{\frac{2}{2-\gamma}, 1 \}$ and let $u \in C^2 (Q^+_1) \cap C(\overline{Q^+_1})$ be a solution of \eqref{eq:cst_coeff} satisfying $\|u\|_{L^{\infty}(Q^+_1)}\leq 1$ and $u =0$ on $\{X \in \partial_p Q_1^+ : x_n=0\}$. If for each $i,j=1,2,\cdots,n-1$,
$$ \| f \|_{L^{\infty}(Q^+_1)} \leq 1, \qquad  \| f_{i} \|_{L^{\infty}(Q^+_1)} \leq 1, \qquad \text{and} \qquad \| f_{ij} \|_{L^{\infty}(Q^+_1)} \leq 1 ,$$
then 
$$\|u_{ij}\|_{C_s^{\alpha}(\overline{ Q^+_{1/2} })}  \leq C , $$
where $C>0$ is a universal constant.
\end{corollary}
\begin{lemma}  \label{bdd_ut}
Let $0 < \alpha < \min \{\frac{2}{2-\gamma}, 1 \}$ and let  $u \in C^2 (Q^+_1) \cap C(\overline{Q^+_1})$ be a solution of \eqref{eq:cst_coeff} satisfying $\|u\|_{L^{\infty}(Q^+_1)}\leq 1$ and $u =0$ on $\{X \in \partial_p Q_1^+ : x_n=0\}$. If 
$$ \| f \|_{L^{\infty}(Q^+_1)} \leq 1 \qquad \text{and} \qquad \| f_t \|_{L^{\infty}(Q^+_1)} \leq 1,$$
then 
$$\|u_t\|_{C_s^{\alpha}(Q^+_{1/2})}  \leq C , $$
where $C>0$ is a universal constant.
\end{lemma}

\begin{proof}
Considering the function
\begin{equation*}
w^{\alpha}(x,t) = \frac{u(x, t+h) - u(x,t)}{|h|^{\alpha/2}}, 
\end{equation*}
we know that the proof is exactly the same as \Cref{bdd_uk}.
\end{proof}
\begin{lemma}  \label{bdry_c1a}
Let $u \in C^{\infty}(Q^+_1) \cap C(\overline{Q^+_1})$ be a solution of \eqref{eq:cst_coeff} satisfying $\|u\|_{L^{\infty}(Q^+_1)}\leq 1$ and $u =0$ on $\{X \in \partial_p Q_1^+ : x_n=0\}$. If for each $i,j,k=1,2,\cdots,n-1$,
$$\begin{aligned}
\| f \|_{L^{\infty}(Q^+_1)} &\leq 1, \\
\| f_{ijk} \|_{L^{\infty}(Q^+_1)} &\leq 1,
\end{aligned}\qquad 
\begin{aligned}
\| f_i \|_{L^{\infty}(Q^+_1)} &\leq 1, \\
\| f_t \|_{L^{\infty}(Q^+_1)} &\leq 1,
\end{aligned} \qquad 
\begin{aligned}
\| f_{ij} \|_{L^{\infty}(Q^+_1)} &\leq 1, \\
\| f_{kt} \|_{L^{\infty}(Q^+_1)} &\leq 1,
\end{aligned}
$$
and there exist constants $\theta \ge \gamma/2$ and $M  > 0$ such that
$$\left\{\begin{aligned}
|f(X)| &\leq M x_n^{\theta} \quad \mbox{for all } X \in \overline{Q^+_1} \\
|f_k(X)| &\leq M x_n^{\theta} \quad \mbox{for all } X \in \overline{Q^+_1}, \\
\end{aligned}\right.$$
then $u_n$ is well-defined on $\overline{Q^+_{1/2}}$ and
$$|u_n(x',x_n,t) - u_n(x',y_n,t)  |  \leq \left\{\begin{aligned} 
& C|x_n - y_n|^{1-\gamma/2}  && (0< \gamma \leq 1) \\
& C|x_n - y_n|  && (\gamma \leq 0)
\end{aligned}\right.$$
for all $(x',x_n,t),(x',y_n,t) \in \overline{Q^+_{1/2}}$, where $C>0$ is a constant depending only on $n$, $\lambda$, $\Lambda$, $\gamma$, $\theta$, and $M$.
\end{lemma}

\begin{proof}
Differentiate both sides of \eqref{eq:cst_coeff} with respect to $x_k\,(k \ne n)$, then 
\begin{equation}
\left\{\begin{aligned}\label{eq:uk}
v_t & = L_0  v + f_k  && \mbox{in } Q^+_1 \\
v & = 0 && \mbox{on } \{X \in \partial_p  Q_1^+ : x_n=0\},
\end{aligned}\right.
\end{equation}
where $v = u_k$. By \Cref{bdd_uk}, $v$ is bounded solution of \eqref{eq:uk}, we can apply \Cref{lip_est} to $v$, we have
\begin{equation*}
	|u_k(X)| = |v(X)| \leq  
	\begin{cases} 
		C x_n & (\gamma <1 ) \\ 
		- C x_n \log x_n  & (\gamma = 1) 
	\end{cases}
\end{equation*}
for all  $X \in \overline{Q^+_{1/2}}$. We can obtain the same inequality not only for $u_k$ but also for $u_t$ and $u_{ij} \,(i,j=1,2,\cdots,n-1)$. An integration by parts yields
\begin{align} 
\left| \int_0^z x_n^{-\gamma/2} u_{i'n}(x,t) \,dx_n \right| &= \left| \bigg[  x_n^{-\gamma/2} u_{i'}(x,t) \bigg]_{x_n=0}^{x_n=z} + \frac{\gamma}{2} \int_0^z x_n^{-\gamma/2-1} u_{i'}(x,t) \, dx_n \right|  \nonumber \\
&\leq \left| z^{-\gamma/2} u_{i'}(x',z,t) \right| + \frac{|\gamma|}{2} \int_0^z | x_n^{-\gamma/2-1} u_{i'} (x,t) | \, dx_n \label{ieq1:u_in} \\
&\leq  \left\{\begin{aligned} 
& C z^{1-\gamma/2}  && (\gamma <1 ) \\ 
& -C \sqrt{z} \log z  && (\gamma = 1) 
\end{aligned}\right. \label{ieq2:u_in}
\end{align}
for all $(x',z,t) \in \overline{Q^+_{1/2}}$. Similarly, we also have
\begin{align} 
\left| \int_0^z x_n^{-\gamma/2} u_n(x,t) \,dx_n \right| \leq  \left\{\begin{aligned} 
& C z^{1-\gamma/2}  && (\gamma <1 ) \\ 
& -C \sqrt{z} \log z  && (\gamma = 1) 
\end{aligned}\right. \label{ieq2:u_n}
\end{align}
for all $(x',z,t) \in \overline{Q^+_{1/2}}$. Combining \eqref{eq:cst_coeff}, \eqref{ieq2:u_in}, and \eqref{ieq2:u_n} gives
\begin{align}
	\left| \int_0^z u_{nn}\, dx_n \right|  &= \left| \int_0^z \frac{u_t   - A^{i'j'} u_{i'j'} - 2 x_n^{\gamma/2} A^{i'n} u_{i'n} -  B^{i'} u_{i'} - x_n^{\gamma/2} B^n u_{n} - C^- u - f}{ x_n^{\gamma} A^{nn} }  \,dx_n \right|  \label{ieq:int_u_nn}   \\
	& \leq  \left\{\begin{aligned} 
		& C z^{1-\gamma/2}  && (\gamma <1 ) \\ 
		& -C \sqrt{z} \log z && (\gamma = 1) 
	\end{aligned}\right. \nonumber 
\end{align}
for all $(x',z,t) \in \overline{Q^+_{1/2}}$ and hence for $\varepsilon < 2^{-\frac{2}{2-\gamma}}$,
\begin{align*}
|u_n(x',\varepsilon, t)| 
&= \left| u_n(x',2^{-\frac{2}{2-\gamma}},t) - \int_{0}^{2^{-\frac{2}{2-\gamma}}} u_{nn}(x,t)\, dx_n +  \int_{0}^{\varepsilon} u_{nn}(x,t)\, dx_n \right| \\
&\leq \left\{\begin{aligned} 
& |u_n(x',2^{-\frac{2}{2-\gamma}},t) | + C (2^{-1} + \varepsilon^{1-\gamma/2})  && (\gamma <1 ) \\ 
&  |u_n(x',1/4,t) |  + C ( \log 2  - \sqrt{\varepsilon} \log \varepsilon )&& (\gamma = 1) .
\end{aligned}\right. 
\end{align*}
By interior gradient estimates for uniformly parabolic equations, $u_n$ is bounded by a universal constant on $\{x_n = 2^{-\frac{2}{2-\gamma}}\}$. Therefore, we have $|u_n| \leq C$ in  $Q^+_{1/2}.$
This implies that even though $\gamma = 1$, we have the following Lipschitz estimate
\begin{equation} \label{lip_est_all} 
|u(x,t)| \leq C x_n \quad \mbox{on }  \overline{Q^+_{1/2}}.
\end{equation}
Furthermore $v=u_k \, (k\neq n)$ satisfies assumptions of \Cref{bdry_c1a} again, we have $|u_{kn}| \leq C $ in $Q^+_{1/2}.$ Thus, \eqref{ieq:int_u_nn} gives
\begin{align}\label{eq:un_holder}
	| u_{n}(x',x_n,t) -  u_{n}(x',y_n,t) | &= \left| \int_{y_n}^{x_n}  u_{nn}(x',z,t) \,dz \right| \nonumber \\
	&\leq C \left( \left| \int_{y_n}^{x_n} z^{1-\gamma} \,dz \right| + \left| \int_{y_n}^{x_n} z^{-\gamma/2} \,dz \right
| +  \left| \int_{y_n}^{x_n} z^{\theta-\gamma} \,dz \right
| \right) \nonumber \\
	& =  C( |x_n^{2-\gamma}- y_n^{2-\gamma}| + |x_n^{1-\gamma/2}- y_n^{1-\gamma/2}| +  |x_n^{\theta -\gamma + 1}- y_n^{\theta-\gamma+1}|) \\
	& \leq \begin{cases} 
		C |x_n - y_n|^{1-\gamma/2} & (0 < \gamma \leq 1 ) \\ 
		C |x_n - y_n|   & (\gamma \leq 0) 
	\end{cases} \label{eq:un_hol}
\end{align}
for all $x_n, y_n \in (0,2^{-\frac{2}{2-\gamma}})$. 

In order to extend \eqref{eq:un_hol} to $\overline{Q_{1/2}^+}$, we define the function $u_n^*:[0,2^{-\frac{2}{2-\gamma}}] \to \mathbb{R}$ as in \cite{McS34}:
$$u_n^*(x_n) = \begin{cases} 
\displaystyle \inf_{0< y_n < 2^{-\frac{2}{2-\gamma}}} \{ u_n (x',y_n,t) + C |x_n - y_n|^{1-\gamma/2} \} & (0 <\gamma \leq 1) \\
\displaystyle \inf_{0< y_n < 2^{-\frac{2}{2-\gamma}}} \{ u_n (x',y_n,t) + C |x_n - y_n|  \}   & (\gamma \leq 0).
\end{cases}$$
Then, $u_n^*$ is uniformly continuous on $[0,2^{-\frac{2}{2-\gamma}}]$ such that 
$$|u_n^*(x_n) - u_n^*(y_n)| \leq \begin{cases} 
C |x_n - y_n|^{1-\gamma/2} & (0 < \gamma \leq 1)  \\
C |x_n - y_n| & (\gamma \leq 0) 
\end{cases}$$
for all $x_n,y_n \in [0,2^{-\frac{2}{2-\gamma}}]$ and $u_n^*(x_n) = u_n(x',x_n,t) $ for all $x_n \in (0,2^{-\frac{2}{2-\gamma}})$. Therefore, the value $u_n(x',0,t)$ can be defined as $u_n^*(0)$ and we conclude that 
$$|u_n(x',x_n,t) - u_n(x',y_n,t)  |  \leq \left\{\begin{aligned} 
& C|x_n - y_n|^{1-\gamma/2}  && (0< \gamma \leq 1) \\
& C|x_n - y_n|  && (\gamma \leq 0)
\end{aligned}\right.$$
for all $(x',x_n,t),(x',y_n,t) \in \overline{Q^+_{1/2}}$.
\end{proof}

\begin{remark} \label{rmk:bdry_hi}
In \Cref{bdry_c1a}, we can verify that the proof can continue for $u_k\,(k \ne n)$, $u_t$, or higher order differentiation of $u$ except for $x_n$-direction. Therefore, on the same assumption of \Cref{bdry_c1a} for forcing term $D_{(x',t)}^{\beta} f(X) $ for each $ \beta \in \mathbb{N}_{0}^n$, the partial derivative $D_{(x',t)}^{\beta} u_n(X)$ is well-defined on $\overline{Q^+_{1/2}}$ and  
\begin{equation*}
 |D_{(x',t)}^{\beta} u(X)| \leq Cx_n \qquad \text{and} \qquad |D_{(x',t)}^{\beta} u_n(X)| \leq C
\end{equation*}
for all $X \in \overline{Q^+_{1/2}}$, where $C>0$ is a constant depending only on $n$, $\lambda$, $\Lambda$, $\gamma$, $\theta$, and $M$. 
\end{remark}
\begin{lemma} \label{lem:ccf_hi_zf}
Let $u \in C^{\infty}(Q^+_1) \cap C(\overline{Q^+_1})$ be a solution of \eqref{eq:cst_coeff} with $f=0$. If $\|u\|_{L^\infty(Q^+_1)}\leq 1$ and $u =0$ on $\{X \in \partial_p  Q_1^+ :x_n=0\}$, then for each $N \in \mathbb{N}_0$, there exists an $s$-polynomial $p$ of degree $\big( N +\frac{2}{2-\gamma}\big) $ at $O$ such that
\begin{equation*}
	|(p_t -L_0 p) (X)|\leq C x_n^{1+\frac{2-\gamma}{2}(N-1)} \quad \mbox{for all } X \in \overline{Q^+_{1/2}}
\end{equation*}
and
\begin{equation*}
	|u(X)-p(X)|\leq Cs[X,O]^{N+1+\frac{2}{2-\gamma}} \quad \mbox{for all } X \in \overline{Q^+_1},
\end{equation*}
where $C>0$  is a constant depending only on $n$, $\lambda$, $\Lambda$, $\gamma$, and $N$.
\end{lemma}

\begin{proof}
By \Cref{bdry_c1a}, the function $U^{0}(x',t) \coloneqq u_n(x',0,t)$ is well-defined. Furthermore, by \Cref{rmk:bdry_hi}, $U^{0}$ is infinitely differentiable and hence we have $D_{(x',t)}^{\beta} U^{0} (x',t)= D_{(x',t)}^{\beta} u_{n}(x',0,t) $ for all $\beta \in \mathbb{N}_{0}^n.$ For each $N \in \mathbb{N}_0$, consider a function $v^N$ of the form
\begin{equation*}
v^{N}(X) = \sum_{i=0}^{N} U^{i}(x',t) x_n^{1+\frac{2 - \gamma}{2}i}.
\end{equation*}
Then, the function $ (v^N_t - L_0 v^N)$ is expressed as follows:   
\begin{equation} \label{mv_L0v}
	(v^N_t - L_0 v^N) (X) =  \sum_{i=1}^{N+2} g^{i} (x',t) x_n^{1+ \frac{2-\gamma}{2}(i-2)} 
\end{equation}
where each $g^{i}$ is an unknown smooth function. For arbitrary fixed positive integer $l \leq N$, the function $g^l$ in $( \partial_t - L_0) v^N$ is the sum of the coefficient functions of 
\begin{equation} \label{ith_mono}		 
	x_n^{1+\frac{2-\gamma}{2} (l-2)}
\end{equation}
in the functions obtained by expanding the following operator
\begin{equation} \label{op_exp}
	(v^N_t - L_0 v^N) (X) =  \sum_{i=0}^N  (\partial_t - L_0) \left(U^{i} (x',t) x_n^{1+\frac{2-\gamma}{2} i} \right).
\end{equation}
If we choose $U^1$ such that $g^1 \equiv 0$, then 
\begin{equation*}
	U^1(x',t) = -\frac{4}{(2-\gamma)(4-\gamma)A^{nn}} \big(2A^{i'n} U_{i'}^0(x',t) + B^n U^0(x',t) \big).
\end{equation*}
In the expansion of \eqref{op_exp}, the coefficient function for the monomial of the form \eqref{ith_mono} can be obtained as follows: 
\begin{align*}
	&-  x_n^{\gamma} A^{nn} D_n^2 \left( U^l  x_n^{1+\frac{2-\gamma}{2} l}  \right)  -  x_n^{\gamma/2} (2 A^{i'n} D_{i'n} + B^n D_n) \left( U^{l-1}  x_n^{1+\frac{2-\gamma}{2} (l-1)}  \right)\\
	&\qquad\qquad\qquad\qquad\qquad + (\partial_t - A^{i'j'} D_{i'j'} -B^{i'}D_{i'}- C^- )\left( U^{l-2}  x_n^{1+\frac{2-\gamma}{2} (l-2)}  \right)  \\
	& \quad= -\bigg[  \frac{(2-\gamma)l }{2}   \Big( 1+ \frac{(2-\gamma)l}{2} \Big) A^{nn} U^{l} +  \Big( 1+ \frac{(2-\gamma)(l-1)}{2} \Big) (2A^{i'n} D_{i'}+ B^n)U^{l-1} \\
	&\qquad\qquad\qquad\qquad\qquad\qquad\quad - (\partial_t - A^{i'j'} D_{i'j'} -B^{i'}D_{i'}- C^- ) U^{l-2} \bigg] x_n^{1+\frac{2-\gamma}{2}(l-2)}.
\end{align*}
Now, taking $U^{l}$ such that $g^{l} \equiv 0$, we can determine $U^{l}$ inductively using the following formula:
\begin{equation} \label{form_ui} 
	U^{l} =  -\frac{\Big( 1+ \frac{(2-\gamma)(l-1)}{2} \Big) \left(2A^{i'n} D_{i'} + B^n  \right)U^{l-1}  -(\partial_t  - A^{i'j'} D_{i'j'} - B^{i'} D_{i'} - C^-) U^{l-2}}{ \frac{(2-\gamma)l }{2} \Big( 1+ \frac{(2-\gamma)l}{2} \Big) A^{nn} } 
\end{equation}
Since $l$ was arbitrary, we found all $U^i$. 

Using the recurrence relation \eqref{form_ui}, for any $N \in \mathbb{N}_0$, we can construct the function $v^N$ such that all monomials of the form
$$ g^{i}(x',t)x_n^{1+\frac{2-\gamma}{2}(i-2)} \quad  ( i \leq N  ) $$ 
are deleted in $( v^N_t - L_0 v^N)  $. Thus, we have
\begin{equation*}
	(v_t^N - L_0 v^N) (X) = g^{N+1} (x',t) x_n^{1+\frac{2-\gamma}{2}(N-1)} + g^{N+2} (x',t) x_n^{1+\frac{2-\gamma}{2}N} .
\end{equation*}
This implies that 
\begin{equation} \label{rem_N-th}
	|D_{(x',t)}^\beta (v_t^N - L_0 v^N)(X)| \leq Cx_n^{1 + \frac{2-\gamma}{2}(N-1)}  
\end{equation}
for all $X \in \overline{Q^+_{3/4}}$ and $\beta \in \mathbb{N}_{0}^{n}$, where $C$ depends on $n$, $\lambda$, $\Lambda$, $\gamma$, $\beta$, and $N$. 

We now claim, for any $N \in \mathbb{N}_0$, the following inequality holds:
\begin{equation} \label{approx_N-th}
| D_{(x',t)}^\beta (u - v^N )(X) | \leq  C x_n^{1+ \frac{2-\gamma}{2}(N+1)}
\end{equation}
for all $ X \in \overline{Q^+_{1/2}}$ and $\beta \in \mathbb{N}_{0}^n$, where $C>0$ is a constant depending only on $n$, $\lambda$, $\Lambda$, $\gamma$, $\beta$, and $N$.
The proof is by induction on $N$. Suppose first $N=0$ and let $w^0 = D_{(x',t)}^\beta (u-v^0) $, then $w^0$ is a solution of $w_t^0 = L_0 w^0 + f^0$ in $Q^+_{3/4}$, where $f^0=D_{(x',t)}^\beta ( L_0 v^0 - v_t^0) $. From \eqref{rem_N-th}, $f^{0}$ satisfy assumptions of \Cref{bdry_c1a}, thus inequality \eqref{eq:un_holder} yields
$$| w_n^0 (X)  |  = | w_n^0(X) - w_n^0(x',0,t) | \leq  Cx_n^{\frac{2-\gamma}{2}} \quad \mbox{for all } X \in \overline{Q^+_{1/2}} $$
and hence
$$|D_{(x',t)}^\beta (u - v^0) (X) |=|w^0(X) - w^0(x',0,t) | \leq  Cx_n^{1 + \frac{2 - \gamma}{2}} \quad \mbox{for all } X \in \overline{Q^+_{1/2}}. $$
Next, assume \eqref{approx_N-th} is valid for some nonnegative integer $(N-1)$. Let $w^N = D_{(x',t)}^\beta (u - v^N)$, then
\begin{align}
	|w^N(X)| 
	& \leq  | D_{(x',t)}^\beta u(X) -  D_{(x',t)}^\beta v^{N-1}(X) | + | D_{(x',t)}^\beta U^{N}(x',t) x_n^{1+\frac{2-\gamma}{2}N}| \leq C x_n^{1+\frac{2-\gamma}{2}N} \label{inq:wN}
\end{align}
for all $X\in \overline{Q^+_{1/2}}$ and $w_t^N = L_0 w^N + f^{N} $ in $Q^+_{3/4}$, where $f^{N} =D_{(x',t)}^\beta (L_0 v^{N}- v_t^{N}) $. Since $\beta$ was arbitrary, the partial derivatives of $w^N$ except in the $x_n$-direction satisfy \eqref{inq:wN} again. We know from \eqref{rem_N-th} that $f^N$ satisfies assumptions of \Cref{bdry_c1a}. Combining \eqref{ieq1:u_in} and \eqref{inq:wN} gives 
\begin{align*}
	\left| \int_0^{x_n} z^{-\gamma/2} w_{i'n}^N (x',z,t) \,dz \right| &\leq \left| x_n^{-\gamma/2} w_{i'}^N (x,t) \right| + \frac{|\gamma|}{2} \int_0^{x_n} | z^{-\gamma/2-1} w_{i'}^N (x',z,t)| \, dz   \leq C x_n^{\frac{2-\gamma}{2}(N+1)}
\end{align*}
for all $X \in \overline{Q^+_{1/2}}$. Similarly, we also have 
\begin{equation*}
 \left| \int_0^{x_n} z^{-\gamma/2} w_n^N(x',z,t) \,dz \right| \leq C x_n^{\frac{2-\gamma}{2}(N+1)}
\end{equation*}
for all $X \in \overline{Q^+_{1/2}}$. Combining \eqref{ieq:int_u_nn}, \eqref{rem_N-th}, and \eqref{inq:wN} gives 
\begin{align*}
	|w_n^N (X) |&= \left| \int_0^{x_n} w_{nn}^N (x',z,t)\, dz \right| \nonumber \\
	&   \leq C \left( \int_0^{x_n} z^{1+\frac{2-\gamma}{2}N - \gamma} \,dz + \left| \int_0^{x_n} z^{-\gamma/2} w_{i'n}^N (x',z,t) \,dz \right|  \right. \\
	&\qquad \qquad \left. +  \left| \int_0^{x_n} z^{-\gamma/2} w_n^N(x',z,t) \,dz \right| +  \left| \int_0^{x_n} z^{-\gamma} f^N(x',z,t) \,dz \right|  \right)  \\
	&   \leq C x_n^{\frac{2-\gamma}{2}(N+1)} 
\end{align*}
for all $X \in \overline{Q^+_{1/2}}$ and hence 
\begin{equation*} 
	|D_{(x',t)}^\beta(u - v^N) (X) | = |w^N(X) | \leq  C x_n^{1 + \frac{2-\gamma}{2}(N+1)} \quad \mbox{for all } X \in \overline{Q^+_{1/2}} . 
\end{equation*}
Thus, by mathmetical indution, \eqref{approx_N-th} holds for all $N \in \mathbb{N}_0$. 

Since $U^{l}\, (l=0,1,2,\cdots,N)$ is smooth, by Taylor theorem, for any $\beta =(\beta_1, \beta_2, \cdots, \beta_{n-1}) \in \mathbb{N}_{0}^{n-1}$ and $k \in \mathbb{N}_0$ with $|\beta| + 2k = N + 1 - l $, there exists Taylor polynomial $T^l(x',t)$ of degree $(N - l)$ such that 
$$|U^{l}(x',t) -T^l(x',t) | \leq C \sum_{|\beta| + 2k = N - l + 1 } |x_1|^{\beta_1} \cdots |x_{n-1}|^{\beta_{n-1}} |t|^{k}  $$
for all $|x_i| < 1/2 \, (i=1,2,\cdots,n-1)$ and $- 1/4 < t \leq 0$. This implies that 
\begin{align}
	\left| v^N(X) - \sum_{l=0}^{N} T^l(x',t) x_n^{1+\frac{2 - \gamma}{2}l} \right| & \leq C  \sum_{l=0}^{N}  \sum_{|\beta| + 2k = N - l + 1} |x_1|^{\beta_1} \cdots |x_{n-1}|^{\beta_{n-1}} |t|^{k} x_n^{1+\frac{2 - \gamma}{2}l} \nonumber \\
	& \leq C s[X,O]^{N + 1 +  \frac{2}{2-\gamma} }  \label{appox_taylor}
\end{align}
for all $X \in \overline{Q^+_{1/2}}$. Now put 
\begin{equation}  \label{appox_spol}
	p (X) =\sum_{l=0}^{N} T^l(x',t) x_n^{1+\frac{2 - \gamma}{2}l}
\end{equation}
which is an $s$-polynomial of degree $(N + \frac{2}{2-\gamma})$. Then, from \eqref{rem_N-th}, we know that
\begin{equation*}
	|(p_t -L_0 p) (X)| \leq Cx_n^{1 + \frac{2-\gamma}{2}(N-1)}\quad \mbox{for all } X \in \overline{Q^+_{1/2}}.
\end{equation*}
Combining  \eqref{approx_N-th} and \eqref{appox_taylor} gives 
\begin{align*}
	|u(X) - p (X)| & \leq |u(X)- v^{N}(X) | + |v^{N}(X) - p (X)| \\
	& \leq C x_n^{1 + \frac{2-\gamma}{2}(N+1)} + C s[X,O]^{N+1+\frac{2}{2-\gamma}}   \\
	& \leq C s[X,O]^{N+1+\frac{2}{2-\gamma}}
\end{align*}
for all $X\in \overline{Q^+_{1/2}}$ which is extensible throughout $Q_1^+$.
\end{proof}
\begin{lemma} \label{calc_TBT_L0}
Let $k \in \mathbb{N}_{0}$, $0<\alpha<1$ and assume that the function $f \in C_s^{k,\alpha}(\overline{Q_1^+})$ is of the form 
\begin{equation*}  
f (X) =\begin{dcases}
		\sum_{ \substack { \frac{2i}{2-\gamma} + j <  k+\alpha \\ i,j \in \mathbb{N}_0 }} \tilde{f}^{ij}(x',t) x_n^{i+\frac{2-\gamma}{2}j} & (  \gamma < 1) \\
		\sum_{ \substack{ i \leq j < k+\alpha \\ i,j \in \mathbb{N}_0 } } \tilde{f}^{ij}(x',t) (\log x_n)^i x_n^{j/2} & (  \gamma = 1) ,
	\end{dcases}
\end{equation*}
where each $\tilde{f}^{ij}$ is a smooth function for $(x',t)$. Then there is a function $h  \in C_s^{k+2,\alpha}(\overline{Q_1^+})$ of the form
\begin{equation} \label{expan_h_L0}
	h(X) = \begin{dcases}
		\sum_{ \substack {\frac{2i}{2-\gamma} + j < k+\alpha \\ i,j \in \mathbb{N}_0}} \tilde{h}^{ij}(x',t) x_n^{i+\frac{2-\gamma}{2} (j+2)} & (  \gamma < 1) \\
		\sum_{ \substack{ i \leq j  < k+\alpha \\  i,j \in \mathbb{N}_0} } \tilde{h}^{ij}(x',t) (\log x_n)^{i+1}  x_n^{(j+2)/2} & (  \gamma = 1) 
	\end{dcases}
\end{equation}	
such that
\begin{equation*} 
	|(h_t - L_0 h - f) (X) | \leq C \|f\|_{C^{k,\alpha}_s(\overline{Q_1^+})}x_n^{\frac{2-\gamma}{2}(k+\alpha)}\quad \text{for all } X \in \overline{Q_1^+}
\end{equation*}
and
\begin{equation*} 
	\|h\|_{C^{k+2,\alpha}_s(\overline{Q_1^+})} \leq C \|f\|_{C^{k,\alpha}_s(\overline{Q_1^+})},
\end{equation*}
where each $\tilde{h}^{ij}$ is a smooth functions for $(x',t)$ and $C$ is a universal constant.
\end{lemma}
In fact, \Cref{calc_TBT_L0} holds true even when the coefficients are $s$-polynomial, and the general case will be proved in \Cref{calc_TBT}, so it is omitted here.
\begin{lemma} \label{lem:ccf_hi_nzf}
Let  $k \in \mathbb{N}_{0}$, $0 < \alpha < 1$ with $ k+1+\alpha-\frac{2}{2-\gamma}  \not\in \mathbb{N}_0 $, and assume $f \in C^{k,\alpha}_s(\overline{Q^+_1})$. Suppose $u \in C^2 (Q^+_1)\cap  C(\overline{Q^+_1})$ is a solution of \eqref{eq:cst_coeff} satisfying $u=0$ on $\{ X \in \partial_p Q^+_1 :x_n=0\}$. Then there exists an $s$-polynomial $p$ of degree $m$ corresponding to $\kappa \coloneqq k + 2 + \alpha$ at $O$ such that  
\begin{equation*}
|u(X)-p(X)| \leq C\left(\|u\|_{L^\infty(Q^+_1)}+\|f\|_{C^{k,\alpha}_s(\overline{Q^+_1})} \right) s[X,O]^{k+2+\alpha} \quad \mbox{for all } X \in \overline{Q^+_1} ,
\end{equation*}
where $C > 0$ is a universal constant.
\end{lemma}

\begin{proof}
By considering $u/(\|u\|_{L^\infty(Q^+_1)} + \varepsilon^{-1} \|f\|_{C^{k,\alpha}_s(\overline{Q^+_1})}) $  for $\varepsilon > 0$, we may assume that $\|u\|_{L^\infty(Q^+_1)}\leq1$ and $\|f\|_{C^{k,\alpha}_s(\overline{Q^+_1})} \leq \varepsilon$. Since $f  \in C^{k,\alpha}_s(\overline{Q^+_1})$, we also assume that 
\begin{align}
	|f(X) - F(X)| &\leq \varepsilon s[X,O]^{k+\alpha} \quad \text{for all } X \in \overline{Q_1^+} ,\label{approx_f_L0} 
\end{align}
where $F$ is an $s$-polynomials with degree $\tilde{m}$ corresponding to $(k+\alpha)$ at $O$. By \Cref{calc_TBT_L0},  there exists a function $h  \in C_s^{k+2,\alpha}(\overline{Q_1^+})$ of the form \eqref{expan_h_L0} such that
\begin{equation}  \label{inq:op_h_L0} 
	|(F + L_0 h - h_t) (X) | \leq C_* \varepsilon s[X,O]^{k+\alpha}\quad \text{for all } X \in \overline{Q_1^+}
\end{equation}
and  $ \|h\|_{C^{k+2,\alpha}_s(\overline{Q_1^+})} \leq C_* \varepsilon,$ where $C_*$ is a universal constant. 

Decompose $(u-h)$ into the sum of $v^1$ and $w^1$ such that
\begin{equation*}
\left\{\begin{aligned}
	v^1_t &= L_0 v^1 && \mbox{in } Q_1^+ \\
	v^1 &= u- h && \mbox{on }\partial_p Q_1^+
\end{aligned} \right.
\qquad \textnormal{and} \qquad 
\left\{\begin{aligned}
	w_t^1 &= L_0 w^1 + \tilde{f}^1 && \mbox{in } Q_1^+ \\
	w^1 & = 0 && \mbox{on }\partial_p Q_1^+
\end{aligned}\right.
\end{equation*}
where $\tilde{f}^1 = f + L_0 h - h_t$. Then, combining \eqref{approx_f_L0} and \eqref{inq:op_h_L0} leads us to the estimate
\begin{equation} \label{inq:tf<s_L0}
	|\tilde{f}^1(X)|   \leq  |f(X) -  F(X)| + |(F + L_0 h - h_t)(X)|  \leq (1+C_*) \varepsilon s[X,O]^{k+\alpha}  
\end{equation}
for all $X \in \overline{Q_1^+}$. 
By \Cref{apriori_bound1}, for fixed $r \in (0, 1)$, \eqref{inq:tf<s_L0} gives
\begin{equation} \label{w_1st_hcf_L0}
	|w^1(X)| \leq C \| \tilde{f}^1 \|_{L^{\infty}( Q_1^+ )} \leq (1+C_*)\varepsilon \quad \text{for all } X \in \overline{Q_r^+} 
\end{equation}
and hence, we have
\begin{align*}
	|u(X) - v^1(X) - h(X)| =  |w^1(X)| \leq  (1+C_*) \varepsilon  \quad \text{for all } X \in \overline{Q_r^+} . 
\end{align*}
Choose now $\varepsilon$ small enough such that  $(1+C_*) \varepsilon < r^{k+2+\alpha}$. Then, 
$$|u(X) - v^1(X) - h(X)| \leq  r^{k+2 + \alpha} \quad \text{for all }X \in \overline{Q_r^+}.$$
Let us now define
$$u^2(X) \coloneqq \frac{(u - v^1- h ) (rX) }{r^{k+2+\alpha}} \qquad \text{and} \qquad \tilde{f}^2(X) \coloneqq \frac{\tilde{f}^1(rX) }{r^{k+\alpha}} .$$
Then, $\|u^2\|_{L^{\infty}(Q_1^+ )} \leq 1$ and $u_t^2  =  L_0^2 u^2 + \tilde{f}^2 $ in $ Q_1^+$, where
\begin{align*}
	L_0^2 = A^{i'j'} D_{i'j'} + 2 x_n^{\gamma/2} A^{i'n} D_{i'n} + x_n^{\gamma} A^{nn} D_{nn}  + rB^{i'} D_{i'} + x_n^{\gamma/2} r B^n  D_n + r^2C^-   .
\end{align*}
Furthermore, \eqref{inq:tf<s_L0} leads us to the estimate
\begin{align*}
	|\tilde{f}^{2}(X)| 
& \leq r^{-(k+\alpha)} \left(  \left| f(rX) -F(rX) \right| + \left| (F + L_0 h - h_t) (rX) \right|  \right)\nonumber 
	 \leq (1+C_*)\varepsilon s[X,O]^{k+\alpha}   
\end{align*}
for all $X \in \overline{Q_1^+}$. That is, the same hypotheses as before are fulfilled. Repeating the same procedure, we decompose $u^2$ into the sum of $v^2$ and $w^2$ such that
\begin{equation*}
\left\{\begin{aligned}
	v^2_t &= L_0^2 v^2 && \mbox{in } Q_1^+ \\
	v^2 &= u^2 && \mbox{on }\partial_p Q_1^+,
\end{aligned} \right.
\qquad  \qquad 
\left\{\begin{aligned}
	w_t^2 &= L_0^2 w^2 + \tilde{f}^2 && \mbox{in } Q_1^+ \\
	w^2 & = 0 && \mbox{on }\partial_p Q_1^+,
\end{aligned}\right.
\end{equation*}
and  $|u^2(X)-v^2(X)| \leq r^{k+2+\alpha}$ for all $X \in \overline{Q_r^+}$. By substituting back, we have
\begin{equation*}
	|u(X) - h(X)  - v^1(X) - r^{k+2+\alpha} v^2 (r^{-1}X)  | \leq r^{2(k+2 + \alpha)} \quad \text{for all } X  \in \overline{Q_{r^2}^+} .
\end{equation*}
Continuing iteratively, for each integer $l \ge 3$, let us define the sequence of functions $\{u^l\}$ inductively as follows:
$$u^{l}(X) \coloneqq \frac{(u^{l-1} - v^{l-1}) (rX ) }{r^{k+2+\alpha}}  \qquad \text{and} \qquad \tilde{f}^l(X) \coloneqq \frac{\tilde{f}^{l-1}(rX) }{r^{k+\alpha}}  .$$
Then, $\|u^l\|_{L^{\infty}(Q_1^+ )} \leq 1$ and $u_t^l  =  L_0^l u^l + \tilde{f}^l $ in $Q_1^+$, where
\begin{align*}
	L_0^l  &= A^{i'j'} D_{i'j'} + 2 x_n^{\gamma/2} A^{i'n} D_{i'n} + x_n^{\gamma} A^{nn} D_{nn}  + r^{l-1} B^{i'} D_{i'} + x_n^{\gamma/2} r^{l-1} B^n D_{n} + r^{2(l-1)}C^-  .
\end{align*}
Furthermore, $|\tilde{f}^l(X)| \leq (1+C_*)\varepsilon s[X,O]^{k+\alpha} $ for all $X \in \overline{Q_1^+}$, and hence we decompose $u^l$ into the sum of $v^l$ and $w^l$ such that
\begin{equation*}
\left\{\begin{aligned}
	v^l_t &= L_0^l v^l && \mbox{in } Q_1^+ \\
	v^l &= u^l && \mbox{on }\partial_p Q_1^+,
\end{aligned} \right.
\qquad  \qquad 
\left\{\begin{aligned}
	w_t^l &= L_0^l w^l + \tilde{f}^l && \mbox{in } Q_1^+ \\
	w^l & = 0 && \mbox{on }\partial_p Q_1^+,
\end{aligned}\right.
\end{equation*}
and  $|u^l(X)-v^l(X)| \leq r^{k+2+\alpha} $ for all $X \in \overline{Q_r^+}$. By substituting back, we have
\begin{equation} \label{sbt_back_est}
	\left|u(X) - h(X)  -  \sum_{i=1}^l r^{(i-1)(k+2+\alpha)}  v^i (r^{-i+1}X) \right| \leq r^{l(k+2 + \alpha)} \quad \text{for all } X  \in \overline{Q_{r^l}^+}.
\end{equation}

Since $k+1+\alpha-\frac{2}{2-\gamma}  \not\in \mathbb{N}_0$, so we can choose $N \in \mathbb{N}_0$ such that
\begin{equation*}
k+1+\alpha -\frac{2}{2-\gamma} < N < k+2+\alpha -\frac{2}{2-\gamma}.
\end{equation*}
By \Cref{lem:ccf_hi_zf} and \eqref{appox_spol}, for each $l \in \mathbb{N}$, there exists an $s$-polynomial 
\begin{equation*}
	p^l(X) = \sum_{|\beta|+i+2j+\frac{2}{2-\gamma} < k+2+\alpha} A_l^{\beta ij} x'^{\beta} x_n^{1 + \frac{2-\gamma}{2} i} t^{j} 
\end{equation*}
of degree $(N+\frac{2}{2-\gamma})$ at $O$
such that
\begin{equation} \label{est_vl_pl}
	|v^l(X)-p^l(X)| \leq  Cs[X,O]^{N+1+\frac{2}{2-\gamma}} \quad \text{for all } X \in \overline{Q^+_1},
\end{equation}
where $C>0$ is a universal constant. Combining \eqref{sbt_back_est} and \eqref{est_vl_pl} leads us to the estimate
\begin{align}
	&\left|u(X) - h(X)  -  \sum_{i=1}^l r^{(i-1)(k+2+\alpha)}  p^i (r^{-i+1}X) \right| \nonumber \\
	&\quad \leq r^{l(k+2 + \alpha)} + C \sum_{i=1}^l r^{(i-1)(k+2+\alpha)} s[r^{-i+1}X,O]^{N+1+\frac{2}{2-\gamma}} \nonumber \\
	&\quad \leq r^{l(k+2 + \alpha)} + Cr^{l(N+1+\frac{2}{2-\gamma})} \sum_{i=1}^l r^{-(i-1)(N+\frac{2}{2-\gamma} -k-1-\alpha)} \nonumber \\
		&\quad = r^{l(k+2 + \alpha)} + \frac{C( r^{(l+1)(N+\frac{2}{2-\gamma}-k-1-\alpha)} -r^{N+\frac{2}{2-\gamma}-k-1-\alpha})}{r^{N+\frac{2}{2-\gamma}-k-1-\alpha}-1} r^{l(k+2+\alpha)} \nonumber \\
	&\quad \leq C r^{l(k+2+\alpha)} \label{est_u-h-pl}
\end{align}
for all $X  \in \overline{Q_{r^l}^+}$. Now put
\begin{equation*}
	P^l(X)= \sum_{i=1}^l r^{(i-1)(k+2+\alpha)} p^i (r^{-i+1}X)  .
\end{equation*}
Then, \eqref{est_u-h-pl} gives
\begin{align}
	|P^l(X)-P^{l-1}(X)|& \leq  \left|u(X) - h(X)  - P^l(X) \right| + \left|u(X) - h(X)  - P^{l-1}(X) \right| \nonumber \\
	&\leq  C r^{l(k+2 + \alpha)} \label{est_p-p}  
\end{align}
for all $X  \in \overline{Q_{r^l}^+}$. Consider the rescaling $s$-polynomial
\begin{align*}
	\tilde{P} (X) &\coloneqq P^l(r^l X) - P^{l-1}(r^lX)  \\
	&=r^{(l-1)(k+2+\alpha)}  \sum_{|\beta|+i+2j+\frac{2}{2-\gamma} < k+2+\alpha} r^{|\beta|+i+2j+\frac{2}{2-\gamma} } A_l^{\beta ij} x'^{\beta} x_n^{1 + \frac{2-\gamma}{2} i} t^{j} . \\
\end{align*}
From \eqref{est_p-p}, we know that $\|\tilde{P}\|_{L^{\infty}( Q_1^+ )} \leq C r^{l(k+2+\alpha)}$. Since the coefficients of $s$-polynomial $\tilde{P}$ on $Q_1^+$ are controlled by the $L^{\infty}$ norm, we have
\begin{equation*}
	|A_l^{\beta ij} | \leq \frac{C r^{ k+2 + \alpha }}{r^{ |\beta|+ i + 2j + \frac{2}{2-\gamma} }}
\end{equation*}
for all $\beta \in \mathbb{N}_0^{n-1}$ and $i,j \in \mathbb{N}_0$ with $|\beta|+ i + 2j + \frac{2}{2-\gamma} < k+2+\alpha$. It follows that $P^l$ converges uniformly to an $s$-polynomial  
$$P^{\infty}(X) =  \sum_{|\beta|+ i + 2j + \frac{2}{2-\gamma} < k+2+\alpha} B^{\beta ij} x'^{\beta} x_n^{1 + \frac{2-\gamma}{2} i} t^{j}
$$ 
of degree $(N+\frac{2}{2-\gamma})$ at $O$ such that  
\begin{equation*}
|u(X)-h(X)-P^{\infty}(X)| \leq C s[X,O]^{k+2+\alpha} \quad \mbox{for all } X \in \overline{Q^+_1}.
\end{equation*}
Finally, as in  \Cref{lem:ccf_hi_zf}, if $h$ is approximated with an $s$-polynomial, the desired $s$-polynomial $p$ can be obtained.
\end{proof}
%
%
\section{Generalized Schauder Theory} \label{sec:gst}
\subsection{$C^{2+\alpha}_s$-Regularity} 
In this section, we establish $C_s^{2+\alpha}$-regularity of solutions for \eqref{eq:main}. Since generalized Schauder theory approximates coefficients using $s$-polynomials rather than constants, boundary $C^{1,\alpha}$-regularity and higher regularity for equations with $s$-polynomials coefficients must first be obtained. To prove these, we need $C_s^{2+\alpha}$-regularity of solutions for \eqref{eq:main}.

For this section only, $L_0$ is considered the operator with the following constant coefficients
$$A^{ij}=a^{ij}(O), \quad B^i =  b^{i}(O), \quad C^-=c(O). $$ 
\begin{theorem} \label{thm:c2a_hc}
Let $0<\alpha< 1$ with $2+\alpha \notin \mathcal{D}$, and assume 
$$a^{ij}, \, b^i, \, c, \, f \in C^{\alpha}_s(\overline{Q_1^+}) \quad (i,j=1,2,\cdots,n).$$
Suppose $u \in C^2(Q_1^+) \cap C(\overline{Q_1^+})$ is a solution of \eqref{eq:main} satisfying $u=0$ on $\{X \in \partial_p  Q_1^+ : x_n=0\}$. Then $u\in C^{2+\alpha}_s(\overline{Q^+_{1/2}})$ and 
\begin{equation*}
	\|u\|_{C^{2+\alpha}_s(\overline{Q^+_{1/2}})}\leq C\left(\|u\|_{L^\infty(Q^+_1)}+\|f\|_{C^{\alpha}_s(\overline{Q^+_1})} \right),
\end{equation*}
where $C$ is a positive constant depending only on $n$, $\lambda$, $\Lambda$, $\gamma$, $\alpha$, $\|a^{ij}\|_{C^{\alpha}_s(\overline{Q^+_1})}$, $\|b^{i}\|_{C^{\alpha}_s(\overline{Q^+_1})}$, and $\|c\|_{C^{\alpha}_s(\overline{Q^+_1})}$.
\end{theorem}
\begin{lemma} [Approximation Lemma] \label{approx_lem}
Let $0 < \alpha < \min \{\frac{2}{2-\gamma}, 1 \}$, and assume $f \in L^{\infty} ( Q^+_1 ) $. Suppose $u \in C^2(Q_1^+) \cap C(\overline{Q^+_1})$ is a solution of \eqref{eq:main} satisfying $u=0$ on $\{X \in \partial_p  Q_1^+ : x_n=0\}$ with $\|u\|_{L^{\infty}(Q_1^+)} \leq 1$. If for any $\varepsilon > 0$,
$$\|a^{ij} - A^{ij}\|_{L^{\infty} (Q^+_1)} \leq  \varepsilon, \quad \|b^i - B^i\|_{L^{\infty} (Q^+_1)} \leq  \varepsilon, \quad \|c - C^-\|_{L^{\infty} (Q^+_1)} \leq  \varepsilon,$$
then there exists a solution $h \in  C^{\infty}(Q^+_{3/4}) \cap C^{\alpha}_s(\overline{Q^+_{3/4}})$ of
\begin{equation*} 
\left\{\begin{aligned}
h_t &= L_0 h  && \mbox{in } Q_{3/4}^+  \\
h &= u &&  \mbox{on } \partial_p Q_{3/4}^+ .
\end{aligned}\right.
\end{equation*}
such that 
$$\|u - h \|_{L^{\infty}(Q_{1/2}^+)} \leq C \left(\varepsilon^{\theta} + \|f\|_{L^{\infty} ( Q^+_1)} \right),$$
where $C>0$ and $\theta \in (0,1)$ are universal constants.
\end{lemma}

\begin{proof}
By \Cref{bdry_holder} and \Cref{gb_holder}, there exists a unique bounded solution $h \in C^2(Q^+_{3/4}) \cap C^{\alpha}_s(\overline{Q^+_{3/4}})$ of 
\begin{equation*} 
\left\{\begin{aligned}
h_t &= L_0 h  && \mbox{in } Q_{3/4}^+  \\
h &= u &&  \mbox{on } \partial_p Q_{3/4}^+ ,
\end{aligned}\right.
\end{equation*}
and 
$$\|h\|_{C_s^{\alpha}(\overline{Q^+_{3/4}})} \leq C \Big(\|h\|_{L^{\infty}(Q_{3/4}^+)} + \|u\|_{C_s^{\alpha}(\overline{Q_{3/4}^+})} \Big) \leq C \left(2 + \|f\|_{L^{\infty}(Q_1^+)} \right).$$
For $X \in \partial_p Q_{3/4 - \delta}^+$, we can choose $Y \in \partial_p Q_{3/4}^+$ such that $s[X,Y] = \delta$ and hence we have
$$|u(X)-h(X)| \leq \|u-h\|_{C^{\alpha}_s(\overline{Q^+_{3/4}})} s[X,Y]^{\alpha} \leq C \delta^{\alpha} \left(2 + \|f\|_{L^{\infty} ( Q^+_1 )} \right).$$
By \Cref{apriori_bound1} and \Cref{thm:ccf_eq}, we have for any $\delta \in (0,1/4)$
$$ \|h\|_{C_s^{2,\alpha}(\overline{Q_{3/4 - \delta}^+})} \leq C \delta^{ -2 - \alpha} \| h \|_{L^{\infty}(Q_{3/4}^+)} \leq  C \delta^{-2-\alpha} \| u \|_{L^{\infty}(Q_{3/4}^+)} \leq C \delta^{-2-\alpha}.$$
Let $v= u - h$, then $v_t = Lv + \tilde{f} $ in $Q_{3/4-\delta}^+$, where 
\begin{align*}
	\tilde{f}&=(a^{i'j'} - A^{i'j'}) h_{i'j'} + 2 x_n^{\gamma/2} (a^{i'n}- A^{i'n}) h_{i'n} + x_n^{\gamma} (a^{nn}- A^{nn}) h_{nn} \\
	& \qquad\qquad\qquad\quad + (b^{i'} - B^{i'}) h_{i'} + x_n^{\gamma/2} (b^n - B^n) h_{n} + (c-C^-) h + f.
\end{align*}
By \Cref{apriori_bound1} and \Cref{apriori_bound2} with scaling argument,
\begin{align*}
	\|u-h\|_{L^{\infty}(Q_{3/4-\delta}^+)}  &\leq 
	\begin{cases}
	C \left(\|u-h\|_{L^{\infty}(\partial_p Q_{3/4-\delta}^+)} + \|\tilde{f}\|_{L^{\infty}(Q_{3/4-\delta}^+)} \right) & (0<\gamma<1) \\
	C\left(\|u-h\|_{L^{\infty}(\partial_p Q_{3/4-\delta}^+)} + \|x_n^{-\gamma}\tilde{f}\|_{L^{\infty}(Q_{3/4-\delta}^+)} \right)& (\gamma <0)
	\end{cases}\\
	&\leq C\left( \delta^{\alpha} (2 + \|f\|_{L^{\infty} (Q^+_1)} ) + \varepsilon \delta^{-2-\alpha} + \|f\|_{L^{\infty} ( Q^+_1 )} \right) .
\end{align*}
Take $\delta = \varepsilon^{\frac{1}{2(2+\alpha)}}$ and $\theta = \frac{\alpha}{2(2+\alpha)}$, we conclude that
$$\|u-h\|_{L^{\infty}(Q_{3/4-\delta}^+)}  \leq C (\varepsilon^\theta + \|f\|_{L^{\infty} ( Q^+_1 )} ).$$
\end{proof}
\begin{lemma} \label{bdry_c2a_vc}
Let $0 < \alpha < \min \{\frac{2}{2-\gamma}, 1 \}$ with $2+\alpha \notin \mathcal{D}$, and assume 
$$a^{ij}, \, b^i, \, c, \, f\in C^{\alpha}_s(\overline{Q_1^+}) \quad (i,j=1,2,\cdots,n).$$
Suppose $u \in  C^2 ( Q_1^+ ) \cap C (\overline{Q_1^+})$ is a solution of \eqref{eq:main} satisfying $u=0$ on $\{X \in \partial_p  Q_1^+ : x_n=0\}$. Then there exists an $s$-polynomial $p$ of degree $m$ corresponding to $\kappa \coloneqq 2 + \alpha$ at $O$ such that
\begin{equation*}
|u(X) - p(X) |\leq C \left( \|u\|_{L^\infty(Q^+_1)}+\|f\|_{C^{\alpha}_s(\overline{Q^+_1})} \right) s[X,O]^{2 + \alpha} \quad \mbox{for all } X \in \overline{Q_1^+},
\end{equation*}
where $C$ is a positive constant depending only on $n$, $\lambda$, $\Lambda$, $\gamma$, $\alpha$, $\|a^{ij}\|_{C^{\alpha}_s(\overline{Q^+_1})}$, $\|b^{i}\|_{C^{\alpha}_s(\overline{Q^+_1})}$, and $\|c\|_{C^{\alpha}_s(\overline{Q^+_1})}$.
\end{lemma}

\begin{proof}
By considering $u/(\|u\|_{L^\infty(Q^+_1)} + \varepsilon^{-1} \|f\|_{C^{\alpha}_s(\overline{Q^+_1})} )$ for sufficiently small $\varepsilon >0$, we may assume that $\|u\|_{L^\infty(Q^+_1)}\leq1$ and $\|f\|_{C^{\alpha}_s(\overline{Q^+_1})} \leq \varepsilon$. By scaling we also assume that $[ a^{ij} ]_{C^{\alpha}_s(\overline{Q^+_1})} \leq \varepsilon$, $[ b^i ]_{C^{\alpha}_s(\overline{Q^+_1})} \leq \varepsilon$, $[ c ]_{C^{\alpha}_s(\overline{Q^+_1})} \leq \varepsilon,$ and by the H\"older continuity of coefficients, we have
\begin{equation} \label{hol_coeff}
	\left\{ \begin{aligned}
		|a^{ij}(X) - A^{ij}| & \leq [ a^{ij} ]_{C^{\alpha}_s(\overline{Q^+_1})} s[X,O]^{\alpha} \leq \varepsilon \\
		|b^{i}(X) - B^{i}| & \leq [ b^i ]_{C^{\alpha}_s(\overline{Q^+_1})} s[X,O]^{\alpha}  \leq \varepsilon \\
		|c(X) - C^-| & \leq [ c ]_{C^{\alpha}_s(\overline{Q^+_1})} s[X,O]^{\alpha}  \leq \varepsilon 
	\end{aligned} \right. 
\end{equation}
for all $X \in \overline{Q_1^+}$. By considering 
\begin{equation*}
	\tilde{u}(X) = u(X) + \begin{dcases} 
		\frac{f(O)}{(2-\gamma)(1-\gamma)}x_n^{2-\gamma} & (\gamma<1)\\
		f(O) x_n \log x_n & (\gamma=1),
	\end{dcases}
\end{equation*}
we may assume $f(O)=0$. By \Cref{approx_lem}, there exists a bounded solutions $h^1 \in C^{\infty}(Q_{3/4}^+) \cap C^{\alpha}_s(\overline{Q^+_{3/4}})$ of 
\begin{equation*} 
	\left\{\begin{aligned}
		h_t^1  &= L_0 h^1  && \mbox{in } Q_{3/4}^+  \\
		h^1  &= u &&  \mbox{on } \partial_p Q_{3/4}^+ 
	\end{aligned}\right.
\end{equation*}
such that
\begin{equation*}
\|u-h^1\|_{L^{\infty}(Q_{1/2}^+)}  \leq C \left(\varepsilon^\theta + \|f\|_{L^{\infty} ( Q^+_1 )} \right) \leq 2C \varepsilon^{\theta}.
\end{equation*}

By \Cref{thm:ccf_eq}, $h^1 \in C_s^{2+\alpha}(\overline{Q_{1/2}^+})$ and 
\begin{equation} \label{h1_c2a}
	\|h^1\|_{C_s^{2+\alpha}(\overline{Q_{1/2}^+})} \leq C, 
\end{equation}
where $C$ is a universal constant. Let $r \in (0,1/2)$ be a fixed number and choose now $\varepsilon$ small enough such that $2C\varepsilon^{\theta} <r^{2+\alpha}$. Then, 
\begin{equation} \label{norm_u2}
	\|u-h^1\|_{L^{\infty}(Q_r^+)} \leq r^{2+\alpha}.
\end{equation}
Let us now define
\begin{equation*}
	u^2(X) \coloneqq \frac{(u-h^1)(rX)}{r^{2+\alpha}} \qquad \text{and} \qquad \tilde{f}^2(X) \coloneqq \frac{(f+Lh^1-L_0 h^1)(rX)}{r^{\alpha}}.
\end{equation*}
Then, $\|u^2\|_{L^{\infty}(Q_1^+)}\leq 1$ and $u_t^2 = L^2 u^2 + \tilde{f}^2$ in $ Q_1^+$, where 
\begin{align*}
	L^2 &= a^{i'j'}(rX)D_{i'j'} + 2x_n^{\gamma/2} a^{i'n} (rX)D_{i'n} +x_n^{\gamma}a^{nn}(rX) D_{nn} \\
	&\qquad\qquad\qquad\qquad\qquad\qquad + rb^{i'}(rX) D_{i'} + r x_n^{\gamma/2}b^{n}(rX) D_n + r^2 c(rX).
\end{align*}
From \eqref{hol_coeff}, \eqref{h1_c2a}, and \eqref{norm_u2}, we can see that the same hypotheses as before are fulfilled. Replacing the decomposition of solutions with \Cref{approx_lem}, as in the proof of \Cref{lem:ccf_hi_nzf}, we can show that there exists an $s$-polynomial $p^{\infty}$ of degree $m$ corresponding to $(2+\alpha)$ at $O$ such that
$$|u(X) -p^{\infty}(X)| \leq Cs[X,O]^{2+\alpha} \quad \mbox{for all } X \in \overline{Q_{1}^+}.$$
\end{proof}
Satisfying the assumptions of \Cref{bdry_c2a_vc}, we know that $u \in C_s^{\tilde{\alpha}} (\overline{Q_{1/2}^+})$ for any $0 < \tilde{\alpha} <1$. Applying the approximation lemma again to $ \tilde{\alpha}$, we have $u \in C_s^{2+\tilde{\alpha}} (\overline{Q_{1/2}^+})$  for any $0 < \tilde{\alpha} <1$ with $2+\tilde{\alpha} \notin \mathcal{D}$.
\subsection{Higher regularity of solutions for equations with $s$-polynomial coefficients} 
In the previous section, $C^{2+\alpha}_s$-regularity can be obtained by freezing coefficients method. However due to the degeneracy/singular order of \eqref{eq:main}, higher regularity cannot be obtained inductively like the classical bootstrap method. We solve this problem by considering freezing coefficients by $s$-polynomial. In this section, we consider equations of the form
\begin{equation} \label{eq:poly_coeff}
u_t  = L_p u + f ,
\end{equation}
where the operator $L_p$ is given as
\begin{align*}
	L_p  &= P^{i'j'}(X) D_{i'j'} + 2 x_n^{\gamma/2} P^{i'n}(X) D_{i'n} + x_n^{\gamma} P^{nn} (X) D_{nn} \\
	& \qquad\qquad\qquad\qquad\qquad\qquad  + Q^{i'}(X) D_{i'} + x_n^{\gamma/2} Q^n(X) D_{n} + R (X)   
\end{align*}
and the coefficients $P^{ij}$ are $s$-polynomials at $O$ satisfying the following condition
\begin{equation*} 
\lambda |\xi|^2 \leq P^{ij}(X) \xi_i \xi_j \leq \Lambda |\xi|^2 \quad \text{for any } X \in Q_1^+, ~ \xi \in \mathbb R^n
\end{equation*}
and the coefficients $Q^i,\cdots, Q^n, R$ are $s$-polynomials at $O$ satisfying
\begin{equation*}
	\sum_{i=1}^n \|Q^i\|_{L^{\infty}(Q_1^+)} +\|R\|_{L^{\infty}(Q_1^+)}  \leq \Lambda  \qquad \text{and} \qquad R(X) < 0 \quad \text{for any } X \in Q_1^+.
\end{equation*}

\begin{lemma} \label{calc_TBT}
Let $k \in \mathbb{N}_{0}$, $0<\alpha<1$ and assume that the coefficients $P^{ij}$, $Q^i$, and $R$ are $s$-polynomials at $O$ of degree $\mu$ corresponding to $(k+\alpha)$ and the function $f \in C_s^{k,\alpha}(\overline{Q_1^+})$ is of the form 
\begin{equation}  \label{expan_f}
f (X) = 
\begin{dcases}
\sum_{ \substack { \frac{2i}{2-\gamma} + j <  k+\alpha \\ i,j \in \mathbb{N}_0 }} \tilde{f}^{ij}(x',t) x_n^{i+\frac{2-\gamma}{2}j} & (  \gamma < 1) \\
\sum_{ \substack{ i \leq j < k+\alpha \\ i,j \in \mathbb{N}_0 } } \tilde{f}^{ij}(x',t) (\log x_n)^i x_n^{j/2} & (  \gamma = 1) ,
\end{dcases}
\end{equation}
where each $\tilde{f}^{ij}$ is a smooth function for $(x',t)$. Then there is a function $h  \in C_s^{k+2,\alpha}(\overline{Q_1^+})$ of the form
\begin{equation} \label{expan_h}
h(X) = \begin{dcases}
\sum_{ \substack {\frac{2i}{2-\gamma} + j < k+\alpha \\ i,j \in \mathbb{N}_0}} \tilde{h}^{ij}(x',t) x_n^{i+\frac{2-\gamma}{2} (j+2)} & (  \gamma < 1) \\
\sum_{ \substack{ i \leq j  < k+\alpha \\  i,j \in \mathbb{N}_0} } \tilde{h}^{ij}(x',t) (\log x_n)^{i+1}  x_n^{(j+2)/2} & (  \gamma = 1) 
\end{dcases}
\end{equation}	
such that
\begin{equation*} 
	|(h_t - L_p h - f) (X) | \leq C \|f\|_{C^{k,\alpha}_s(\overline{Q_1^+})}x_n^{\frac{2-\gamma}{2}(k+\alpha)}\quad \text{for all } X \in \overline{Q_1^+}
\end{equation*}
and
\begin{equation*} 
	\|h\|_{C^{k+2,\alpha}_s(\overline{Q_1^+})} \leq C \|f\|_{C^{k,\alpha}_s(\overline{Q_1^+})},
\end{equation*}
where each $\tilde{h}^{ij}$ is a smooth function for $(x',t)$ and $C$ depends on $n$, $\gamma$, $k$, $\alpha$, $\|P^{ij}\|_{C^{k,\alpha}_s(\overline{Q_1^+})}$, $\|Q^{i}\|_{C^{k,\alpha}_s(\overline{Q_1^+})}$, and $\|R\|_{C^{k,\alpha}_s(\overline{Q_1^+})}$.
\end{lemma}

\begin{proof}
If $f \equiv 0$, obviously the function we are looking for is $h \equiv 0$. Thus, it is sufficient to consider only the case $f \not \equiv 0$, and $\|f\|_{C^{k,\alpha}_s(\overline{Q_1^+})}=1$ can be assumed. The desired function $h$ can be obtained by comparing $ (h_t - L_p h)$ and $f$ by terms for $x_n$ of both sides. If the operator $(\partial_t - L_p)$ is applied to \eqref{expan_h}, then $ (h_t-L_p h)$ is expressed as follows:   
\begin{equation*}
( h_t - L_p h) (X) =
\begin{dcases}
\sum_{\substack { \frac{2i}{2-\gamma} + j <  k+\alpha+\mu+2 \\ i,j \in \mathbb{N}_0 } } g^{ij} (x',t) x_n^{i+\frac{2-\gamma}{2}j} & (  \gamma < 1) \\
\sum_{\substack {i \leq j <  k+\alpha+\mu+2 \\ i,j \in \mathbb{N}_0 } } g^{ij} (x',t) (\log x_n)^i x_n^{j/2} & (  \gamma = 1), 
\end{dcases}
\end{equation*}
where each $g^{ij}$ is an unknown smooth function 

We will find $\tilde{h}^{ij}$ inductively from the process of tracing $g^{ij}$. Let us introduce some symbols for computational convenience. We add an auxiliary spacial dimension like
\begin{equation*}
	x^\ast =(x_1, \cdots, x_n, z) \in \mathbb{R}^{n+1} \qquad \text{and} \qquad X^* = (x^*, t).
\end{equation*}
Also, we define a modified version of the function $u$ and the operator $L_p$ as follows: The modified function $u^*$ of $u$ is given by the expression
\begin{equation*}  
u^*(X^*)=
\begin{dcases}
	\sum_{i,j \in \mathbb{N}_0} \tilde{u}^{ij}(x',t)  z^i x_n^{\frac{2-\gamma}{2}j} & (\gamma < 1) \\
	\sum_{i,j \in \mathbb{N}_0} \tilde{u}^{ij}(x',t)  z^i x_n^{j/2}  & (\gamma = 1)
\end{dcases} \text{ for }
u(X)=
\begin{dcases}
	\sum_{i,j \in \mathbb{N}_0} \tilde{u}^{ij}(x',t)  x_n^{i + \frac{2-\gamma}{2}j} & (\gamma < 1) \\
	\sum_{i,j \in \mathbb{N}_0} \tilde{u}^{ij}(x',t)  (\log x_n)^i x_n^{j/2}  & (\gamma = 1)
\end{dcases}
\end{equation*}
and the modified operator $L_p^*$ of $L_p$ is given by the expression 
\begin{equation*}
\begin{aligned}
	L_p^*&= {P^{i'j'}}^* D_{i'j'} + 2 {P^{i'n}}^* D_{i'} {\bar{D}}^*_n +  {P^{nn}}^* \bar{D}_{nn}^*  + {Q^{i'}}^* D_{i'} +  {Q^n}^* {\bar{D}}_n^* + R^* ,
\end{aligned}
\end{equation*}
where 
\begin{equation*}  
	{\bar{D}}_n^* = 
	\begin{dcases}
		x_n^{\gamma/2} D_n + z x_n^{-\frac{2-\gamma}{2}} \partial_z & (\gamma<1) \\ 
		x_n^{1/2} D_n + x_n^{-1/2} \partial_z & (\gamma = 1),
	\end{dcases} \quad  
	\bar{D}_{nn}^*  =  
	\begin{dcases}
			x_n^{\gamma} D_n^2 + 2 z x_n^{\gamma-1} D_n \partial_z + z^2 x_n^{\gamma-2} \partial_z^2 & (\gamma <1) \\
			x_n D_n^2 + 2 D_n \partial_z + x_n^{-1} \partial_z^2 - x_n^{-1} \partial_z & (\gamma = 1).
	\end{dcases}
\end{equation*}
Then, we can revert the modified version to the original version as follows:
\begin{equation*} 
	(u_t - L_p u) (X) = 
	\begin{dcases}
		(u_t^* - L_p^* u^*) (X^*) \bigg|_{z=x_n} & (\gamma < 1) \\
		(u_t^* - L_p^* u^*) (X^*) \bigg|_{z=\log x_n} & (\gamma = 1),
	\end{dcases} \quad
	u(X) = 
	\begin{dcases}
		u^*(X^*)\bigg|_{z=x_n} & (\gamma < 1)  \\
		u^*(X^*)\bigg|_{z=\log x_n} & (\gamma = 1) .
	\end{dcases}
\end{equation*}
It is enough to find $\tilde{h}^{ij}$ in the modified version. For arbitrary fixed integers $l,m \in \mathbb{N}_0$, the function $g^{lm}$ in $( h_t^* - L_p^* h^*) $ is the sum of the coefficient functions of 
\begin{equation} \label{fix_id_mono}
	 z^l x_n^{\frac{2-\gamma}{2} m} 
\end{equation}
in the functions obtained by expanding the following operators
\begin{equation} \label{op_mono}
( h_t^* - L_p^* h^*) (X^*) =
\begin{dcases}
	\sum_{ \substack {\frac{2i}{2-\gamma} + j < k+\alpha \\ i,j \in \mathbb{N}_0}} (\partial_t - L_p^*) \left( \tilde{h}^{ij} (x',t) z^i x_n^{\frac{2-\gamma}{2} (j+2)} \right) & (\gamma <1) \\
	\sum_{ \substack {i \leq j < k+\alpha \\ i,j \in \mathbb{N}_0}} (\partial_t - L_p^*) \left( \tilde{h}^{ij} (x',t)z^{i+1} x_n^{(j+2)/2} \right) & (\gamma =1).
\end{dcases}
\end{equation}
We consider several cases according to the range of $i,j$. The first case is when $j  > m$. Then, in the expansion of \eqref{op_mono}, each monomial for $(z,x_n)$ has the order of $x_n$ greater than $m$. Thus,  the expansion of \eqref{op_mono}  cannot have the monomial for $(z,x_n)$ of the form \eqref{fix_id_mono} and hence the coefficient function of \eqref{fix_id_mono} in \eqref{op_mono} is identically zero. 

The second case is when $j=m$. For $\gamma<1$,  the expansion of \eqref{op_mono}  cannot have the monomial for $(z,x_n)$ of the form \eqref{fix_id_mono} if $i > l$. Also, $\tilde{h}^{lm}$ will be obtained by inductively finding functions $\tilde{h}^{0m}$, $\tilde{h}^{1m}$, $\tilde{h}^{2m}$, $\cdots$, $\tilde{h}^{(l-1)m}$, it is sufficient to consider only the case $i=l$ for now. In this case, in the expansion of \eqref{op_mono}, the coefficient function for the monomial of the form \eqref{fix_id_mono} can be obtained as follows:  
\begin{align*}
- {P^{nn}}^*(x',0,0,t) \bar{D}_{nn}^* \left( \tilde{h}^{lm}(x',t)  z^l x_n^{\frac{2-\gamma}{2} (m+2)}  \right)  = S^{lm}(x',t)  \tilde{h}^{lm}(x',t) z^l x_n^{\frac{2-\gamma}{2}m},
\end{align*}
where
\begin{equation*}	
	S^{lm}(x',t) =-\Big( l+ \frac{2-\gamma}{2}(m+2) \Big) \Big( l+ \frac{2-\gamma}{2}(m+2)-1 \Big) {P^{nn}}(x',0,t).
\end{equation*}
For $\gamma =1$, we will show that the functions $\tilde{h}^{0m}$, $\tilde{h}^{1m}$, $\tilde{h}^{2m}$, $\cdots$, $\tilde{h}^{mm} $ are solutions of some system of equations. For $l=0,1,\cdots,m$ and $m \ne 0$, in the expansion of \eqref{op_mono}, the coefficient function for the monomial of the form \eqref{fix_id_mono} can be obtained as follows:  
\begin{align}  
&\left\{
\begin{aligned}
	&-{P^{nn}}^*(x',0,0,t) \left[ (m+1) \tilde{h}^{0m} + 2 \tilde{h}^{1m} \right] x_n^{m/2} && ( l = 0 ) \\
	&-{P^{nn}}^*(x',0,0,t) \left[ \frac{1}{4} m(m+2)  \tilde{h}^{(l-1)m} \right.\\
	&\qquad\qquad\quad  \left.+ (l+1)(m+1) \tilde{h}^{lm} + (l+1)(l+2) \tilde{h}^{(l+1)m} \right] z^l x_n^{m/2}  && ( 1 \leq l  < m ) \\
	& - {P^{nn}}^*(x',0,0,t) \left[ \frac{1}{4} m(m+2)  \tilde{h}^{(m-1)m} + (m+1)^2 \tilde{h}^{mm}  \right]  z^m x_n^{m/2}  && (l=m ).
\end{aligned}\right. \label{rrel_3var}
\end{align}

The last case is when $j<m$. Since we will find $\tilde{h}^{lm}$ inductively, it may be assumed that all of the functions
\begin{equation}\label{pre_hlm}
\begin{cases}
	\tilde{h}^{im}  & ( \gamma< 1 , \, i < l ) \\
	\tilde{h}^{ij}  & ( \gamma< 1 , \, j < m ) \\
	\tilde{h}^{ij}  & ( \gamma= 1 , \, j < m ) 
\end{cases}
\end{equation}
have been found.

Now, we consider a truncated function $h^*_{(lm)}$ of $h^*$ as follows:
\begin{equation} \label{tr_fcn}
	h_{(lm)}^*(X^*) =
	\begin{dcases}
		\sum_{\substack{\frac{2i}{2-\gamma}+m<k+\alpha  \\ i < l}} \tilde{h}^{im}(x',t) z^i x_n^{\frac{2-\gamma}{2}(m+2)} +\sum_{\substack{\frac{2i}{2-\gamma}+ j< k+\alpha \\ j < m}} \tilde{h}^{ij}(x',t) z^i x_n^{\frac{2-\gamma}{2}(j+2)} & (\gamma <1) \\
		\sum_{ \substack{ i \leq j  < k+\alpha \\  j< m} } \tilde{h}^{ij}(x',t) z^{i+1}  x_n^{(j+2)/2} &  (\gamma =1),
	\end{dcases}
\end{equation}
As in finding Taylor polynomials,  we can find the coefficient function of \eqref{fix_id_mono} in $(\partial_t-L_p^*) h_{(lm)}^*$ as
\begin{equation*} 
	\psi^{lm}(x',t)\coloneqq \frac{1}{\left(\frac{2-\gamma}{2}\right)^m l! m! } \partial_z^l \bar{D}_n^m (\partial_t-L_p^*) h_{(lm)}^*(X^*)\Big|_{(z,x_n)=(0,0)},
\end{equation*}
where $\bar{D}_n:=x_n^{\gamma/2} D_n$. Since $\psi^{lm}$ is expressed by the functions of \eqref{pre_hlm}, we already know it by the induction assumption.

For $\gamma < 1$, we can represent the coefficient function $g^{lm}(x',t)$ of \eqref{fix_id_mono} in $( h_t^* - L_p^* h^*) $ as
\begin{align*}
g^{lm}(x',t)= S^{lm}(x',t) \tilde{h}^{lm}(x',t) + \psi^{lm}(x',t) . 
\end{align*}
Now, taking $\tilde{h}^{lm}$ such that $g^{lm} \equiv \tilde{f}^{lm}$, we can determine $\tilde{h}^{lm}$ inductively using the following formula:
\begin{equation} \label{rrel_hlm}
	\tilde{h}^{lm} (x',t)=\frac{\tilde{f}^{lm}(x',t) -\psi^{lm}(x',t)}{S^{lm}(x',t)} .
\end{equation}
Since $S^{lm}$ is nonvanishing smooth function, $\tilde{h}^{lm}$ is also smooth function. Furthermore, we can compute 
$$\tilde{h}^{00}(x',t) = \frac{\tilde{f}^{00}(x',t)}{(2-\gamma)(1-\gamma) P^{nn}(x',0,0,t)}$$ 
first and apply the recurrence relation \eqref{rrel_hlm} to index $i$, we have $\tilde{h}^{10}$, $\tilde{h}^{20}$,  $\cdots$, $\tilde{h}^{l0}$,  $\cdots$. Next, for each $i=1,2,\cdots, l$, applying the recurrence relation \eqref{rrel_hlm} to index $j$ again, we have all of the following functions sequentially: 
$$\begin{array}{cccccc}
\tilde{h}^{00}, &\tilde{h}^{10}, & \tilde{h}^{20}, &\cdots &\tilde{h}^{l0}, & \cdots\\
\tilde{h}^{01}, &\tilde{h}^{11}, &\tilde{h}^{21}, &\cdots &\tilde{h}^{l1}, & \cdots \\
\vdots & \vdots & \vdots & \ddots & \vdots &\\
\tilde{h}^{0m}, & \tilde{h}^{1m}, &\tilde{h}^{2m}, &\cdots &\tilde{h}^{lm}. &
\end{array}$$
Since $l,m$ were arbitrary, we found all $\tilde{h}^{ij}$. 

On the other hand, by using \eqref{rrel_3var} for $\gamma =1$,  we can represent the coefficient function $g^{0m}(x',t)$, $g^{1m}(x',t)$, $\cdots$, $g^{mm}(x',t)$ of \eqref{fix_id_mono} in $( h_t^* - L_p^* h^*) $ as
\begin{align*}
\begin{pmatrix}
g^{0m}\\
g^{1m}\\
g^{2m}\\
\vdots \\
g^{mm}
\end{pmatrix} = -{P^{nn}}^*(x',0,0,t) T^m
\begin{pmatrix}
\tilde{h}^{0m}\\
\tilde{h}^{1m}\\
\tilde{h}^{2m}\\
\vdots \\
\tilde{h}^{mm}
\end{pmatrix}  + 
\begin{pmatrix}
\psi^{0m}\\
\psi^{1m}\\
\psi^{2m}\\
\vdots \\
\psi^{mm}
\end{pmatrix}  ,
\end{align*}
where the tridiagonal matrix $T^m$ is given by
$$\begin{pmatrix}
m+1 & 1\cdot2 & & & &    \\
\frac{1}{4}m(m+2) & 2(m+1) & 2\cdot3 & & &   \\
 & \frac{1}{4}m(m+2)  &  3(m+1) & 3\cdot4 & &   \\
   & & \ddots & \ddots & \ddots &  \\
   & & & \frac{1}{4}m(m+2) & m(m+1)  & m(m+1) \\
   & & & & \frac{1}{4}m(m+2)  & (m+1)^2 
\end{pmatrix}.$$
Since the tridiagonal matrix $T^m$ is invertible, we can determine $\tilde{h}^{0m}$, $\tilde{h}^{1m}$, $\cdots$,  $\tilde{h}^{mm}$ inductively using the following formula
\begin{equation}  \label{rrel_hlm2}
\begin{pmatrix}
\tilde{h}^{0m}\\
\tilde{h}^{1m}\\
\tilde{h}^{2m}\\
\vdots \\
\tilde{h}^{mm}
\end{pmatrix} = - \frac{1}{  {P^{nn}}^*(x',0,0,t) } (T^m)^{-1}
\begin{pmatrix}
\tilde{f}^{0m} - \psi^{0m}\\
\tilde{f}^{1m} - \psi^{1m}\\
\tilde{f}^{2m} - \psi^{2m}\\
\vdots \\
\tilde{f}^{mm} - \psi^{mm}\\
\end{pmatrix} ,
\end{equation}
when we take $g^{im} \equiv \tilde{f}^{im}$ $(i=0,1,2,\cdots,m) .$ 
Since $ {P^{nn}}^*(x',0,0,t) $ is nonvanishing standard polynomial, $\tilde{h}^{0m}$, $\tilde{h}^{1m}$, $ \cdots$, $\tilde{h}^{mm}$ are also smooth functions. Furthermore, we can compute 
$$\tilde{h}^{00}(x',t) = \frac{\tilde{f}^{00}(x',t)}{P^{nn}(x',0,0,t)}$$ 
first and apply the recurrence relation \eqref{rrel_hlm2}, we have $\tilde{h}^{01}$ and $\tilde{h}^{11} $.
Next, if we repeat this process, we can find the functions of each row of the following array from all the functions obtained in the previous step. 
\begin{align*}
&\tilde{h}^{02},  \quad \tilde{h}^{12}, \quad \tilde{h}^{22}, \\
&\tilde{h}^{03},  \quad \tilde{h}^{13}, \quad \tilde{h}^{23}, \quad \tilde{h}^{33}, \\ 
&\tilde{h}^{04},  \quad \tilde{h}^{14}, \quad \tilde{h}^{24}, \quad \tilde{h}^{34}  , \quad \tilde{h}^{44}, \\ 
& \qquad\qquad\qquad \vdots \\
&\tilde{h}^{0m},  \quad \tilde{h}^{1m}, \quad \tilde{h}^{2m}, \quad \tilde{h}^{3m}  , \quad \tilde{h}^{4m}, \quad \cdots  , \quad \tilde{h}^{mm}.
\end{align*}
Since $l,m$ were arbitrary, we found all $\tilde{h}^{ij}$.

Finally,  In \eqref{rrel_hlm} and \eqref{rrel_hlm2}, the function $h^*$ is constructed so that all terms of $f^*$ are deleted in $( h_t^* - L_p^* h^* -f^*)$. Thus, we have
\begin{equation*}
(h_t - L_p h - f)(X) = 
\begin{dcases}
\sum_{\substack { \frac{2i}{2-\gamma} + j \ge  k+\alpha \\ i,j \in \mathbb{N}_0 } } g^{ij} (x',t) x_n^{i+\frac{2-\gamma}{2}j} & (  \gamma < 1) \\
\sum_{\substack {i \leq j,\,  j > k+\alpha \\ i,j \in \mathbb{N}_0 } } g^{ij} (x',t) (\log x_n)^i x_n^{j/2} & (  \gamma = 1).
\end{dcases}
\end{equation*}
This implies that
\begin{equation*}
\begin{aligned}
	|(h_t - L_p h - f) (X)| \leq Cx_n^{\frac{2-\gamma}{2}(k+\alpha)}\quad \text{for all }X\in \overline{Q_1^+}
\end{aligned}
\end{equation*}
and $ \|h\|_{C^{k+2,\alpha}_s(\overline{Q_1^+})}\leq C$, where $C$ depends on $n$, $\gamma$, $k$, $\alpha$, $\|P^{ij}\|_{C^{k,\alpha}_s(\overline{Q_1^+})}$, $\|Q^{i}\|_{C^{k,\alpha}_s(\overline{Q_1^+})}$, and $\|R\|_{C^{k,\alpha}_s(\overline{Q_1^+})}$. 
\end{proof}
\begin{lemma} \label{bdry_c1a_pc}
Let $u \in C^{\infty}(Q^+_1) \cap C(\overline{Q^+_1})$ be a solution of \eqref{eq:poly_coeff} satisfying $\|u\|_{L^{\infty}(Q^+_1)}\leq 1$ and $u =0$ on $\{X \in \partial_p Q_1^+ : x_n=0\}$. If for each $i,j,k=1,2,\cdots,n-1$,
$$\begin{aligned}
\| f \|_{L^{\infty}(Q^+_1)} &\leq 1, \\
\| f_{ijk} \|_{L^{\infty}(Q^+_1)} &\leq 1,
\end{aligned}\qquad 
\begin{aligned}
\| f_i \|_{L^{\infty}(Q^+_1)} &\leq 1, \\
\| f_t \|_{L^{\infty}(Q^+_1)} &\leq 1,
\end{aligned} \qquad 
\begin{aligned}
\| f_{ij} \|_{L^{\infty}(Q^+_1)} &\leq 1, \\
\| f_{kt} \|_{L^{\infty}(Q^+_1)} &\leq 1,
\end{aligned}
$$
and 
$$\left\{\begin{aligned}
|f(x,t)| &\leq M x_n^{\theta} \quad \mbox{for all } (x,t) \in \overline{Q^+_1} \\
|f_k(x,t)| &\leq M x_n^{\theta} \quad \mbox{for all } (x,t) \in \overline{Q^+_1} \\
\end{aligned}\right.$$
for some $\theta \ge \gamma/2$ and $M > 0$, then $u_n$ is well-defined on $\overline{Q^+_{1/2}}$ and
$$|u_n(x',x_n,t) - u_n(x',y_n,t)  |  \leq \left\{\begin{aligned} 
& C|x_n - y_n|^{1-\gamma/2}  && (0< \gamma \leq 1) \\
& C|x_n - y_n|  && (\gamma \leq 0)
\end{aligned}\right.$$
for all $(x',x_n,t),(x',y_n,t) \in \overline{Q^+_{1/2}}$, where $C$ is a positive constant depending only on $n$, $\lambda$, $\Lambda$, $\gamma$, $\alpha$, $ \|D_{x'}^{\beta} \partial_t^k P^{ij}\|_{L^{\infty} (Q_1^+)}$, $ \|D_{x'}^{\beta} \partial_t^k Q^i\|_{L^{\infty}(Q_1^+)}$, and $ \|D_{x'}^{\beta} \partial_t^k R\|_{L^{\infty}(Q_1^+)}$ $(\beta \in \mathbb{N}_0^{n-1}, k \in \mathbb{N}_0, |\beta| + 2k \leq 3)$.
\end{lemma}

\begin{proof}
Differentiate both sides of \eqref{eq:poly_coeff} with respect to $x_k\,(k \ne n)$, then 
\begin{equation}
\left\{\begin{aligned}\label{eq:uk_pc}
v_t & = L_p v + F  && \mbox{in } Q^+_{3/4} \\
v & = 0 && \mbox{on } \{X \in \partial_p  Q_1^+ : x_n=0\},
\end{aligned}\right.
\end{equation}
where $v = u_k$ and 
$$F = P^{i'j'}_k u_{i'j'} + 2 x_n^{\gamma/2} P^{i'n}_k u_{i'n} + x_n^{\gamma} P^{nn}_k u_{nn} + Q^{i'}_k u_{i'} + x_n^{\gamma/2} Q^n_k u_{n} + R_k u + f_k .$$ 
By \Cref{thm:c2a_hc}, $v$ is bounded solution of \eqref{eq:uk_pc} with bounded forcing term $F$, we can apply \Cref{lip_est} to $v$, we have
\begin{equation*}
|u_k(x,t)| = |v(x,t)| \leq  \begin{cases} 
C x_n & (\gamma <1 ) \\ 
- C x_n \log x_n  & (\gamma = 1) 
\end{cases}
\end{equation*}
for all  $(x,t) \in \overline{Q^+_{1/2}}$. We can obtain the same inequality not only for $u_k$ but also for $u_t$ and $u_{ij}\,(i,j=1,2,\cdots,n-1)$. An integration by parts yields
\begin{align} 
&\left| \int_0^z x_n^{-\gamma/2} P^{i'n} (x,t) u_{i'n} (x,t) \,dx_n \right| \nonumber \\
&\quad \leq \left| z^{-\gamma/2}  P^{i'n} (x',z,t) u_{i'}(x',z,t) \right| + C \int_0^z x_n^{-\gamma/2-1}  | u_{i'} (x,t) | \, dx_n \label{ieq4:u_in} \\
&\quad \leq  \left\{\begin{aligned} 
& C z^{1-\gamma/2}  && (\gamma <1 ) \\ 
& -C \sqrt{z} \log z  && (\gamma = 1) 
\end{aligned}\right. \label{ieq3:u_in}
\end{align}
for all $(x',z,t) \in \overline{Q^+_{1/2}}$. Similarly, we also have
\begin{align} 
\left| \int_0^z x_n^{-\gamma/2} Q^n (x,t) u_n(x,t) \,dx_n \right| \leq  \left\{\begin{aligned} 
& C z^{1-\gamma/2}  && (\gamma <1 ) \\ 
& -C \sqrt{z} \log z  && (\gamma = 1) 
\end{aligned}\right.  \label{ieq3:u_n}
\end{align}
for all $(x',z,t) \in \overline{Q^+_{1/2}}$. Combining \eqref{eq:poly_coeff}, \eqref{ieq3:u_in}, and \eqref{ieq3:u_n} gives
\begin{align}
 \left| \int_0^z u_{nn}\, dx_n \right| 
& \leq \frac{1}{\lambda} \left( \left| \int_0^z  x_n^{-\gamma} \big( u_t - P^{i'j'} u_{i'j'} - Q^{i'} u_{i'} - R u  \big) \,dx_n \right|  \right. \label{ieq:int_u_nn_pc}  \\
& \qquad \left. +  \left| \int_0^z  2 x_n^{-\gamma/2} P^{i'n} u_{i'n} \,dx_n \right| +   \left| \int_0^z   x_n^{-\gamma/2} Q^n u_{n}  \,dx_n \right| + \left| \int_0^z x_n^{-\gamma} f  \,dx_n \right| \right) \nonumber  \\
& \leq  \left\{\begin{aligned} 
& C z^{1-\gamma/2}  && (\gamma <1 ) \\ 
& -C \sqrt{z} \log z && (\gamma = 1) 
\end{aligned}\right. \nonumber 
\end{align}
for all $(x',z,t) \in \overline{Q^+_{1/2}}$ and hence for $\varepsilon < 2^{-\frac{2}{2-\gamma}}$,
\begin{align*}
|u_n(x',\varepsilon, t)| 
&= \left| u_n(x',2^{-\frac{2}{2-\gamma}},t) - \int_{0}^{2^{-\frac{2}{2-\gamma}}} u_{nn}(x,t)\, dx_n +  \int_{0}^{\varepsilon} u_{nn}(x,t)\, dx_n \right| \\
&\leq \left\{\begin{aligned} 
& |u_n(x',2^{-\frac{2}{2-\gamma}},t) | + C (2^{-1} + \varepsilon^{1-\gamma/2})  && (\gamma <1 ) \\ 
&  |u_n(x',1/4,t) |  + C ( \log 2  - \sqrt{\varepsilon} \log \varepsilon )&& (\gamma = 1) .
\end{aligned}\right. 
\end{align*}
Therefore, we have $|u_n| \leq C$ in $Q^+_{1/2}$. This implies that even though $\gamma = 1$, we have the following Lipschitz estimate
\begin{equation} \label{lip_est_all_pc}
|u(x,t)| \leq C x_n \quad \mbox{on }  \overline{Q^+_{1/2}}.
\end{equation}
Furthermore $v=u_k~(k\neq n)$ satisfies assumptions of \Cref{bdry_c1a_pc} again, we have $|u_{kn}| \leq C $ in $Q^+_{1/2}$. Thus, \eqref{ieq:int_u_nn_pc} gives
\begin{align} 
| u_{n}(x',x_n,t) -  u_{n}(x',y_n,t) | 
& \leq  C( |x_n^{2-\gamma}- y_n^{2-\gamma}| + |x_n^{1-\gamma/2}- y_n^{1-\gamma/2}| +  |x_n^{\theta -\gamma + 1}- y_n^{\theta-\gamma+1}|) \label{un_holder_pc} \\
& \leq \begin{cases} 
C |x_n - y_n|^{1-\gamma/2} & (0 < \gamma \leq 1 ) \\ 
C |x_n - y_n|   & (\gamma \leq 0) 
\end{cases} \label{eq:un_hol_pc}
\end{align}
for all $x_n, y_n \in (0,2^{-\frac{2}{2-\gamma}})$. 

As in the proof of \Cref{bdry_c1a}, we can extend \eqref{eq:un_hol_pc} on $\overline{Q_{1/2}^+}$.
\end{proof}
\begin{lemma} \label{lem_expan_sol}
	Let  $k \in \mathbb{N}_{0}$, $0 < \alpha < 1$, and $u \in C^{\infty}(Q^+_1) \cap C(\overline{Q^+_1})$ be a solution of \eqref{eq:poly_coeff} with $f\equiv0$. If the coefficients $P^{ij}$, $Q^i$, and $R$ are $s$-polynomials at $O$ of degree $\mu$ corresponding to $\kappa = k+\alpha$, $\|u\|_{L^\infty(Q^+_1)}\leq 1$, and $u =0$ on $\{X \in \partial_p  Q_1^+ :x_n=0\}$, then for any $M, N \in \mathbb{N}$, there exists a function $v$ of the form
\begin{equation} \label{expan_sol}
	v (X) =  
	\begin{dcases}
		u_n(x',0,t) x_n + \sum_{\substack{ \frac{2i}{2-\gamma} + j < \frac{2M}{2-\gamma}+N \\ i,j \in \mathbb{N}}} \tilde{v}^{ij} (x',t)x_n^{i+\frac{2-\gamma}{2}j}  &  \textnormal{if } (M,N) \ne (1,1) \\
		u_n(x',0,t) x_n & \textnormal{if } (M,N) =(1,1)
	\end{dcases}
\end{equation}
such that 
\begin{equation*} 
	| (v_t - L_p v) (X)| \leq Cx_n^{M + \frac{2-\gamma}{2}(N-2)}  \quad \text{for all } X \in \overline{Q^+_{1/2}}
\end{equation*}
and
\begin{equation*}
| u(X)- v(X)  | \leq  C x_n^{M+\frac{2-\gamma}{2}N} \quad \text{for all } X \in \overline{Q^+_{1/2}}, 
\end{equation*}
where each $\tilde{v}^{ij}$ is a smooth function for $(x',t)$ and $C$ depends on $n$, $\gamma$, $\alpha$, $M$, $N$, $\|P^{ij}\|_{C^{k,\alpha}_s(\overline{Q_1^+})}$, $\|Q^{i}\|_{C^{k,\alpha}_s(\overline{Q_1^+})}$, and $\|R\|_{C^{k,\alpha}_s(\overline{Q_1^+})}$.
\end{lemma} 

\begin{proof}
For the modified function ${v}^*$ of $v$ and the modified operator $L_p^*$ of $L_p$, the function $ ({v}^*_t - L_p^* {v}^*)$ is expressed as follows:   
\begin{equation} \label{mv_Lv}
	({v}^*_t - L_p^* {v}^*) (X^*) =  \sum_{\substack { \frac{2i}{2-\gamma} + j < \frac{2M}{2-\gamma}+N+\mu+2 \\ i,j \in \mathbb{N} } } g^{ij} (x',t) z^ i x_n^{ \frac{2-\gamma}{2}(j-2)} ,
\end{equation}
where each $g^{ij}$ is an unknown smooth function.

For arbitrary fixed integers $l,m \in \mathbb{N}$, the function $g^{lm}$ in $( \partial_t - L_p^*) {v}^* $ is the sum of the coefficient functions of 
\begin{equation} \label{fix_id_mono2}		 
	z^l x_n^{\frac{2-\gamma}{2} (m-2)}
\end{equation}
in the functions obtained by expanding the following operators
\begin{equation} \label{op_mono2}
	({v}^*_t - L_p^* {v}^*) (X^*) =  (\partial_t -L_p^* ) \Big(u_n(x',0,t) z \Big)  +  \sum_{\substack{ \frac{2i}{2-\gamma} + j < \frac{2M}{2-\gamma}+N \\ i,j \in \mathbb{N}}}  (\partial_t - L_p^*) \left( \tilde{v}^{ij} (x',t) z^i x_n^{\frac{2-\gamma}{2} j} \right).
\end{equation}
Since we will find $\tilde{v}^{ij}$ in the same way as \Cref{calc_TBT},  it is sufficient to consider only the case $(i,j)=(l,m)$ for now. In this case, in the expansion of \eqref{op_mono2}, the coefficient function for the monomial of the form \eqref{fix_id_mono2} can be obtained as follows: 
\begin{align*}
- {P^{nn}}^*(x',0,0,t) \bar{D}_{nn}^* \left( \tilde{v}^{lm}(x',t)  z^l x_n^{\frac{2-\gamma}{2}m}  \right)  = S^{lm}(x',t)  \tilde{v}^{lm}(x',t) z^l x_n^{\frac{2-\gamma}{2}(m-2)},
\end{align*}
where
\begin{equation*}	
	S^{lm}(x',t) =-\Big( l+ \frac{2-\gamma}{2}m \Big) \Big( l+ \frac{2-\gamma}{2}m-1 \Big) {P^{nn}}^*(x',0,0,t).
\end{equation*}
Now, we consider a truncated function ${v^*_{(lm)}}$ of $v^*$ as follows:
\begin{equation*}
	v_{(lm)}^* (X^*) =u_n(x',0,t) z + \sum_{\substack{ \frac{2i}{2-\gamma} + m < \frac{2M}{2-\gamma}+N \\ 1 \leq i < l}}  \tilde{v}^{im}(x',t) z^i x_n^{\frac{2-\gamma}{2}m} +  \sum_{\substack{ \frac{2i}{2-\gamma} + j < \frac{2M}{2-\gamma}+N \\ i \in \mathbb{N}, \, 1 \leq j < m}}  \tilde{v}^{ij}(x',t) z^i x_n^{\frac{2-\gamma}{2}j} ,
\end{equation*}
If $j=1$, the order of $x_n$ is negative in \eqref{mv_Lv}, so considering $x_n^{\frac{2-\gamma}{2}} (\partial_t-L_p^*) v_{(lm)}^*$, we can find the coefficient function of \eqref{fix_id_mono2} in $x_n^{\frac{2-\gamma}{2}} (\partial_t-L_p^*) v_{(lm)}^*$ as
\begin{equation*} 
	\psi^{lm}(x',t)\coloneqq \frac{1}{\left(\frac{2-\gamma}{2}\right)^{m-1} l! (m-1)! } \partial_z^l \bar{D}_n^{m-1} \left( x_n^{\frac{2-\gamma}{2}} (\partial_t-L_p^*) v_{(lm)}^*(X^*) \right) \bigg|_{(z,x_n)=(0,0)},
\end{equation*}
where $\bar{D}_n:=x_n^{\gamma/2} D_n$. Thus, we can represent the function $g^{lm}(x',t)$ in $( \partial_t - L_p^*) {v}^* $ as
\begin{align*}
g^{lm}(x',t)= S^{lm}(x',t) \tilde{v}^{lm}(x',t) + \psi^{lm}(x',t) . 
\end{align*}
Now, taking $\tilde{v}^{lm}$ such that $g^{lm} \equiv 0$, we can determine $\tilde{v}^{lm}$ inductively using the following formula:
\begin{equation} \label{rrel_hlm3} 
	\tilde{v}^{lm} (x',t)= - \frac{\psi^{lm}(x',t)}{S^{lm}(x',t)} 
\end{equation}
Since $l,m$ were arbitrary, we found all $\tilde{v}^{ij}$. 

Using the recurrence relation \eqref{rrel_hlm3}, for any $M,N \in \mathbb{N}$, we can construct the function $v^*$ such that all monomials of the form
$$ g^{ij}(x',t) z^i x_n^{\frac{2-\gamma}{2}(j-2)} \quad \left( \frac{2i}{2-\gamma} + j < \frac{2M}{2-\gamma}+N  \right) $$ 
are deleted in $( {v}^*_t - L_p^*{v}^*)  $. Thus, we have
\begin{equation*}
	(v_t - L_p v) (X) = \sum_{\substack { \frac{2i}{2-\gamma} + j \ge \frac{2M}{2-\gamma}+N \\ i,j \in \mathbb{N}  } } g^{ij} (x',t) x_n^{i+\frac{2-\gamma}{2}(j-2)} .
\end{equation*}
This implies that 
\begin{equation} \label{diff_op_v}
	|D_{(x',t)}^\beta (v_t - L_p v)(X)| \leq Cx_n^{M + \frac{2-\gamma}{2}(N-2)}  
\end{equation}
for all $X \in \overline{Q^+_{3/4}}$ and $\beta \in \mathbb{N}_{0}^{n}$, where $C$ depends on $n$, $\gamma$, $\alpha$, $\beta$, $M$,  $N$, $\|P^{ij}\|_{C^{k,\alpha}_s(\overline{Q_1^+})}$, $\|Q^{i}\|_{C^{k,\alpha}_s(\overline{Q_1^+})}$, and $\|R\|_{C^{k,\alpha}_s(\overline{Q_1^+})}$. 
 
We now claim that  for any $i_1,i_2 \in \mathbb{N}$ with $i_1 + \frac{2-\gamma}{2} i_2 \leq M +\frac{2-\gamma}{2}N$, the following inequalities hold:
\begin{equation} \label{approx_N-th_pc}
	| D_{(x',t)}^{\beta} ( u- v )(X)  | \leq  C x_n^{i_1+\frac{2-\gamma}{2} i_2}  
	\qquad \text{and} \qquad 
	| D_{(x',t)}^{\beta} ( u_n- v_n )(X)  |  \leq  C x_n^{i_1-1+\frac{2-\gamma}{2} i_2}  
\end{equation}
for all $X \in \overline{Q^+_{1/2}}$ and $\beta \in \mathbb{N}_0^n$. The proof is by induction on $M$ and $N$. Suppose first $(i_1,i_2)=(1,1)$. By \Cref{rmk:bdry_hi} and \eqref{expan_sol}, we know that 
\begin{equation*} 
	| D_{(x',t)}^{\beta} ( u- v )(X)  |  \leq  C x_n 
	\qquad \text{and} \qquad 
	| D_{(x',t)}^{\beta} ( u_n- v_n )(X)  | \leq  C 
\end{equation*}
for all $X \in \overline{Q^+_{1/2}}$ and $\beta \in \mathbb{N}_0^n$. The operator $x_n^{\gamma} P^{nn} D_{nn}$ can be represented as
$$x_n^{\gamma} P^{nn} D_{nn} = \partial_t - P^{i'j'} D_{i'j'} - 2 x_n^{\gamma/2} P^{i'n} D_{i'n} - Q^{i'} u_{i'} - x_n^{\gamma/2} Q^n D_{n} - R - (\partial_t -L_p) $$
and hence we know that
\begin{align} \label{op_id} 
	x_n^{\gamma} D_{nn}  D_{(x',t)}^{\beta} (u-v) & = D_{(x',t)}^{\beta} \big(  x_n^{\gamma} D_{nn} (u-v) \big) \nonumber \\
	& = D_{(x',t)}^{\beta} \left(  \frac{1}{P^{nn} } \partial_t - \frac{P^{i'j'} }{P^{nn} } D_{i'j'} - \frac{Q^{i'}}{P^{nn}} D_{i'} - \frac{R}{P^{nn}} \right)(u-v) \nonumber \\
	&\quad    - x_n^{\gamma/2} D_{(x',t)}^{\beta} \left( 2 \frac{P^{i'n}}{P^{nn}} D_{i'} + \frac{Q^n}{P^{nn}} \right)  (u_n-v_n)   + D_{(x',t)}^{\beta} \left( \frac{1}{P^{nn}} ( v_t - L_p v) \right) 
\end{align}
Since $P^{nn}$ is nonvanishing smooth function for $x'$ and $t$, by \Cref{leibniz_form}, \eqref{diff_op_v}, and \eqref{op_id}, we have
\begin{equation}  \label{first_est}
| x_n^{\gamma} D_{nn} D_{(x',t)}^{\beta} (u-v)(X)| \leq C ( x_n + x_n^{\gamma/2} + x_n^{\gamma/2} ) \leq C  x_n^{\gamma/2}
\end{equation} 
for all $X \in \overline{Q^+_{1/2}}$ and $\beta \in \mathbb{N}_0^n$. Since $u = v$ on $\{x_n=0\}$ and $u_n = v_n$ on $\{x_n=0\}$, if we divide \eqref{first_est} by both sides $x_n^{\gamma}$ and integrate on the $x_n$-variable from 0 to $x_n$, we have 
$$| D_{(x',t)}^{\beta} (u-v)(X)| \leq C x_n^{1 + \frac{2-\gamma}{2}} \qquad \text{and} \qquad | D_{(x',t)}^{\beta} (u_n-v_n)(X)| \leq C x_n^{\frac{2-\gamma}{2}}$$
for all $X \in \overline{Q^+_{1/2}}$ and $\beta \in \mathbb{N}_0^n$. 

Now, let us prove \eqref{approx_N-th_pc} for fixed $i_1 \leq M$ by using induction on $i_2$. We assume \eqref{approx_N-th_pc} is valid for some $i_2 \in \mathbb{N}$ with $i_2 < N+ \frac{2}{2-\gamma}(M-i_1)$. From \Cref{leibniz_form}, \eqref{diff_op_v}, \eqref{op_id} and induction assumption, we have
\begin{equation*} 
| x_n^{\gamma} D_{nn} D_{(x',t)}^{\beta} (u-v)(X)| \leq C \left( x_n^{i_1+\frac{2-\gamma}{2} i_2} + x_n^{i_1+ \frac{2-\gamma}{2} (i_2-1)} + x_n^{M + \frac{2-\gamma}{2} (N-2)} \right) \leq C x_n^{i_1+ \frac{2-\gamma}{2} (i_2-1)}\end{equation*} 
for all $X \in \overline{Q^+_{1/2}}$ and $\beta \in \mathbb{N}_0^n$. Then, we know that 
$$| D_{(x',t)}^{\beta} (u-v)(X)| \leq C x_n^{i_1 + \frac{2-\gamma}{2}(i_2+1)} \qquad \text{and} \qquad | D_{(x',t)}^{\beta} (u_n-v_n)(X)| \leq C x_n^{i_1-1 + \frac{2-\gamma}{2}(i_2+1)}$$
for all $X \in \overline{Q^+_{1/2}}$ and $\beta \in \mathbb{N}_0^n$ in the same way as for $(i_1,i_2)=(1,1)$. 

Next, let us prove \eqref{approx_N-th_pc} completely again using induction on $i_1$. We assume that  \eqref{approx_N-th_pc} is valid for some $i_1,i_2 \in \mathbb{N}$ with $i_1 < M$. Then, by the induction process on $i_2$, we can see that \eqref{approx_N-th_pc} holds for $(i_1,j_2)$, where 
$$j_2 \ge N + \frac{2}{2-\gamma}(M-i_1) \ge 1 + \frac{2}{2-\gamma}.$$
It implies that 
\begin{equation*}
	| D_{(x',t)}^{\beta} (u-v)(X)| \leq C x_n^{i_1 + \frac{2-\gamma}{2}j_2} \leq C x_n^{i_1+1 +\frac{2-\gamma}{2}}
\end{equation*}
and
\begin{equation*}
	| D_{(x',t)}^{\beta} (u_n-v_n)(X)| \leq C x_n^{i_1-1 + \frac{2-\gamma}{2}j_2} \leq C x_n^{i_1 + \frac{2-\gamma}{2}}.
\end{equation*}
Thus, we obtain  \eqref{approx_N-th_pc} for $(i_1+1,1)$. Therefore, we conclude that \eqref{approx_N-th_pc} holds for all $i_1,i_2 \in \mathbb{N}$ with $i_1 + \frac{2-\gamma}{2} i_2 \leq M +\frac{2-\gamma}{2}N$. 
\end{proof}

\begin{lemma} \label{lem:pcf_hi_zf}
Let  $k \in \mathbb{N}_{0}$, $0 < \alpha < 1$, and  $u \in C^{\infty}(Q^+_1) \cap C(\overline{Q^+_1})$ be a solution of \eqref{eq:poly_coeff} with $f=0$. If the coefficients $P^{ij}$, $Q^i$, and $R$ are $s$-polynomials at $O$ of degree $\mu$ corresponding to $(k+\alpha)$, $\|u\|_{L^\infty(Q^+_1)}\leq 1$, $\|u\|_{L^\infty(Q^+_1)}\leq 1$, and $u =0$ on $\{X \in \partial_p  Q_1^+ :x_n=0\}$, then for each $M,N \in \mathbb{N}$, there exists an $s$-polynomial $p$ of degree $m$ corresponding to $\kappa \coloneqq N+\frac{2M}{2-\gamma}$ at $O$ such that
\begin{equation*}
	|(p_t -L_p p)(X)| \leq Cx_n^{M + \frac{2-\gamma}{2}(N-2)} \quad \mbox{for all } X \in \overline{Q^+_{1/2}}
\end{equation*}
and
\begin{equation*}
	|u(X)-p(X)|\leq Cs[X,O]^{N+\frac{2M}{2-\gamma}} \quad \mbox{for all } X \in \overline{Q^+_1},
\end{equation*}
where $C$ is a positive constant depending only on $n$, $\lambda$, $\Lambda$, $\gamma$, $M$, $N$, $\|P^{ij}\|_{C^{k,\alpha}_s(\overline{Q^+_1})}$, $\|Q^{i}\|_{C^{k,\alpha}_s(\overline{Q^+_1})}$, and $\|R\|_{C^{k,\alpha}_s(\overline{Q^+_1})}$.
\end{lemma}

\begin{proof}
By \Cref{lem_expan_sol}, for any $M,N \in \mathbb{N}$, there exists a function $v$ of the form \eqref{expan_sol}
such that 
\begin{equation*}
|u(X) - v(X)  | \leq  C x_n^{M+\frac{2-\gamma}{2}N}
\end{equation*}
for all $X \in \overline{Q^+_{1/2}}$ and $\beta \in \mathbb{N}_{0}^{n}$. Since each $\tilde{v}^{ij}(x',t)$ is a smooth function, by Taylor theorem, there exists Taylor polynomial 
$$T^{ij}(x',t) = \sum_{|\beta| + 2k \leq 2M+N-1 } \frac{D_{x'}^{\beta} \partial_t^k \tilde{v}^{ij}(O)}{\beta!  k!}  {x'} ^{\beta} t^{k}  $$ 
of $\tilde{v}^{ij}(x',t)$ at $O$ such that 
$$|\tilde{v}^{ij}(x',t) -T^{ij}(x',t) | \leq C \sum_{|\beta| + 2k = 2M+N } |x_1|^{\beta_1} \cdots |x_{n-1}|^{\beta_{n-1}} |t|^{k}  $$
for all $|x_i| < 1/2 \, (i=1,2,\cdots,n-1)$ and $- 1/4 < t \leq 0$. This implies that 
\begin{align}
	&\left| v(X)  - \sum_{\substack{ \frac{2i}{2-\gamma} + j < \frac{2M}{2-\gamma}+N \\ i,j \in \mathbb{N}  \text{ or } (i,j)=(1,0) }} T^{ij}(x',t) x_n^{i+\frac{2 - \gamma}{2}j} \right| \nonumber\\
	&\quad \leq C \sum_{\substack{ \frac{2i}{2-\gamma} + j < \frac{2M}{2-\gamma}+N \\ i,j \in \mathbb{N}  \text{ or } (i,j)=(1,0) }}   \sum_{|\beta| + 2k = 2M+N } |x_1|^{\beta_1} \cdots |x_{n-1}|^{\beta_{n-1}} |t|^{k} x_n^{i+\frac{2 - \gamma}{2}j} \nonumber \\
	&\quad \leq C s[X,O]^{2M+N+ \frac{2}{2-\gamma} } \label{appox_g-poly}
\end{align}
for all $X \in \overline{Q^+_{1/2}}$. Now put 
\begin{equation} \label{appox_spol_pc}
	p (X) =  \sum_{|\beta| + 2k + \frac{2i}{2-\gamma} + j  < N + \frac{2M}{2-\gamma} } \frac{D_{x'}^{\beta} \partial_t^k \tilde{v}^{ij}(O)}{\beta! k!}  {x'}^{\beta} x_n^{i+\frac{2 - \gamma}{2}j} t^{k} 
\end{equation}
which is an $s$-polynomial of degree $m$ corresponding to $\kappa = N + \frac{2M}{2-\gamma}$. Then, from \eqref{diff_op_v}, we know that
\begin{equation*}
	|(p_t -L_p p) (X)| \leq Cx_n^{M + \frac{2-\gamma}{2}(N-2)}\quad \mbox{for all } X \in \overline{Q^+_{1/2}}.
\end{equation*}
Combining  \eqref{approx_N-th_pc} and \eqref{appox_g-poly} gives 
\begin{align*}
	|u(X) - p(X)| & \leq |u(X)- v(X) | + \left| v(X) -  \sum_{\substack{ \frac{2i}{2-\gamma} + j < \frac{2M}{2-\gamma}+N \\ i,j \in \mathbb{N}  \text{ or } (i,j)=(1,0) }}  T^{ij}(x',t) x_n^{i+\frac{2 - \gamma}{2}j} \right|  \\
	&\qquad\qquad\qquad\qquad\qquad\quad + \left| \sum_{\substack{ \frac{2i}{2-\gamma} + j < \frac{2M}{2-\gamma}+N \\ i,j \in \mathbb{N}  \text{ or } (i,j)=(1,0) }} T^{ij}(x',t) x_n^{i+\frac{2 - \gamma}{2}j}  - p (x,t) \right| \\
	& \leq C x_n^{M + \frac{2-\gamma}{2}N} + C s[X,O]^{2M +N+\frac{2}{2-\gamma}} \\
	&\qquad   + \sum_{|\beta| + 2k + \frac{2i}{2-\gamma} + j  \ge N + \frac{2M}{2-\gamma} } \frac{ |D_{x'}^{\beta} \partial_t^k \tilde{v}^{ij}(O)|}{\beta! k!}  |x_1|^{\beta_1} \cdots |x_{n-1}|^{\beta_{n-1}} x_n^{i+\frac{2 - \gamma}{2}j} |t|^{k}  \\
	& \leq C s[X,O]^{N+\frac{2M}{2-\gamma}}
\end{align*}
for all $X \in \overline{Q^+_{1/2}}$ which is extensible throughout $Q_1^+$.
\end{proof}
\subsection{Generalized coefficient freezing method} 
In this section, we prove the main theorem. By combining $C_s^{2+\alpha}$-regularity of solutions for \eqref{eq:main} and boundary $C_s^{k,2+\alpha}$-regularity of solutions for \eqref{eq:poly_coeff}, boundary $C_s^{k,2+\alpha}$-regularity of solutions for \eqref{eq:main} can be obtained, and finally we have \Cref{thm:main} by combining \Cref{cond1_grt}, \Cref{thm:int_bd_gb} and \Cref{lem:hi_nzf}.
\begin{lemma} \label{lem:hi_nzf}
Let  $k \in \mathbb{N}_{0}$, $0<\alpha< 1$ with $ k+2+\alpha \notin \mathcal{D}$, and assume 
$$a^{ij}, \, b^i, \, c, \, f \in C^{k,\alpha}_s(\overline{Q_1^+}) \quad (i,j=1,2,\cdots,n).$$
Suppose $u \in C^2(Q^+_1)\cap  C(\overline{Q^+_1})$ is a solution of \eqref{eq:main} satisfying $u=0$ on $\{ X \in \partial_p Q^+_1 :x_n=0\}$. Then there exists an $s$-polynomial $p$ of degree $m$ corresponding to $\kappa \coloneqq k + 2 + \alpha$ at $O$ such that  
\begin{equation*}
	|u(X)-p(X)| \leq C(\|u\|_{L^\infty(Q^+_1)}+\|f\|_{C^{k,\alpha}_s(\overline{Q^+_1})} ) s[X,O]^{k+2+\alpha} \quad \mbox{for all } X \in \overline{Q^+_1} ,
\end{equation*}
where $C$ is a positive constant depending only on $n$, $\lambda$, $\Lambda$, $\gamma$, $k$, $\alpha$, $\|a^{ij}\|_{C^{k,\alpha}_s(\overline{Q^+_1})}$, $\|b^{i}\|_{C^{k,\alpha}_s(\overline{Q^+_1})}$, and $\|c\|_{C^{k,\alpha}_s(\overline{Q^+_1})}$.
\end{lemma}

\begin{proof}
By considering $u/(\|u\|_{L^\infty(Q^+_1)} + \varepsilon^{-1} \|f\|_{C^{\alpha}_s(\overline{Q^+_1})} )$ for $\varepsilon >0$, we may assume without loss of generality that  $\|u\|_{L^\infty(Q^+_1)}\leq1 $ and $\|f\|_{C^{k,\alpha}_s(\overline{Q^+_1})} \leq \varepsilon$. By scaling and the $C^{k,\alpha}_s$-H\"older continuity for $a^{ij}$, $b^i$, $c$, and $f$, we also assume that 
\begin{align}
	|a^{ij}(X) - P^{ij}(X)| & \leq \varepsilon s[X,O]^{k+\alpha}, \label{approx_aij} \\	
	|b^{i}(X) - Q^{i}(X)| & \leq \varepsilon s[X,O]^{k+\alpha}, \label{approx_bi}  \\	
	|c(X) - R(X)| & \leq \varepsilon s[X,O]^{k+\alpha}, \label{approx_c}  \\
	|f(X) - F(X)| &\leq \varepsilon s[X,O]^{k+\alpha}  \label{approx_f} 
\end{align}
for all $X \in \overline{Q_1^+}$, where $P^{ij} $, $Q^i$, $R$, and $F$ are $s$-polynomials with degree $\tilde{m}$ corresponding to $(k+\alpha)$ at $O$. 

Let $L_p $ be the operator given by
\begin{align*}
	L_p &= P^{i'j'}(X) D_{i'j'} + 2 x_n^{\gamma/2} P^{i'n}(X) D_{i'n} + x_n^{\gamma} P^{nn} (X) D_{nn} \\
	& \qquad\qquad\qquad\qquad\qquad\qquad\qquad  + Q^{i'}(X) D_{i'} + x_n^{\gamma/2} Q^n(X) D_{n} + R (X)  
\end{align*}
From \eqref{parabolicity} and \eqref{approx_aij}, we know that 
\begin{equation*}
	(\lambda - \varepsilon n^2) |\xi|^2 \leq P^{ij} (X) \xi_i \xi_j \leq  (\Lambda  + \varepsilon n^2)|\xi|^2 
\end{equation*}
for all $X \in \overline{Q_1^+}$, $\xi \in \mathbb R^n$ and hence for sufficiently small $\varepsilon >0$, we have 
\begin{equation*}
\frac{\lambda}{2} |\xi|^2 \leq P^{ij} (X) \xi_i \xi_j \leq 2 \Lambda |\xi|^2\quad\mbox{for any } X \in \overline{Q_1^+}, ~ \xi \in \mathbb R^n .
\end{equation*}
By \Cref{calc_TBT},  there exists a function $h  \in C_s^{k+2,\alpha}(\overline{Q_1^+})$ of the form \eqref{expan_h} such that
\begin{equation}  \label{inq:op_h} 
	|(F + L_p h - h_t) (X) | \leq C \varepsilon s[X,O]^{k+\alpha}\quad \text{for all } X \in \overline{Q_1^+}
\end{equation}
and $\|h\|_{C^{k+2,\alpha}_s(\overline{Q_1^+})} \leq C \varepsilon,$ where $C$ depends on $n$, $\gamma$, $k$, $\alpha$, $\|a^{ij}\|_{C^{k,\alpha}_s(\overline{Q_1^+})}$, $\|b^i\|_{C^{k,\alpha}_s(\overline{Q_1^+})}$, and $\|c\|_{C^{k,\alpha}_s(\overline{Q_1^+})}$. 

By \Cref{thm:c2a_hc}, we know that $u \in C_s^{2+\alpha} ( \overline{Q_{1/2}^+}).$
Now, we decompose $(u-h)$ into the sum of $v^1$ and $w^1$ such that
\begin{equation*}
\left\{\begin{aligned}
	v^1_t &= L_p v^1 && \mbox{in } Q_{1/2}^+ \\
	v^1 &= u-h && \mbox{on }\partial_p Q_{1/2}^+
\end{aligned} \right.
\qquad \textnormal{and} \qquad 
\left\{\begin{aligned}
	w_t^1 &= L _p w^1 + \tilde{f}^1  && \text{in } Q_{1/2}^+ \\
	w^1 & = 0 && \text{on }\partial_p Q_{1/2}^+,
\end{aligned}\right.
\end{equation*}
where $\tilde{f}^1 = f + L_p h - h_t +  Lu-L_pu $. Then, combining \eqref{approx_aij}-\eqref{inq:op_h} leads us to the estimate
\begin{align}
	|\tilde{f}^1(X)| & \leq  |f(X) -  F(X)| + |(F + L_p h - h_t)(X)| + |(Lu-L_p u) (X)| \nonumber \\
	& \leq \varepsilon s[X,O]^{k+\alpha}  + C \varepsilon s[X,O]^{k+\alpha}  + C \varepsilon \|u\|_{C^{2,\alpha}_s (\overline{Q_{1/2}^+})}  s[X,O]^{k+\alpha} \nonumber \\
	& \leq C \varepsilon s[X,O]^{k+\alpha}  \label{inq:tf<s}
\end{align}
for all $X \in \overline{Q_{1/2}^+}$. By \Cref{apriori_bound1}, for fixed $r \in (0,1/2)$,  \eqref{inq:tf<s} gives
\begin{equation} \label{w_1st_hcf}
	|w^1(X)| \leq C \| \tilde{f}^1 \|_{L^{\infty}( Q_{1/2}^+ )} \leq C\varepsilon  \quad \text{for all } X \in \overline{Q_r^+} 
\end{equation}
and hence, we have $|u(X) - v^1(X) - h(X)| = |w^1(X)|  \leq C \varepsilon$ for all $X \in \overline{Q_r^+}$. 
Choose now $\varepsilon$ small enough such that  $ C \varepsilon < r^{k+2+\alpha}$. Then, $ |u(X) - v^1(X) - h(X)| \leq  r^{k+2 + \alpha}$ for all $X \in \overline{Q_r^+}$. Let us now define
$$  u^2(X) \coloneqq \frac{(u - v^1- h ) (rX) }{r^{k+2+\alpha}} \qquad \text{and} \qquad f^2(X) \coloneqq \frac{ \tilde{f}^1 (rX) }{r^{k+\alpha}}.$$
Then, $\|u^2\|_{L^{\infty}(Q_1^+ )} \leq 1$ and $u_t^2  =  L_p^2 u^2 + \tilde{f}^2 $ in $ Q_1^+$, where
\begin{align*}
	L_p^2  &= P^{i'j'}_{(2)}(X) D_{i'j'} + 2 x_n^{\gamma/2} P^{i'n}_{(2)}(X) D_{i'n} + x_n^{\gamma} P^{nn}_{(2)} (X) D_{nn} \\
	& \qquad\qquad\qquad\qquad\qquad\qquad\qquad  + Q^{i'}_{(2)}(X) D_{i'} + x_n^{\gamma/2} Q^n_{(2)}(X) D_{n} + R_{(2)} (X)  
\end{align*}
with
\begin{equation*}
P_{(2)}^{ij}(X) = P^{ij} (rX) , \quad 
Q_{(2)}^{i}(X) = r Q^{i} (rX) , \quad \text{and} \quad 
R_{(2)}(X) = r^2 R (rX ). 
\end{equation*}
Furthermore, \eqref{inq:tf<s} leads us to the estimate $|\tilde{f}^{2}(X)|  = r^{-(k+\alpha)}  \left|\tilde{f}^1(rX) \right|  \leq C\varepsilon s[X,O]^{k+\alpha}$ for all $X \in \overline{Q_{1/2}^+}$. That is, the same hypotheses as before are fulfilled. Repeating the same procedure, we decompose $u^2$ into the sum of $v^2$ and $w^2$ such that
\begin{equation*}
\left\{\begin{aligned}
	v^2_t &= L_p^2 v^2 && \mbox{in } Q_{1/2}^+ \\
	v^2 &= u^2 && \mbox{on }\partial_p Q_{1/2}^+,
\end{aligned} \right.
\qquad \qquad 
\left\{\begin{aligned}
	w_t^2 &= L _p^2 w^2 + \tilde{f}^2  && \text{in } Q_{1/2}^+ \\
	w^2 & = 0 && \text{on }\partial_p Q_{1/2}^+,
\end{aligned}\right.
\end{equation*}
and  $|u^2(X) - v^2(X)  | \leq r^{k+2 + \alpha}$ for all $X  \in \overline{Q_r^+}$. By substituting back, we have
\begin{equation*}
	|u(X) - h(X)  - v^1(X) - r^{k+2+\alpha} v^2 (r^{-1}X)  | \leq r^{2(k+2 + \alpha)} \quad \text{for all } X  \in \overline{Q_{r^2}^+}. 
\end{equation*}
Continuing iteratively, for each integer $l \ge 3$, let us define the sequence of functions $\{u^l\}$ inductively as follows:
$$ u^l(X) \coloneqq \frac{(u^{l-1} - v^{l-1}) (rX ) }{r^{k+2+\alpha}}  \qquad \text{and} \qquad f^l(X) \coloneqq \frac{ \tilde{f}^{l-1} (rX) }{r^{k+\alpha}}.$$
Then, $\|u^l\|_{L^{\infty}(Q_1^+ )} \leq 1$ and $u_t^l = L_p^l u^l + \tilde{f}^l $ in $ Q_1^+$, where
\begin{align*}
	L_p^l  &= P^{i'j'}_{(l)}(X) D_{i'j'} + 2 x_n^{\gamma/2} P^{i'n}_{(l)}(X) D_{i'n} + x_n^{\gamma} P^{nn}_{(l)} (X) D_{nn} \\
	& \qquad\qquad\qquad\qquad\qquad\qquad + Q^{i'}_{(l)}(X) D_{i'} + x_n^{\gamma/2} Q^n_{(l)} (X) D_{n} + R_{(l)} (X)  
\end{align*}
with
\begin{equation*}
P_{(l)}^{ij}(X) = P_{(l-1)}^{ij} (rX)  , \quad 
Q_{(l)}^{i}(X)= r Q_{(l-1)}^{i} (rX) , \quad \text{and} \quad
R_{(l)}(X) = r^2 R_{(l-1)} (rX).
\end{equation*}
Furthermore, $|\tilde{f}^l(X)| \leq C\varepsilon s[X,O]^{k+\alpha} $ for all $X \in \overline{Q_{1/2}^+}$ and hence we decompose $u^l$ into the sum of $v^l$ and $w^l$ such that
\begin{equation*}
\left\{\begin{aligned}
	v^{l}_t &= L_p^l v^{l} && \mbox{in } Q_{1/2}^+ \\
	v^{l} &= u^{l} && \mbox{on }\partial_p Q_{1/2}^+,
\end{aligned} \right.
\qquad \qquad 
\left\{\begin{aligned}
	w_t^{l} &= L _p^l w^{l} + \tilde{f}^{l}  && \text{in } Q_{1/2}^+ \\
	w^{l} & = 0 && \text{on }\partial_p Q_{1/2}^+,
\end{aligned}\right.
\end{equation*}
and  $|u^{l}(X)-v^{l}(X)|  \leq  r^{k+2 + \alpha} $ for all $X \in \overline{Q^+_r}$. By substituting back, we have
\begin{equation} \label{sbt_back_est_pc}
	\left|u(X) - h(X)  -  \sum_{i=1}^{l} r^{(i-1)(k+2+\alpha)}  v^i (r^{-i+1}X) \right| \leq r^{l(k+2 + \alpha)} \quad \text{for all } X  \in \overline{Q_{r^{l}}^+}.
\end{equation}

We choose $M, N \in \mathbb{N}$ such that
\begin{equation*}
k+2+\alpha < N + \frac{2M}{2-\gamma} \eqqcolon \min \Big\{i + \frac{2j}{2-\gamma} \in(k+2+\alpha,\infty): i,j \in \mathbb{N} \Big\}.
\end{equation*}
By \Cref{lem:pcf_hi_zf} and \eqref{appox_spol_pc}, for each $l \in \mathbb{N}$, there exists an $s$-polynomial $p^l(X)$ of the form
\begin{equation*}
	p^l(X) = \sum_{|\beta| + \frac{2i_1}{2-\gamma} + i_2 + 2j  <N+\frac{2M}{2-\gamma}} A_l^{\beta i j k} {x'}^{\beta} x_n^{i_1+\frac{2 - \gamma}{2} i_2} t^j 
\end{equation*}
such that
\begin{equation} \label{est_vl_pl_pc}
	|v^l(X)-p^l(X)| \leq Cs[X,O]^{N+\frac{2M}{2-\gamma}} \quad \text{for all } X \in \overline{Q^+_1},  
\end{equation}
where $C$ is a positive constant depending only on $n$, $\lambda$, $\Lambda$, $\gamma$, $M$, $N$, $\|a^{ij}\|_{C^{k,\alpha}_s(\overline{Q^+_1})}$, $\|b^i\|_{C^{k,\alpha}_s(\overline{Q^+_1})}$, and $\|c\|_{C^{k,\alpha}_s(\overline{Q^+_1})}$. 
As in the proof of \Cref{lem:ccf_hi_nzf}, we can see
\begin{equation*}
	\left|u(X) - h(X)  -  \sum_{i=1}^l r^{(i-1)(k+2+\alpha)}  p^i (r^{-i+1}X) \right| \leq C r^{l(k+2+\alpha)}
\end{equation*}
for all $X  \in \overline{Q_{r^l}^+}$. Now put 
\begin{equation*}
	P^l(X)= \sum_{i=1}^l r^{(i-1)(k+2+\alpha)} p^i (r^{-i+1}X)  .
\end{equation*}
Then, \eqref{sbt_back_est_pc} gives
\begin{align}
	|P^l(X)-P^{l-1}(X)| 
	&\leq  C r^{l(k+2 + \alpha)}  \label{est_p-p_pc}
\end{align}
for all $X  \in \overline{Q_{r^l}^+}$. Since $ k+2+\alpha \notin \mathcal{D}$, each $p^l(X)$ has degree $m$ corresponding to $\kappa=k+2+\alpha$, so we can consider the rescaling $s$-polynomial
\begin{align*}
	\tilde{P} (X) &\coloneqq P^l(r^l X) - P^{l-1}(r^lX)  \\
	&=r^{(l-1)(k+2+\alpha)}  \sum_{|\beta|+\frac{2 i_1}{2-\gamma}+i_2+2j  < k+2+\alpha} r^{|\beta|+\frac{2 i_1}{2-\gamma}+i_2+2j } A_l^{\beta i_1 i_2 j} x'^{\beta} x_n^{i_1 + \frac{2-\gamma}{2} i_2} t^{j} .
\end{align*}
Finally, as in the proof of \Cref{lem:pcf_hi_zf}, if $h$ is approximated with an $s$-polynomial, we can show that there exists an $s$-polynomial $p$ of degree $m$ corresponding to $\kappa =k+2+\alpha$ at $O$ such that
$$|u(X)-p(X)| \leq Cs[X,O]^{k+2+\alpha} \quad \mbox{for all } X \in \overline{Q_{1}^+}.$$
\end{proof}

\section*{Acknowledgments}
Ki-Ahm Lee is supported by the National Research Foundation of Korea (NRF) grant: NRF- 2021R1A4A1027378. Ki-Ahm Lee also holds a joint appointment with the Research Institute of Mathematics of Seoul National University.

\bibliography{begin-ref}
\bibliographystyle{alpha}


\end{document}